\documentclass{amsart}
\usepackage{amsmath,amssymb,latexsym,hyperref,url,graphicx}
\usepackage[usenames,dvipsnames]{color}

\newtheorem{theorem}{Theorem}
\newtheorem{conjecture}[theorem]{Conjecture}
\newtheorem{corollary}[theorem]{Corollary}
\newtheorem{lemma}[theorem]{Lemma}
\newtheorem{proposition}[theorem]{Proposition}
\newtheorem*{theorem4/3}{Theorem~\ref{4/3}}
\newtheorem*{theorem6/5}{Theorem~\ref{6/5}}

\theoremstyle{definition}
\newtheorem*{definition*}{Definition}
\newtheorem*{example*}{Example}
\newtheorem*{notation*}{Notation}
\newtheorem*{question*}{Open question}

\newcommand{\Q}{\mathbb{Q}}
\newcommand{\Z}{\mathbb{Z}}
\newcommand{\ab}{\frac{a}{b}}
\newcommand{\pq}{\frac{p}{q}}
\newcommand{\Imin}{I_\textnormal{min}}
\newcommand{\Imax}{I_\textnormal{max}}
\newcommand{\mmax}{m_\textnormal{max}}
\newcommand{\word}{\textnormal{\textbf{w}}}
\newcommand{\colonequal}{\mathrel{\mathop:}=}

\newcommand{\gray}[1]{\textcolor{Gray}{#1}}

\begin{document}

\title{Avoiding fractional powers over the natural numbers}
\author{Lara Pudwell}
\address{
	Department of Mathematics and Statistics\\
	Valparaiso University\\
	Valparaiso, Indiana 46383 \\
	USA
}
\author{Eric Rowland}
\thanks{The second-named author was supported in part by a Marie Curie Actions COFUND fellowship at the University of Li\`ege.}
\address{
	Department of Mathematics \\
	Hofstra University \\
	Hempstead, NY 11549 \\
	USA
}
\date{April 8, 2018}

\begin{abstract}
We study the lexicographically least infinite $a/b$-power-free word on the alphabet of non-negative integers.
Frequently this word is a fixed point of a uniform morphism, or closely related to one.
For example, the lexicographically least $7/4$-power-free word is a fixed point of a $50847$-uniform morphism.
We identify the structure of the lexicographically least $a/b$-power-free word for three infinite families of rationals $a/b$ as well many ``sporadic'' rationals that do not seem to belong to general families.
To accomplish this, we develop an automated procedure for proving $a/b$-power-freeness for morphisms of a certain form, both for explicit and symbolic rational numbers $a/b$.
Finally, we establish a connection to words on a finite alphabet.
Namely, the lexicographically least $27/23$-power-free word is in fact a word on the finite alphabet $\{0, 1, 2\}$, and its sequence of letters is $353$-automatic.
\end{abstract}

\maketitle

%%%%%%%%%%%%%%%%%%%%%%%%%%%%%%%%%%%%%%%%%%
\section{Introduction}\label{Introduction}
%%%%%%%%%%%%%%%%%%%%%%%%%%%%%%%%%%%%%%%%%%

A major thread of the combinatorics on words literature is concerned with avoidability of patterns.
Beginning with work of Thue~\cite{Thue 1906, Thue 1912, Berstel}, a basic question has been the following.
Given a pattern, on what size alphabet does there exist an infinite word containing no factors matching the pattern?
For example, it is easy to see that squares (words of the form $ww$ where $w$ is a nonempty word) are unavoidable on a binary alphabet, but Thue~\cite{Thue 1906} exhibited an infinite square-free word on a ternary alphabet.

If it is not known whether a given pattern is avoidable on a given alphabet, it is natural to attempt to construct long finite words that avoid the pattern as follows.
Choose an order on the alphabet.
Begin with the empty word, and then iteratively lengthen the current word by appending the least letter of the alphabet, or, if that letter introduces an instance of the pattern, the next least letter, etc.
If no letter extends the word, then backtrack to the previous letter and increment it instead.
If there exists an infinite word avoiding the pattern, then this procedure eventually computes prefixes of the lexicographically least infinite word avoiding that pattern.

Lexicographically least words avoiding patterns have become a subject of study in their own right.
An \emph{overlap} is a word of the form $c x c x c$ where $c$ is a letter.
On a binary alphabet, the lexicographically least overlap-free word is $001001\varphi^\infty(1)$, where $\varphi(0) = 01, \varphi(1) = 10$ and $\varphi^\infty(1)$ is the complement of the Thue--Morse word~\cite{Allouche--Currie--Shallit}.

Guay-Paquet and Shallit~\cite{Guay-Paquet--Shallit} began the study of lexicographically least words avoiding patterns on the alphabet~$\Z_{\geq 0}$.
They gave morphisms generating the lexicographically least words on $\Z_{\geq 0}$ avoiding overlaps and avoiding squares.
Since the alphabet is infinite, prefixes of such words can be computed without backtracking.

Rowland and Shallit~\cite{Rowland--Shallit} gave a morphism description for the lexicographically least $\frac{3}{2}$-power-free word.
A fractional power is a partial repetition, defined as follows.
Let $a$ and $b$ be relatively prime positive integers.
If $v = v_0 v_1 \cdots v_{l-1}$ is a nonempty word whose length $l$ is divisible by $b$, define
\[
	v^{a/b} \colonequal v^{\lfloor a/b \rfloor} v_0 v_1 \cdots v_{l \cdot \text{FractionalPart}(a/b) - 1},
\]
where $\text{FractionalPart}(\ab) = \ab - \lfloor \ab \rfloor$.
For example, $(0111)^{3/2} = 011101$.
We say that $v^{a/b}$ is an \emph{$\ab$-power}.
Note that $|v^{a/b}| = \ab |v|$.
If $\ab > 1$, then a word $w$ is an $\ab$-power if and only if $w$ can be written $v^e x$ where $e$ is a non-negative integer, $x$ is a prefix of $v$, and $\frac{|w|}{|v|} = \ab$.
We say that a word is \emph{$\ab$-power-free} if none of its factors are $\ab$-powers.
Avoiding $\frac{3}{2}$-powers, for example, means avoiding factors $xyx$ where $|x| = |y| \geq 1$.
Avoiding $\frac{5}{4}$-powers means avoiding factors $xyx$ where $3 |x| = |y| \geq 1$.
More generally, if $1 < \ab < 2$ then an $\ab$-power is a word of the form $xyx$ where $\frac{|xyx|}{|xy|} = \ab$.
A \emph{bordered word} is a word of the form $xyx$ where $x$ is nonempty, so for $1 < \ab < 2$ one can think of an $\ab$-power as a bordered word with a prescribed relationship between $|x|$ and $|y|$.

Basic terminology is as follows.
If $\Sigma$ is an alphabet (finite or infinite), $\Sigma^*$ denotes the set of finite words with letters from~$\Sigma$.
We index letters in a finite or infinite word starting with position~$0$.
A \emph{morphism} on an alphabet $\Sigma$ is a map $\varphi \colon \Sigma \to \Sigma^*$.
A morphism on $\Sigma$ extends naturally to finite and infinite words by concatenation.
A morphism $\varphi$ on $\Sigma$ is \emph{$k$-uniform} if $|\varphi(n)| = k$ for all $n \in \Sigma$.
If there is a letter $c \in \Sigma$ such that $c$ is the first letter of $\varphi(c)$, then iterating $\varphi$ gives a word $\varphi^\infty(c)$ which begins with $c$ and which is a fixed point of~$\varphi$.

In this paper, we show that for some rational numbers $\ab$, the lexicographically least $\ab$-power-free word on $\Z_{\geq 0}$ is a fixed point of a uniform morphism.
For other rationals, this word is the image, under a coding, of a fixed point of a morphism on the alphabet $\Z_{\geq 0} \cup \Sigma$ for some finite set~$\Sigma$.
In both cases, the morphisms ${\varphi |}_{\Z_{\geq 0}}$ are $\ab$-power-free, meaning that if $w$ avoids $\ab$-powers then $\varphi(w)$ avoids $\ab$-powers.
By studying lexicographically least words, we discover many $\ab$-power-free morphisms, which are interesting in their own right.

The outline of the paper is as follows.
In Section~\ref{Examples} we discuss the lexicographically least word avoiding $\ab$-powers for several explicit rationals $\ab$ and discuss the $k$-regularity of their sequences of letters.
In Section~\ref{first intervals} we show that, for an infinite family of rationals in the interval $\frac{5}{3} \leq \ab < 2$, the lexicographically least $\ab$-power-free word is a fixed point of an $\ab$-power-free $(2 a - b)$-uniform morphism.
In Section~\ref{power-free morphisms} we discuss automating proofs of $\ab$-power-freeness and prove similar theorems for other morphisms.
In Section~\ref{Families of words} we establish the structure of the lexicographically least $\ab$-power-free word for two additional infinite families of rationals.
Using the machinery we have built, in Section~\ref{Sporadic words} we address some sporadic words that we have not found to belong to infinite families.

The theorems in Section~\ref{power-free morphisms} and, in part, Sections~\ref{Families of words} and \ref{Sporadic words}, are proved by automated symbolic case analysis.
Since the morphisms are symbolic in $a$ and $b$, a substantial amount of symbolic computation is required.
The \textit{Mathematica} package \textsc{SymbolicWords} was written to manipulate words with symbolic run lengths and perform these computations.
It can be downloaded from the web site of the second-named author\footnote{\url{http://people.hofstra.edu/Eric_Rowland/packages.html\#SymbolicWords} as of this writing.}.

%%%%%%%%%%%%%%%%%%%%%%%%%%%%%%%%%%%%%%%%%%
\section{Some explicit rationals}\label{Examples}
%%%%%%%%%%%%%%%%%%%%%%%%%%%%%%%%%%%%%%%%%%

The following word is our main object of study.

\begin{notation*}
Let $a$ and $b$ be relatively prime positive integers such that $\ab > 1$.
Define $\word_{a/b}$ to be the lexicographically least infinite word on $\Z_{\geq 0}$ avoiding $\ab$-powers.
\end{notation*}

We require $\ab > 1$, because if $0 < \ab \leq 1$ then every word of length $a$ is an $\ab$-power and $\word_{a/b}$ does not exist.
For $\ab > 1$, the word $\word_{a/b}$ exists, since, given a prefix, appending an integer that doesn't occur in the prefix yields a word with no $\ab$-power suffix.
It is clear that $\word_{a/b}$ is not eventually periodic, since the periodic word $xxx\cdots$ contains the $\ab$-power $x^a = (x^b)^{a/b}$.

In this section we examine $\word_{a/b}$ for some explicit rational numbers~$\ab$.
Guay-Paquet and Shallit~\cite{Guay-Paquet--Shallit} showed that the lexicographically least square-free word on $\Z_{\geq 0}$ is
\[
	\word_2 = \varphi^\infty(0) = 01020103010201040102010301020105 \cdots,
\]
where $\varphi$ is the $2$-uniform morphism given by $\varphi(n) = 0 (n+1)$.
More generally, for an integer $a \geq 2$ we have $\word_a = \varphi^\infty(0)$, where $\varphi(n) = 0^{a-1}(n+1)$.
If we write $\word_a = w(0) w(1) \cdots$ so that $w(i)$ is the letter at position $i$ in $\word_a$, then
\[
	w(a i + r) =
	\begin{cases}
		0		& \text{if $0 \leq r \leq a-2$} \\
		w(i) + 1	& \text{if $r = a-1$}.
	\end{cases}
\]

The results in this paper can been seen as generalizations of this morphism and recurrence to fractional powers.
Many of the morphisms that will appear are of the form $\varphi(n) = u \, (n + d)$, where $u$ is a word of length $k - 1$ and $d \in \Z_{\geq 0}$.
If we write $\varphi^\infty(0) = w(0) w(1) \cdots$, then the letter sequence $w(i)_{i \geq 0}$ satisfies
\begin{equation}\label{constant columns recurrence}
	w(k i + r) =
	\begin{cases}
		w(r)		& \text{if $0 \leq r \leq k-2$} \\
		w(i) + d	& \text{if $r = k-1$}
	\end{cases}
\end{equation}
for all $i \geq 0$.

Consider
\[
	\word_{3/2} = 001102 \, 100112 \, 001103 \, 100113 \, 001102 \, 100114 \, 001103 \, 100112 \, \cdots.
\]
The first array in Figure~\ref{3/2 5/3 9/5 arrays} shows the first several letters of $\word_{3/2}$, partitioned into rows of length $6$, with the integers $0$ through $6$ rendered in gray levels from white to black.
The first five columns are periodic, and the last column is a ``self-similar'' column consisting of the letters of $\word_{3/2}$, each increased by~$2$.
The letters of $\word_{3/2}$ satisfy
\[
	w(6 i + r) =
	\begin{cases}
		w(r)		& \text{if $r \in \{0, 2, 4\}$ and $i$ is even} \\
		1 - w(r)	& \text{if $r \in \{0, 2, 4\}$ and $i$ is odd} \\
		w(r)		& \text{if $r \in \{1, 3\}$} \\
		w(i) + 2	& \text{if $r = 5$}
	\end{cases}
\]
for all $i \geq 0$, which follows from the recurrence given by Shallit and the second-named author~\cite{Rowland--Shallit}.
Moreover, $\word_{3/2}$ is the image under a coding of a fixed point of a $6$-uniform morphism as follows.
Consider the alphabet $\Z_{\geq 0} \cup \{0', 1'\}$ extended by new letters $0'$ and~$1'$.
Let $\varphi$ be the morphism on $\Z_{\geq 0} \cup \{0', 1'\}$ defined by
\[
	\varphi(n) =
	\begin{cases}
		0' 0 1' 1 0' 2		& \text{if $n = 0'$} \\
		0' 0 1' 1 0' 3		& \text{if $n = 1'$} \\
		1' 0 0' 1 1' (n + 2)	& \text{if $n \in \Z_{\geq 0}$}.
	\end{cases}
\]
Let $\tau$ be the coding defined by $\tau(0') = 0$, $\tau(1') = 1$, and $\tau(n) = n$ for $n \in \Z_{\geq 0}$; this morphism $\tau$ will be the same throughout the paper.
Then $\word_{3/2} = \tau(\varphi^\infty(0'))$.

\begin{figure}
	\includegraphics[scale=.5]{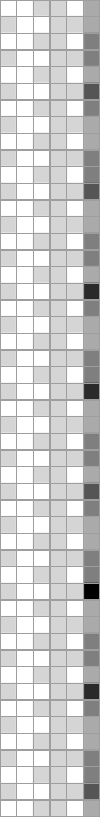} \hspace{2cm}
	\includegraphics[scale=.5]{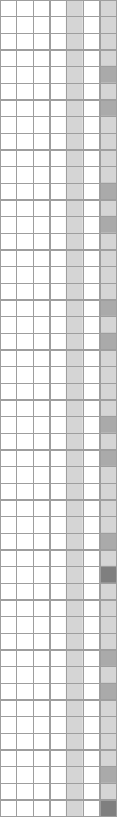} \hspace{2cm}
	\includegraphics[scale=.5]{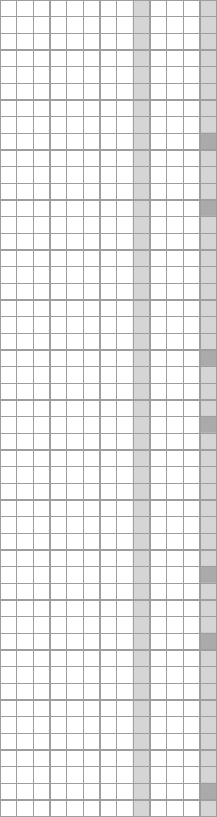}
	\caption{Prefixes of $\word_{3/2}$ (left), $\word_{5/3}$ (center), and $\word_{9/5}$ (right), partitioned into rows of width~$k$.}
	\label{3/2 5/3 9/5 arrays}
\end{figure}

The integer $6$ features prominently in the structural description of $\word_{3/2}$, and the sequence of letters of $\word_{3/2}$ is a $6$-regular sequence in the sense of Allouche and Shallit~\cite{Allouche--Shallit 1992}.
For an integer $k \geq 2$, a sequence $s(i)_{i \geq 0}$ is said to be \emph{$k$-regular} if the $\Z$-module generated by the set of subsequences $\{ s(k^e i + j)_{i \geq 0} : \text{$e \geq 0$ and $0 \leq j \leq k^e - 1$} \}$ is finitely generated.
This implies that $s(i)$ can be computed from the base-$k$ digits of $i$ from some finite set of linear recurrences (such as Equation~\eqref{constant columns recurrence}) and initial conditions.

One of the main motivations of the present paper is to put the `$6$' for $\word_{3/2}$ into context by studying $\word_{a/b}$ for a number of other rationals~$\ab$.
We will see that $\word_{a/b}$ is often $k$-regular for some value of~$k$.
For each integer $a \geq 2$, the word $\word_a$ is $a$-regular.
To demonstrate the variety of values that occur for $b \geq 2$, let us survey $\word_{a/b}$ for some rationals with small numerators and denominators.
We start with some words with fairly simple structure and progress toward more complex words.

For $\ab = \frac{5}{3}$ we obtain the word
\[
	\word_{5/3} = 0000101 \, 0000101 \, 0000101 \, 0000101 \, 0000102 \, 0000101 \, 0000102 \, \cdots.
\]
Partitioning $\word_{5/3}$ into rows of length $7$ produces the second array in Figure~\ref{3/2 5/3 9/5 arrays}.
There are $6$ constant columns and one self-similar column in which the sequence reappears with every term increased by~$1$.
Therefore the sequence of letters seems to satisfy Equation~\eqref{constant columns recurrence} with $k = 7$ and $d = 1$, and in fact we have the following.

\begin{theorem}\label{5/3}
Let $\varphi$ be the $7$-uniform morphism defined by
\[
	\varphi(n) = 000010(n + 1)
\]
for all $n \in \Z_{\geq 0}$.
Then $\word_{5/3} = \varphi^\infty(0)$.
\end{theorem}

The last array in Figure~\ref{3/2 5/3 9/5 arrays} shows the word $\word_{9/5}$ partitioned into rows of length $k = 13$.
Again we see $12$ constant columns and one self-similar column.
Indeed this word is generated by the following morphism.

\begin{theorem}\label{9/5}
Let $\varphi$ be the $13$-uniform morphism defined by
\[
	\varphi(n) = 000000001000(n + 1)
\]
for all $n \in \Z_{\geq 0}$.
Then $\word_{9/5} = \varphi^\infty(0)$.
\end{theorem}

We prove Theorems~\ref{5/3} and \ref{9/5} in Section~\ref{first intervals}.
For $\ab = \frac{8}{5}$ the value of $k$ is somewhat larger.

\begin{theorem}\label{8/5}
Let $\varphi$ be the $733$-uniform morphism defined by
\small
\begin{align*}
	\varphi(n) = {}
	& 0000000100100000100100000001001100000001001000001001000000010020000 \\
	& 0100100100000001001000001001000001001000000010010010000000100100000 \\
	& 1001000001001000000010010010000000100100000100100000100100000001001 \\
	& 0010000000100100000100100000100100000001001001000000010010000010010 \\
	& 0000100100000001001001000000010010000010010000010010000000100100100 \\
	& 0000010010000010010000010010000000100100100000001001000001001000001 \\
	& 0010110000000100100000100100000001002000001001001000000010010000010 \\
	& 0100000100100000001001001000000010010000010010000010010000000100100 \\
	& 1000000010010000010010000010010000000100100100000001001000001001000 \\
	& 0010010001000100010001000100010001101000000010010000010010000000101 \\
	& 00010001000100010001000100010100000001001000001001000000010100(n + 2)
\end{align*}
\normalsize
for all $n \in \Z_{\geq 0}$.
Then $\word_{8/5} = \varphi^\infty(0)$.
\end{theorem}

Among rationals with denominator $4$ we come across an even longer morphism.

\begin{theorem}\label{7/4}
There is a $50847$-uniform morphism
\[
	\varphi(n) = 000000100100000010010000001001000011\cdots1000000100100000010(n + 2)
\]
such that $\word_{7/4} = \varphi^\infty(0)$.
\end{theorem}

We do not have any intuition to explain the large value $k = 50847$.
Indeed, to even guess this value one must compute several hundred thousand letters of $\word_{7/4}$.

Let us look at a couple more rationals with denominator~$5$.
For $\frac{6}{5}$ the correct value is $k = 1001$.
However, the array obtained by partitioning $\word_{6/5}$ into rows of length $1001$, shown in Figure~\ref{6/5 array}, does not have constant columns but columns that become constant after $30$ rows.
Subsequent rows suggest a certain $1001$-uniform morphism $\varphi$, but we must build in the earlier rows by defining $\varphi(0') = v \, \varphi(0)$ for some word $v$ whose first letter is~$0'$.
Note that $\varphi$ is no longer uniform.
We refer to the prefix $v$ as the \emph{transient}.
Let $\tau$ be as before; $\tau(0') = 0$ and $\tau(n) = n$ for $n \in \Z_{\geq 0}$.
Then we have the following.

\begin{figure}
	\includegraphics[width=\textwidth]{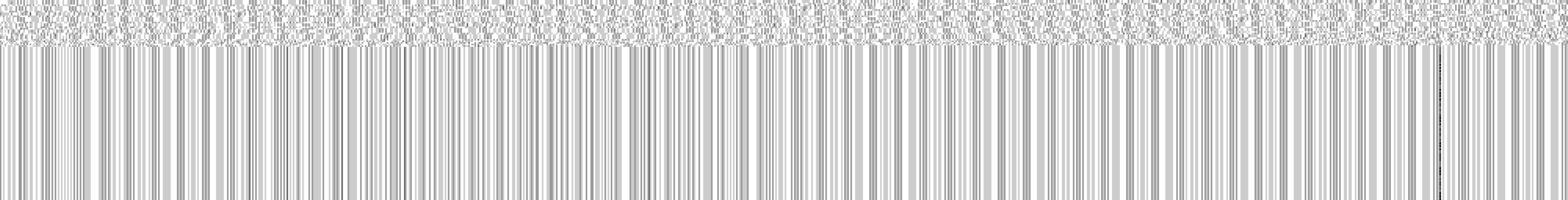}
	\caption{A prefix of $\word_{6/5}$, partitioned into rows of width $k = 1001$.}
	\label{6/5 array}
\end{figure}

\begin{theorem}\label{6/5}
There exist words $u, v$ of lengths $|u| = 1001 - 1$ and $|v| = 29949$ such that $\word_{6/5} = \tau(\varphi^\infty(0'))$, where
\[
	\varphi(n) =
	\begin{cases}
		v \, \varphi(0)		& \text{if $n = 0'$} \\
		u \, (n + 3)			& \text{if $n \in \Z_{\geq 0}$}.
	\end{cases}
\]
\end{theorem}

Although we state Theorem~\ref{6/5} as an existence result, the words $u$ and $v$ can be obtained explicitly by computing the appropriate prefix of $\word_{6/5}$.

The sequence of letters in $\word_{6/5}$ does not satisfy Equation~\eqref{constant columns recurrence} but does satisfy a modified equation accounting for the transient.
Write $\word_{6/5} = w(0) w(1) \cdots$.
Then for all $i \geq 0$ we have
\[
	w(1001 i + 29949 + r) =
	\begin{cases}
		w(29949 + r)	& \text{if $0 \leq r \leq 999$} \\
		w(i) + 3		& \text{if $r = 1000$}.
	\end{cases}
\]
That is, the letters of $\word_{6/5}$ reappear as a subsequence, with every term increased by $3$, beginning at $w(30949) = w(0) + 3$.

For $\frac{7}{5}$ there seems to be a similar transient, with $k = 80874$.

\begin{conjecture}\label{7/5}
There exist words $u, v$ of lengths $|u| = 80874 - 1$ and $|v| = 93105$ such that $\word_{7/5} = \tau(\varphi^\infty(0'))$, where
\[
	\varphi(n) =
	\begin{cases}
		v \, \varphi(0)		& \text{if $n = 0'$} \\
		u \, (n + 1)			& \text{if $n \in \Z_{\geq 0}$}.
	\end{cases}
\]
\end{conjecture}

For the word $\word_{4/3}$, partitioning into rows of length $k = 56$ gives $k - 1$ eventually periodic columns.
Unlike the previous examples, the self-similar column for $\word_{4/3}$ does not contain the sequence simply transformed by adding a constant $d$; each $0$ is increased by $1$ and other integers are increased by~$2$.

\begin{theorem}\label{4/3}
There exist words $u, v$ of lengths $|u| = 56 - 1$ and $|v| = 18$ such that $\word_{4/3} = \tau(\varphi^\infty(0'))$, where
\[
	\varphi(n) =
	\begin{cases}
		v \, \varphi(0)		& \text{if $n = 0'$} \\
		u \, 1				& \text{if $n = 0$} \\
		u \, (n + 2)			& \text{if $n \in \Z_{\geq 1}$}.
	\end{cases}
\]
\end{theorem}

We prove Theorems~\ref{8/5}--\ref{6/5} and Theorem~\ref{4/3} in Section~\ref{Sporadic words}.
In principle, Conjecture~\ref{7/5} can be proved in the same manner, although the computation would take longer than we choose to wait.

Finally, we mention a word whose structure we do not know.
We have computed the prefix of $\word_{5/4}$ of length $400000$.
When partitioned into rows of length $12$ (or $6$ or $24$), all but one column appears to be eventually periodic, but we have not been able to identify self-similar structure in this last column.

With the possible exception of $\word_{5/4}$, the structure of each of these words is organized around some integer~$k$.
This integer is the length of $\varphi(n)$ for each $n \in \Z_{\geq 0}$.
We next show that the sequence of letters in such a word forms a $k$-regular sequence.
For the fixed point $\varphi^\infty(0)$ of a $k$-uniform morphism $\varphi(n) = u \, (n + d)$, the $k$-regularity follows directly from Equation~\eqref{constant columns recurrence}.

\begin{theorem}\label{regular sequences}
Let $k \geq 2$ and $d \geq 0$.
Let $u$ be a word on $\Z_{\geq 0}$ of length $k - 1$.
Let $v$ be a nonempty finite word on $\Z_{\geq 0} \cup \{0'\}$ whose first letter is $0'$ and whose remaining letters are integers.
Let
\[
	\varphi(n) =
	\begin{cases}
		v \, \varphi(0)		& \text{if $n = 0'$} \\
		u \, (n + d)			& \text{if $n \in \Z_{\geq 0}$}.
	\end{cases}
\]
Then the sequence of letters in $\tau(\varphi^\infty(0'))$ is a $k$-regular sequence.
\end{theorem}

\begin{proof}
Let $w(i)$ be the letter at position $i$ in $\tau(\varphi^\infty(0'))$.
Let $u(i)$ be the letter at position $i$ in $u$.
We have
\begin{equation}\label{transient recurrence}
	w(k i + |v| + r) =
	\begin{cases}
		u(r)		& \text{if $0 \leq r \leq k - 2$} \\
		w(i) + d	& \text{if $r = k - 1$}
	\end{cases}
\end{equation}
for all $i \geq 0$.

The set of sequences $\{ w(k^e i + j)_{i \geq 0} : \text{$e \geq 0$ and $0 \leq j \leq k^e - 1$} \}$ is called the \emph{$k$-kernel} of $w(i)_{i \geq 0}$.
We show that the $\Z$-module generated by the $k$-kernel of $w(i)_{i \geq 0}$ is finitely generated, and hence $w(i)_{i \geq 0}$ is $k$-regular.

First we consider sequences in the $k$-kernel of the form $w(k^e i + j)_{i \geq 0}$ where
\[
	j = -k^e q + \tfrac{k^e - 1}{k - 1} \left(|v| + k - 1\right)
\]
for some integer $q$.
Since $0 \leq j < k^e$, we have $Q - 1 < q \leq Q$ (and therefore $q = \lfloor Q \rfloor$), where
\[
	Q = \tfrac{k^e - 1}{k^e (k - 1)} \left(|v| + k - 1\right).
\]
Since $Q$ approaches the finite limit $\frac{|v| + k - 1}{k - 1}$ as $e \to \infty$, for sufficiently large $e$ all the $q$ are the same.
Moreover, $e$ applications of Equation~\eqref{transient recurrence} show that $w(k^e i + j) = w(i - q) + e d$ for all $i \geq q$.
This isn't sufficient, because we would like to show that the terms $w(k^e i + j)$ are related for all $i \geq 0$.
However, $N$ applications of Equation~\eqref{transient recurrence} show that
\[
	w(k^e i + j) = w\!\left(k^{e-N} (i - q) + \tfrac{k^{e-N} - 1}{k - 1} (|v| + k - 1)\right) + N d,
\]
provided the argument of $w$ on the right side is non-negative.
To ensure this is the case, we solve for $N$ and find that
\[
	N = e + \left\lfloor \log_k \left(1 - q \cdot \tfrac{k - 1}{|v| + k - 1}\right) \right\rfloor
\]
suffices.
(Note that the argument of $\log_k$ is less than $1$, so $N < e$.)
We conclude that $w(k^e i + j) - N d$ is independent of $e$.
Therefore, up to addition by multiples of $d$, there are only finitely many distinct sequences of the form $w(k^e i + j)_{i \geq 0}$, and the $\Z$-module they generate is finitely generated.

Every sequence in the $k$-kernel that is not of the form discussed in the previous paragraph is eventually constant, since iteratively applying Equation~\eqref{transient recurrence} shows that it is some multiple of $d$ plus a subsequence of one of $k - 1$ eventually constant sequences.
Aside from the multiple of $d$, there are only finitely many distinct sequences of this form.
So these sequences also generate a finitely generated $\Z$-module.
\end{proof}

The value of $k$ for which a given sequence is $k$-regular is not unique; for each $\alpha \geq 1$, a sequence is $k$-regular if and only if it is $k^\alpha$-regular~\cite[Theorem~2.9]{Allouche--Shallit 1992}.
Define an equivalence relation $\sim$ on $\Z_{\geq 2}$ in which $k \sim l$ if there exist positive integers $s$ and $t$ such that $k^s = l^t$.
If $k \sim l$, then $k$ and $l$ are said to be \emph{multiplicatively dependent}.
Corollary~\ref{unique k} uses the following bound on the growth rate of letters in $\word_{a/b}$ and a result of Bell~\cite{Bell} to show that if $\word_{a/b}$ is $k$-regular then $k$ is unique up to multiplicative dependence.

\begin{theorem}\label{growth}
Let $a, b$ be relatively prime positive integers such that $\ab > 1$.
Let $w(i)$ be the letter at position $i$ in $\word_{a/b}$.
Then $w(i) = O(\sqrt{i})$.
\end{theorem}

\begin{proof}
If $\word_{a/b}$ is a word on a finite alphabet, then the conclusion clearly holds, so assume that every $n \geq 0$ occurs in $\word_{a/b}$.
Let $i_n$ be the position of the first occurrence of $n$ in $\word_{a/b}$.

We say that a word $v$ is a \emph{pre-$\ab$-power} if $|v| = ma$ for some integer $m \geq 1$ and if there exists an integer $c$ such that the word obtained by changing the last letter of $v$ to $c$ is an $\ab$-power.

Fix $n \geq 1$.
For each $c$ with $0 \leq c \leq n-1$, let $m_c a$ be the length of a pre-$\ab$-power ending at position $i_{n}$ in $\word_{a/b}$ that prevents the letter $c$ from occurring at position $i_n$.
Since an $\ab$-power of length $m_c a$ consists of repeated factors of length $\frac{b}{a} (m_c a) = m_c b$, the letter $c$ occurs at position $i_n - m_c b$.
It follows that $m_0, \dots, m_{n-1}$ are distinct.

Let $M = \max(m_0, \dots, m_{n-1})$.
We claim that $i_n - i_{n-1} \geq M$.
We consider two cases.

First suppose $m_{n-1}b \geq M$.
We have already established that $n-1$ occurs at position $i_n - m_{n-1} b$.
Since $i_{n-1}$ is the position of the first occurrence of $n-1$, this implies that $i_{n-1} \leq i_n - m_{n-1} b$, so $i_{n}-i_{n-1} \geq m_{n-1} b \geq M$ as claimed.

Otherwise, $m_{n-1}b < M$.
Since $M=\max(m_0, \dots, m_{n-1})$, we have $M =m_c$ for some $c$.
Therefore $m_{n-1}b < m_c$ for this $c$, and $c \neq n-1$.
By definition of $m_c$, we have that $w(i_n-j) = w(i_n-m_c b -j)$ for $1 \leq j < m_c$.
(In fact, if $b >1$ the statement holds for $1 \leq j < m_c \cdot (a \bmod b)$.)
Since $n-1$ occurs at position $i_n - m_{n-1} b$, for $j = m_{n-1}b$ this shows that $n - 1 = w(i_n - m_{n-1}b) = w(i_n - m_c b - m_{n-1}b)$.
This implies that $i_{n-1} \leq i_n - m_c b -m_{n-1}b$, so $i_n - i_{n-1} \geq m_c b + m_{n-1} b > m_c b \geq m_c =M$ as claimed.

Since $m_0, \dots, m_{n-1}$ are distinct positive integers, it follows that $i_{n} - i_{n-1} \geq M \geq n$.
It follows that $i_n$ grows at least like $n^2$, so $w(i) = O(\sqrt{i})$.
\end{proof}

\begin{corollary}\label{unique k}
Let $a, b$ be relatively prime positive integers such that $\ab > 1$.
The values of $k$ for which $\word_{a/b}$ is $k$-regular are equivalent modulo $\sim$.
\end{corollary}

\begin{proof}
Write $\word_{a/b} = w(0) w(1) \cdots$.
Bell's generalization~\cite{Bell} of Cobham's theorem implies that if $w(i)_{i \geq 0}$ is both $k$-regular and $l$-regular, where $k \geq 2$ and $l \geq 2$ are multiplicatively independent, then $\sum_{i \geq 0} w(i) x^i$ is the power series of a rational function whose poles are roots of unity.
Therefore it suffices to show that $\sum_{i \geq 0} w(i) x^i$ is not such a power series.

The coefficients of a rational power series whose poles are roots of unity are given by an eventual quasi-polynomial, that is, $d_e(i) i^e + \dots + d_1(i) i + d_0(i)$ where each $d_j(i)_{i \geq 0}$ is eventually periodic.
The sequence $w(i)_{i \geq 0}$ is not given by an eventual quasi-polynomial of degree $e = 0$ since it is not eventually periodic.
Theorem~\ref{growth} rules out degree each $e \geq 1$.
\end{proof}

Corollary~\ref{unique k} suggests a partial function
\[
	\rho \colon \Q_{> 1} \to \Z_{\geq 2}/{\sim}
\]
where $\rho(\ab)$ is the class of integers $k$ such that $\word_{a/b}$ is $k$-regular, if this class exists.
If this class does not exist, we leave $\rho(\ab)$ undefined.
When $\rho(\ab)$ is defined, we will abuse notation and write $\rho(\ab) = k$, choosing one integer $k$ from the class.
For example, $\rho(a) = a$ for each integer $a \geq 2$, and $\rho(\frac{3}{2}) = 6$.
Theorems~\ref{5/3}--\ref{6/5} imply the values $\rho(\frac{5}{3}) = 7$, $\rho(\frac{9}{5}) = 13$, $\rho(\frac{8}{5}) = 733$, $\rho(\frac{7}{4}) = 50847$, and $\rho(\frac{6}{5}) = 1001$.
Conjecture~\ref{7/5} would imply $\rho(\frac{7}{5}) = 80874$.
The proof of Theorem~\ref{regular sequences} applies to the morphism in Theorem~\ref{4/3}, since $w(i)_{i \geq 0}$ is the only sequence in the $56$-kernel that contains a $0$ and is not eventually constant, so for the remaining sequences we can take $d = 2$ to be constant; therefore $\rho(\frac{4}{3}) = 56$.
We do not know if $\rho(\frac{5}{4})$ is defined.

\begin{question*}
For each rational number $\ab > 1$, does there exist $k$ such that $\word_{a/b}$ is $k$-regular?
In other words, is $\rho(\ab)$ defined?
\end{question*}

We end this section by noting that there are two other natural notions of avoidance for fractional powers.
For each notion we may define the lexicographically least word avoiding all words in the corresponding pattern.

\begin{notation*}
Let $a$ and $b$ be relatively prime positive integers such that $\ab > 1$.
\begin{itemize}
\item
$\word_{\geq a/b}$ is the lex.\ least infinite word on $\Z_{\geq 0}$ avoiding $\pq$-powers for all $\pq \geq \ab$.
\item
$\word_{> a/b}$ is the lex.\ least infinite word on $\Z_{\geq 0}$ avoiding $\pq$-powers for all $\pq > \ab$.
\end{itemize}
\end{notation*}

For an integer $a \geq 2$, the structure of $\word_{\geq a}$ is as follows.

\begin{proposition}
Let $a \geq 2$ be an integer.
Then $\word_{\geq a} = \word_a$.
\end{proposition}

\begin{proof}
Since the language of $a$-powers is a subset of the language of \gray{(}$\geq a$\gray{)}-powers, we have $\word_{\geq a} \geq \word_a$ lexicographically.
Conversely, every $\pq$-power contains a $\lfloor \pq \rfloor$-power, so $\word_a$ does not contain any $\pq$-power with $\pq \geq a$; since $\word_{\geq a}$ is the lexicographically least such word, it follows that $\word_a \geq \word_{\geq a}$ lexicographically.
Therefore $\word_{\geq a} = \word_a$.
\end{proof}

If $b \geq 2$, it is possible to avoid $\ab$-powers while containing a $\pq$-power with $\pq > \ab$.
For example, $001102$ avoids $\frac{3}{2}$-powers but contains squares.
For this reason, $\word_{\geq a/b}$ and $\word_{a/b}$ are not equal in general.

\begin{proposition}\label{inequality of words}
Let $a, b$ be relatively prime positive integers such that $\ab > 1$ and $b \geq 2$.
Then $\word_{\geq a/b} \neq \word_{a/b}$.
\end{proposition}

\begin{proof}
To show that $\word_{\geq a/b}$ and $\word_{a/b}$ are unequal, we compare their prefixes.
We start with $\word_{a/b}$.
The word $0^{a-1}$ is $\ab$-power-free, because the shortest $\ab$-powers have length $a$.
The word $0^a = (0^b)^{a/b}$ is an $\ab$-power, so the length-$a$ prefix of $\word_{a/b}$ is $0^{a-1} 1$.
The word $\word_{\geq a/b}$ begins with $0^{\lceil a/b \rceil - 1}$, since every $\pq$-power with $\pq \geq \ab$ has length at least $p \geq \frac{a q}{b} \geq \ab$.
However, $0^{\lceil a/b \rceil}$ is a $\lceil \ab \rceil$-power, so $\word_{\geq a/b}$ begins with $0^{\lceil a/b \rceil - 1} 1$.
Since $\lceil \ab \rceil < a$, the two words are not equal.
\end{proof}

Even though $\word_{\geq a/b} \neq \word_{a/b}$ for $b \geq 2$, these two words are nonetheless closely related for certain rationals.
Consider
\[
	\word_{\geq 3/2} = 01203 \, 10213 \, 01204 \, 10214 \, 01203 \, 10215 \, 01204 \, 10213 \, \cdots.
\]
The words $\word_{\geq 3/2}$ and $\word_{3/2}$ are generated by the same underlying $6$-uniform morphism~\cite{Rowland--Shallit}.
In particular, $\word_{\geq 3/2}(5 i + 4) = \word_{3/2}(i) + 3$ for all $i \geq 0$.
The words $\word_{\geq 4/3}$ and $\word_{4/3}$ also appear to have the essentially the same self-similar column.
It would be interesting to know whether similar statements hold for other rationals.

\begin{conjecture}\label{4/3 words}
For all $i \geq 0$ we have
\[
	\word_{\geq 4/3}(336 i + 1666) = \word_{4/3}(56 i + 17) + 4.
\]
\end{conjecture}

On the other hand, $\word_{> a/b}$ and $\word_{a/b}$ need not be related, even for $b = 1$.
Guay-Paquet and Shallit~\cite{Guay-Paquet--Shallit} studied the overlap-free word
\[
	\word_{> 2} = 001001100100200100110010021001002001001100\cdots
\]
and exhibited a (non-uniform) morphism $\varphi$ such that $\word_{> 2} = \varphi^\infty(0)$.
The words $\word_{> 2}$ and $\word_2$ appear to be unrelated.
It seems likely that the structure of $\word_{> a/b}$ is typically more difficult to determine than $\word_{a/b}$.

%%%%%%%%%%%%%%%%%%%%%%%%%%%%%%%%%%%%%%%%%%
\section{The intervals $\ab \geq 2$ and $\frac{5}{3} \leq \ab < 2$}\label{first intervals}
%%%%%%%%%%%%%%%%%%%%%%%%%%%%%%%%%%%%%%%%%%

It turns out that for $\ab \geq 2$ the lexicographically least $\ab$-power-free word is a word we have already seen.
For example, one computes
\[
	\word_{5/2} = 00001 \, 00001 \, 00001 \, 00001 \, 00002 \, 00001 \, 00001 \, 00001 \, 00001 \, 00002 \, \cdots
\]
and observes that $\word_{5/2}$ agrees with $\word_5$ on a long prefix.
In fact these two words are the same.

\begin{theorem}\label{large rationals}
Let $a, b$ be relatively prime positive integers such that $\ab \geq 2$.
Then $\word_{a/b} = \word_a$.
In particular, $\rho(\ab) = a$.
\end{theorem}

\begin{proof}
We show that $\word_{a/b}$ is $a$-power-free, which implies that $\word_a \leq \word_{a/b}$ lexicographically, and that $\word_a$ is $\ab$-power-free, which implies that $\word_{a/b} \leq \word_a$ lexicographically.

Since the $a$-power $v^a$ is an $\ab$-power $(v^b)^{a/b}$ and $\word_{a/b}$ is $\ab$-power-free, it follows immediately that $\word_{a/b}$ is $a$-power-free.

Suppose toward a contradiction that $\word_a$ contains an $\ab$-power.
Let $v^{a/b}$ (where $|v|$ is divisible by $b$) be an $\ab$-power in $\word_a$ of minimal length.
Since  $\ab \geq 2$, $v^2$ occurs in $\word_a$.
We have $|v^{a/b}| \geq a$, since $a$ and $b$ are relatively prime.
On the other hand, the longest zero factor of $\word_a$ is $0^{a-1}$, so $v$ contains at least one nonzero letter.
The nonzero letters of $\word_a = w(0) w(1) \cdots$ are $w(i)$ for $i \equiv a-1 \mod a$.
Since nonzero letters occur spaced by multiples of $a$ and $v^2$ occurs in $\word_a$, it follows that $|v|$ is divisible by $a$.
Let $x$ be the word obtained by deleting the $0$ letters of $v$.
Then the word obtained by deleting the $0$ letters of $v^{a/b}$ is $x^{a/b}$, which is also the word obtained by sampling every $a$ letters of $v^{a/b}$ starting from the first nonzero letter.
By Equation~\eqref{constant columns recurrence} with $d = 1$, the word obtained by subtracting $1$ from each letter in $x^{a/b}$ occurs in $\word_a$, and since $|x| = |v| / a$ this contradicts the minimality of $v$.
\end{proof}

In light of Theorem~\ref{large rationals}, the remainder of the paper is concerned with $\word_{a/b}$ for $1 < \ab < 2$.
Let $\Q_{(1, 2)} \colonequal \Q \cap (1, 2)$ denote the set of rational numbers in this interval.
The following result relates the prefixes of $\word_{a/b}$ to each other.

\begin{proposition}
The order on $\Q_{(1, 2)} \cup \Z_{\geq 2}$ induced by the lexicographic order on $\{\word_{a/b} : \ab \in \Q_{(1, 2)} \cup \Z_{\geq 2}\}$ is
\[
	2 \succ \frac{3}{2} \succ 3 \succ \frac{4}{3} \succ 4 \succ \frac{5}{4} \succ \frac{5}{3} \succ 5 \succ \frac{6}{5} \succ 6 \succ \frac{7}{6} \succ \frac{7}{5} \succ \frac{7}{4} \succ 7 \succ \frac{8}{7} \succ \frac{8}{5} \succ 8 \succ \cdots.
\]
\end{proposition}

\begin{proof}
Let $a$ and $b$ be relatively prime positive integers such that $\ab \in \Q_{(1, 2)} \cup \Z_{\geq 2}$.
As in the proof of Proposition~\ref{inequality of words}, the length-$a$ prefix of $\word_{a/b}$ is $0^{a-1} 1$.
Therefore the rationals in $\Q_{(1, 2)} \cup \Z_{\geq 2}$ with a given numerator form an interval in the order $\succ$.
If $a = 2$ then $b = 1$ and this prefix is followed by $02$.
If $a \geq 3$ and $b = 1$ then this prefix is followed by $0^{a-1} 1$.
Finally, if $b > 1$ then this prefix is followed by $0^{a - b - 1} 1$, since $0^{a-1} 1 0^{a - b - 1}$ is $\ab$-power-free but $0^{a-1} 1 0^{a - b}$ contains the $\ab$-power $0^{b - 1} 1 0^{a - b} = (0^{b-1} 1)^{a/b}$.
This is sufficient information to distinguish all $\word_{a/b}$ and hence order them.
In particular, $\ab \succ \frac{a}{b'}$ if and only if $b > b'$.
\end{proof}

Recall our claims in Theorems~\ref{5/3} and \ref{9/5} that $\word_{5/3} = \varphi^\infty(0)$ for the $7$-uniform morphism $\varphi(n) = 000010(n+1)$ and $\word_{9/5} = \varphi^\infty(0)$ for the $13$-uniform morphism $\varphi(n) = 000000001000(n+1)$.
These morphisms differ only in their run lengths.
It is not immediately obvious why $7$ is the correct value of $k$ for $\word_{5/3}$ and why $13$ is the correct value for $\word_{9/5}$.
However, these values can be understood in the context of an infinite family of morphisms that generate words $\word_{a/b}$.

\begin{theorem}\label{2 2a-b}
Let $a, b$ be relatively prime positive integers such that $\frac{5}{3} \leq \ab < 2$ and $\gcd(b,2) = 1$.
Let $\varphi$ be the $(2a-b)$-uniform morphism defined by
\[
	\varphi(n) = 0^{a - 1} \, 1 \, 0^{a - b - 1} \, (n + 1)
\]
for all $n \in \Z_{\geq 0}$.
Then $\word_{a/b} = \varphi^\infty(0)$.
In particular, $\rho(\ab) = 2a-b$.
\end{theorem}

We devote the remainder of this section to a proof of Theorem~\ref{2 2a-b}.
We will use the following concept.
A morphism $\varphi$ is \emph{$\ab$-power-free} if it preserves $\ab$-power-freeness; that is, if $w$ is $\ab$-power-free then $\varphi(w)$ is $\ab$-power-free.
For example, the morphism $\varphi(n) = 0 \, (n + 1)$ is square-free~\cite{Guay-Paquet--Shallit}.
Since the word $0$ is $\ab$-power-free, if $\varphi$ is an $\ab$-power-free morphism then $\varphi^\infty(0)$ is also $\ab$-power-free.
If, moreover, $\varphi^\infty(0)$ is the lexicographically least $\ab$-power-free word, then $\word_{a/b} = \varphi^\infty(0)$.

In the interval $1 < \ab < 2$, an $\ab$-power is a word of the form $(xy)^{a/b} = xyx$, where $|xy| = m b$ and $|xyx| = m a$ for some $m \geq 1$.
It follows that $|x| = m \cdot (a-b)$ and $|y| = m \cdot (2b-a)$.
We use this in the following proof and repeatedly throughout the paper.
Note that $\ab < 2$ implies $y$ is not the empty word.

\begin{proof}[Proof of Theorem~\ref{2 2a-b}]
To show that $\varphi$ is $\ab$-power-free, we show that if $w$ is a finite word such that $\varphi(w)$ contains an $\ab$-power, then $w$ contains an $\ab$-power.
Suppose that $v^{a/b}$ is an $\ab$-power factor of $\varphi(w)$, where $|v| = m b$ for some $m \geq 1$.

First consider $m = 1$.
The word $\varphi(w)$ is a finite concatenation of words of the form $0^{a-1} \, 1 \, 0^{a-b-1} \, (n + 1)$ for $n \geq 0$.
We would like to survey all length-$a$ factors of $\varphi(w)$ and verify that they are not $\ab$-powers.
To do this, it suffices to slide a window of length $a$ through the circular word
\[
	0^{a-1} \, 1 \, 0^{a-b-1} \, (n + 1),
\]
since the length of this word is $2 a - b > a$.
There are $|\varphi(n)| = 2 a - b$ length-$a$ factors of $0^{a-1} \, 1 \, 0^{a-b-1} \, (n + 1)$, but we can partition them into the following five forms parameterized by $i$.
\[
\begin{array}{c|l}
	\text{length-$a$ factor}						& \text{interval for $i$} \\ \hline
	0^{a - 1 - i} \, 1 \, 0^i						& 0 \leq i \leq a - b - 1 \\
	0^{b - 1 - i} \, 1 \, 0^{a - b - 1} \, (n + 1) \, 0^i				& 0 \leq i \leq 2 b - a - 1 \\
	0^{a - b - 1 - i} \, 1 \, 0^{a - b - 1} \, (n + 1) \, 0^{2 b - a + i}		& 0 \leq i \leq 2 a - 3 b - 1 \\
	0^{2 b - a - 1 - i} \, 1 \, 0^{a - b - 1} \, (n + 1) \, 0^{a - b + i}		& 0 \leq i \leq 2 b - a - 1 \\
	0^{a - b - 1 - i} \, (n + 1) \, 0^{b + i}					& 0 \leq i \leq a - b - 1
\end{array}
\]
Therefore each length-$a$ factor of $\varphi(w)$ appears in the previous table for some $n \geq 0$ and some $i$ in the specified range.
Each $\ab$-power of length $a$ is of the form $xyx$ where $|x| = a-b$ and $|y| = 2b-a$, so to check that no length-$a$ factor of $\varphi(w)$ is an $\ab$-power we refine the previous table by writing each factor as $xyz$ with $|x| = a-b$, $|y| = 2b-a$, and $|z| = a-b$.
\[
\begin{array}{rcl|l}
	\text{$x$ \gray{(length $a - b$)}}		& \text{$y$ \gray{(length $2 b - a$)}}	& \text{$z$ \gray{(length $a - b$)}}			& \text{interval for $i$} \\ \hline
	0^{a - b}					& 0^{2 b - a}				& 0^{a - b - 1 - i} \, 1 \, 0^i				& 0 \leq i \leq a - b - 1 \\
	0^{a - b}					& 0^{2 b - a - 1 - i} \, 1 \, 0^i	& 0^{a - b - 1 - i} \, (n + 1) \, 0^i			& 0 \leq i \leq 2 b - a - 1 \\
	0^{a - b - 1 - i} \, 1 \, 0^i			& 0^{2 b - a}				& 0^{2 a - 3 b - 1 - i} \, (n + 1) \, 0^{2 b - a + i}	& 0 \leq i \leq 2 a - 3 b - 1 \\
	0^{2 b - a - 1 - i} \, 1 \, 0^{2 a - 3 b + i}	& 0^{2 b - a - 1 - i} \, (n + 1) \, 0^i	& 0^{a - b}						& 0 \leq i \leq 2 b - a - 1 \\
	0^{a - b - 1 - i} \, (n + 1) \, 0^i		& 0^{2 b - a}				& 0^{a - b}						& 0 \leq i \leq a - b - 1
\end{array}
\]
Since $n + 1 \neq 0$, we see that $x \neq z$ for each factor.
(When $n=0$, we have $x \neq z$ in the third row since $i \neq 2b-a+i$.)
It follows that $\varphi(w)$ contains no $\ab$-power factor of length $a$.

Therefore $m \geq 2$.
There are two cases.

First consider the possibility that $x = 0^{m \cdot (a-b)}$.
The longest zero factor of $\varphi(w)$ is $0^{a-1}$, so $m \cdot (a-b) < a$, which implies $m < \frac{a}{a-b} \leq \frac{a}{a - (3/5) a} = \frac{5}{2}$.
Therefore $m=2$.
The factor $y$ has at least one nonzero letter, because $x 0^{2 (2b-a)} x = 0^{2 a}$ is not a factor of $\varphi(w)$.
On the other hand, if $y$ has at least two nonzero letters, then, since the shortest maximal zero factor of $\varphi(w)$ is $0^{a-b-1}$, we have $2 (2b-a) = |y| \geq 1 + (a - b - 1) + 1$, which implies $5 b \geq 3a + 1$, but this contradicts $\frac{5}{3} \leq \ab$.
Therefore $y$ has exactly one nonzero letter.
Therefore $xyx = 0^{2 (a-b)} 0^i (n + 1) 0^{2 (2b-a) - 1 - i} 0^{2 (a-b)}$ for some $n \geq 0$ and $0 \leq i \leq 2 (2b-a) - 1$.
But the longest zero factor of $\varphi(w)$ is $0^{a-1}$, so $2 (a-b) + i \leq a-1$ and $2 (2b-a) - 1 - i + 2(a-b) \leq a-1$, which produce a contradiction when solving for $i$.

Now consider the case that $x$ contains a nonzero letter.
The word $0^{a-b-1} 1 0^{a-b-1}$ is the only nonzero word of its length that can occur in $\varphi(w)$ at two positions that are distinct modulo $k \colonequal |\varphi(n)| = 2a-b$, and extending this word in either direction determines its position modulo $k$; therefore two occurrences in $\varphi(w)$ of any nonzero word of length $\geq 2 (a - b)$ have positions that are congruent modulo $k$.
Since $|x| = m \cdot (a - b) \geq 2 (a - b)$, the positions of the two occurrences of $x$ in $\varphi(w)$ are congruent modulo $k$.
These two positions differ by $|xy| = m b$, so $k \mid m b$.
This implies $k \mid m$, since $\gcd(b,k) = 1$.
Therefore $k$ divides $|xyx|$.
Now we shift $xyx$ appropriately.
Let $j$ be the position of $xyx$ in $\varphi(w)$.
Write $j = k i + r$ for some $0 \leq r \leq k - 1$.
Let $x' y' z'$ be the word of length $|xyx|$ beginning at position $k i$, where $|x'| = |z'| = |x|$ and $|y'| = |y|$.
We claim that $x' = z'$.
Since $|x' y' z'|$ begins at position $k i$ and $|xy|$ is divisible by $k$, the words $x'$ and $z'$ agree on their first $r$ letters.
The remaining $|x| - r$ letters of $x'$ are the first $|x| - r$ letters of $x$, and the remaining $|x| - r$ letters of $z'$ are the first $|x| - r$ letters of the second $x$.
Therefore $x' = z'$, and we have found an $\ab$-power $x'y'x'$ beginning at position $k i$.
Since the positions of $y'$ and both occurrences of $x'$ are all divisible by $k$, we have $x' = \varphi(u)$ and $y' = \varphi(v)$ for some $u$ and $v$, which implies that the $\ab$-power $uvu$ is a factor of $w$.

It remains to show that $\varphi^\infty(0)$ is lexicographically least.
We show that decrementing any nonzero letter in $\varphi^\infty(0)$ to any smaller number introduces an $\ab$-power factor ending at that position.
Since $\varphi(n) = 0^{a-1} 1 0^{a-b-1} (n+1)$, the nonzero letters occur at positions congruent to $a - 1$ or $k - 1$ modulo $k$.
The letter at each position congruent to $a - 1$ modulo $k$ is $1$, and decrementing this $1$ to $0$ introduces the $\ab$-power $0^a = (0^b)^{a/b}$ ending at that position.
The letter at a position congruent to $a - 1$ modulo $k$ is $n+1$ for some $n \geq 0$.
Consider the effect of decrementing $n+1$ to $c$ for some $0 \leq c \leq n$.
If $c = 0$, then this introduces the $\ab$-power $0^{b-1} 1 0^{a-b} = (0^{b-1} 1)^{a/b}$.
Let $c \geq 1$, and assume that decrementing any letter to $c - 1$ introduces an $\ab$-power ending at this $c - 1$.
Let $\varphi(w)$ be a prefix of $\varphi^\infty(0)$ with last letter $n + 1$.
Then $w$ is a prefix of $\varphi^\infty(0)$ with last letter $n$.
Decrementing $n+1$ to $c$ produces the word $\varphi(w')$, where $w'$ is the word obtained by decrementing the last letter of $w$ to $c-1$.
By the inductive assumption, $w'$ contains an $\ab$-power suffix; therefore $\varphi(w')$ does as well.
\end{proof}

%%%%%%%%%%%%%%%%%%%%%%%%%%%%%%%%%%%%%%%%%%
\section{$\ab$-power-free morphisms}\label{power-free morphisms}
%%%%%%%%%%%%%%%%%%%%%%%%%%%%%%%%%%%%%%%%%%

Theorem~\ref{2 2a-b} establishes the structure of $\word_{a/b}$ for an infinite family of rationals~$\ab$.
It turns out there are additional families of words whose structure is given by a symbolic morphism, and there are also many words $\word_{a/b}$ whose structure is given by a morphism that has not been found to belong to a general family.
(Additionally, as in Theorem~\ref{6/5}, sometimes the word $\word_{a/b}$ is not a fixed point of a uniform morphism but is nonetheless related to a uniform morphism.)
Ideally, we would prove a single theorem that captures all these cases.
However, the structures are diverse enough that it is not clear how to unify them.
The next best thing, then, is to identify a general proof scheme so that each individual proof may be carried out automatically.
In this section we describe how to automatically verify that a morphism is $\ab$-power-free.
We then apply this method to $30$ symbolic morphisms.

The basic idea is that we use the special form of the morphisms to reduce the statement that $\varphi$ is $\ab$-power-free to a finite case analysis and then develop software to carry out the case analysis.
In the case of an explicit rational number (as in Theorems~\ref{8/5} and \ref{7/4}) this is more or less straightforward using the results below.
However, for parameterized morphisms that are symbolic in $a$ and $b$ (as in Theorem~\ref{2 2a-b}), this can require a significant amount of symbolic computation.

\subsection{Bounding the factor length}\label{bounding the factor length}

As in the proof of Theorem~\ref{2 2a-b}, to show that $\varphi$ is $\ab$-power-free, we must verify that if $\varphi(w)$ contains an $\ab$-power then $w$ contains an $\ab$-power.
In this subsection we reduce this task to the task of verifying the statement for factors of length $am$ for only finitely many values of $m$.
We use the following concept, which is related to the \emph{synchronization delay} introduced by Cassaigne~\cite{Cassaigne}.

\begin{definition*}
Let $k \geq 2$ and $\ell \geq 1$.
Let $\varphi$ be a $k$-uniform morphism on $\Sigma$.
We say that \emph{$\varphi$ locates words of length $\ell$} if for each word $x$ of length $\ell$ there exists an integer $j$ such that, for all $w \in \Sigma^*$, every occurrence of the factor $x$ in $\varphi(w)$ begins at a position congruent to $j$ modulo $k$.
\end{definition*}

If $\varphi$ locates words of length $\ell$, then $\varphi$ also locates words of length $\ell + 1$, since if $|x| = \ell + 1$ then the position of the length-$\ell$ prefix of $x$ is determined modulo $k$.

For example, one checks that the morphism $\varphi(n) = 000010(n+1)$, for which $\word_{5/3} = \varphi^\infty(0)$, locates words of length $7$.
More generally, under mild conditions the $k$-uniform morphism $\varphi(n) = u \, (n + d)$ locates words of length $k$.

\begin{lemma}\label{locating length}
Let $k \geq 2$ and $d \geq 0$.
Let $u \in \Z_{\geq 0}^{k-1}$ be a word of length $k-1$, and let $\varphi$ be the $k$-uniform morphism defined by $\varphi(n) = u \, (n + d)$.
If, for all integers $n \geq 0$ and all integers $a \geq 2$, the word $u \, (n + d)$ is not an $a$-power, then $\varphi$ locates words of length $k$.
\end{lemma}

\begin{proof}
Write $u = u(0) u(1) \cdots u(k-2)$, and let $w \in \Z_{\geq 0}^*$.
Every length-$k$ factor of $\varphi(w)$ is of the form
\[
	v = u(i) \cdots u(k-2) \, (n + d) \, u(0) \cdots u(i-1)
\]
for some $n \geq 0$ and $0 \leq i \leq k - 1$.
Suppose $v$ occurs elsewhere in $\varphi(w)$.
Then without loss of generality we have
\[
	u \, (n' + d) = u(i) \cdots u(k-2) \, (n + d) \, u(0) \cdots u(i-1)
\]
for some $n'$.
This implies $u(j) = u(j + i \bmod k)$ for each $j$ such that $0 \leq j \leq k - 2$ and $0 \leq (j + i \bmod k) \leq k - 2$.

If $i = 0$ then $n' = n$ and the two occurrences of $v$ begin at positions that are congruent modulo $k$.
If $i \neq 0$, then $u(i-1) = u((i-1) + t i \bmod k)$ for $0 \leq t \leq \frac{k}{\gcd(i,k)} - 2$ (because for $t = \frac{k}{\gcd(i,k)} - 1$ we have $((i-1) + t i \bmod k) = k-1 $ and $u(k-1)$ is not defined).
Letting $t = \frac{k}{\gcd(i,k)} - 2$ gives $n' + d = u(i-1) = u(k-i-1) = n + d$.
Therefore $n' = n$, and, since $u \, (n + d)$ is not an $a$-power, we have $i = 0$.
\end{proof}

An $\ab$-power $(xy)^{a/b} = xyx$ contains two occurrences of $x$, so if a morphism $\varphi$ locates words of some length, then the length of $xy$ for sufficiently long $\ab$-powers in $\varphi(w)$ is constrained to be divisible by $k$.
In Lemma~\ref{big x} and Proposition~\ref{big m} we use this to bound the length of factors of $\varphi(w)$ that we must verify are not $\ab$-powers in order to conclude that $\varphi$ is $\ab$-power-free.

\begin{lemma}\label{big x}
Let $a, b$ be relatively prime positive integers such that $1 < \ab < 2$.
Let $k \geq 2$ such that $\gcd(b, k) = 1$, and let $\ell \geq 1$.
Let $\varphi$ be a $k$-uniform morphism on a (finite or infinite) alphabet $\Sigma$ such that
\begin{itemize}
\item
$\varphi$ locates words of length $\ell$, and
\item
for all $n, n' \in \Sigma$, the words $\varphi(n)$ and $\varphi(n')$ differ in at most one position.
\end{itemize}
Then $w$ contains an $\ab$-power whenever $\varphi(w)$ contains an $\ab$-power $(xy)^{a/b} = xyx$ with $|x| \geq \ell$.
\end{lemma}

The proof generalizes the case $m \geq 2$ in the proof of $\ab$-power-freeness in Theorem~\ref{2 2a-b}.

\begin{proof}
Suppose $w$ is a word such that $\varphi(w)$ contains an $\ab$-power $(xy)^{a/b} = xyx$ with $|x| \geq \ell$.
Let $m = |xyx|/a$; as before, $|x| = m \cdot (a-b)$ and $|y| = m \cdot (2b-a)$.
Since $\varphi$ locates words of length $\ell$, $\varphi$ also locates words of length $|x|$.
Since $xyx$ is a factor of $\varphi(w)$, this implies that $k$ divides $|xy| = m b$.
Since $\gcd(b, k) = 1$, it follows that $k \mid m$, and therefore $k$ divides both $|x| = m \cdot (a-b)$ and $|y| = m \cdot (2b-a)$.

Let $j$ be the position of $xyx$ in $\varphi(w)$.
Write $j = k i_1 + r$ for some $0 \leq r \leq k - 1$.
Then $y$ begins at position $j + |x| = k i_2 + r$ and the second $x$ begins at position $j + |xy| = k i_3 + r$ for some $i_2, i_3$.
Since $\varphi(n)$ and $\varphi(n')$ differ in at most one position, adjacent factors of length $k$ in $\varphi(w)$ differ in at most one position, so we can slide a window of length $|xyx|$ either to the left or to the right from $xyx$ and obtain an $\ab$-power factor $x' y' x'$ beginning at position $k i_1$ or $k i_1 + k$ with $|x'| = |x|$ and $|y'| = |y|$.
Since the positions of $y'$ and both occurrences of $x'$ are all divisible by $k$, we have $x' = \varphi(u)$ and $y' = \varphi(v)$ for some $u$ and $v$, which implies that $uvu = (uv)^{a/b}$ is a factor of $w$.
\end{proof}

\begin{proposition}\label{big m}
Assume the hypotheses of Lemma~\ref{big x}.
Let $\Imin \in \Q$ such that $1 < \Imin < \ab$, and let $c \geq d \geq 0$ be integers such that $\ell = c a - d b$.
Let $\mmax \colonequal \lceil \frac{c \cdot \Imin - d}{\Imin - 1} \rceil - 1$.
Suppose additionally that for every $\ab$-power-free word $w \in \Sigma^*$ the word $\varphi(w)$ contains no $\ab$-power $(xy)^{a/b} = xyx$ of length $m a$ for $1 \leq m \leq \mmax$.
Then $\varphi$ is $\ab$-power-free.
\end{proposition}

\begin{proof}
Let $w \in \Sigma^*$ be $\ab$-power-free.
Suppose toward a contradiction that $\varphi(w)$ contains an $\ab$-power $(xy)^{a/b} = xyx$.
Let $m = |xyx|/a$; then $|x| = m \cdot (a-b)$.
By Lemma~\ref{big x}, $m = \frac{|x|}{a-b} < \frac{\ell}{a-b} = \frac{c a - d b}{a-b}$.
Since $c \geq d$, the function $t \mapsto \frac{c t - d}{t - 1}$ is non-increasing for $t > 1$.
Therefore
\[
	m < \frac{c a - d b}{a-b} = \frac{c\ab-d}{\ab-1} \leq \frac{c \cdot \Imin - d}{\Imin - 1}.
\]
Since $m$ is an integer, we have $m \leq \lceil \frac{c \cdot \Imin - d}{\Imin - 1} \rceil - 1 = \mmax$.
This contradicts an assumption, so $\varphi(w)$ is $\ab$-power-free and hence $\varphi$ is $\ab$-power-free.
\end{proof}

Given particular values of $\ell$ and $\ab$, we may have several choices for $c$ and $d$, and there are many choices for $\Imin$.
But we will be applying Proposition~\ref{big m} to families of morphisms for which $\Imin$, $c$, and $d$ are fixed.

\begin{example*}
Consider the morphism
\[
	\varphi(n) = 0^{a - 1} \, 1 \, 0^{a - b - 1} \, (n + 1)
\]
in Theorem~\ref{2 2a-b}, for which $k = 2 a - b$.
This morphism locates words of length $\ell = a$, as one can verify with the assistance of the table of length-$a$ factors of $\varphi(w)$ in the proof of Theorem~\ref{2 2a-b}.
To show that $\varphi$ is $\ab$-power-free, by Proposition~\ref{big m} it suffices to verify for every $\ab$-power-free $w$ that $\varphi(w)$ contains no $\ab$-power factors of length $m a$ for $m \leq \mmax = \lceil \frac{c \cdot \Imin - d}{\Imin - 1} \rceil - 1 = \lceil \frac{1 \cdot 5/3 - 0}{5/3 - 1} \rceil - 1 = \lceil \frac{5}{2} \rceil - 1 = 2$.
\end{example*}

Under some mild conditions, it turns out that the relevant factors of $\varphi(w)$ are sufficiently short that $w$ is necessarily $\ab$-power-free.
Namely, any factor of $\varphi(w)$ of length $m a$ is necessarily a factor of $\varphi(v)$ for some word $v$ of length $\lceil \frac{m a - 1}{k} \rceil + 1$, and the following lemma shows that $\lceil \frac{\mmax a - 1}{k} \rceil + 1 \leq a - 1$.
Since the shortest $\ab$-powers have length $a$, this guarantees that $v$ is $\ab$-power-free.
Therefore, in verifying the hypotheses of Proposition~\ref{big m}, we will replace ``for every $\ab$-power-free word $w$'' with ``for every word $w$''; we need not expend the computational effort to determine whether each $w$ is $\ab$-power-free, since it is too short to contain an $\ab$-power.

\begin{lemma}\label{short words are a/b-power-free}
Let $s, t, \mmax \in \Z_{\geq 0}$ and $\Imin, \Imax \in \Q$ such that $s \neq 0$ and $1 < \Imin < \Imax \leq 2$.
Let
\[
	a_\textnormal{min} \colonequal \min \left\{a \geq 1 : \text{$\Imin < \tfrac{a}{b} < \Imax$ for some $b \geq 1$ with $\gcd(b, s) = 1$}\right\}.
\]
Assume
\[
	\mmax \leq \left(s - \frac{t}{\Imin}\right) (a_\textnormal{min} - 2).
\]
Let $a, b$ be relatively prime positive integers such that $\Imin < \ab < \Imax$ and $\gcd(b, s) = 1$.
Let $k = s a - t b$.
Then $\lceil \frac{\mmax a - 1}{k} \rceil + 1 \leq a - 1$.
\end{lemma}

\begin{proof}
The condition $s \neq 0$ implies $k \neq 0$, so we can divide by $k$.
The condition $1 < \Imin < \Imax \leq 2$ implies $a_\textnormal{min} \geq 3$.
We have
\[
	\mmax
	\leq \left(s - \frac{t}{\Imin}\right) (a_\textnormal{min} - 2)
	< \left(s - \frac{t}{\ab}\right) (a - 2)
	= \frac{k}{a} (a - 2).
\]
This now implies $\frac{\mmax a}{k} < a - 2$, which implies $\lceil\frac{\mmax a}{k}\rceil \leq a - 2$, and therefore $\lceil \frac{\mmax a - 1}{k} \rceil + 1 \leq \lceil \frac{\mmax a}{k} \rceil + 1 \leq a - 1$.
\end{proof}

There will be no difficulty in satisfying the conditions of Lemma~\ref{short words are a/b-power-free} for the morphisms we encounter.

\begin{example*}[continued]
The $(2a-b)$-uniform morphism $\varphi$ in Theorem~\ref{2 2a-b} locates words of length $a$, so $s = 2, t = 1$ and $c = 1, d = 0$.
Earlier we computed $\mmax = 2$ for this morphism on the interval $\frac{5}{3} < \ab < 2$.
The smallest numerator of a rational number in the interval $\frac{5}{3} < \ab < 2$ with an odd denominator is $a_\textnormal{min} = 9$, so we have $\left(s - \frac{t}{\Imin}\right) (a_\textnormal{min} - 2) = \frac{49}{5} \geq 2 = \mmax$ and Lemma~\ref{short words are a/b-power-free} applies.
To conclude that $\varphi$ is $\ab$-power-free, it remains to verify that no factor of $\varphi(w)$ of length $a$ or $2a$ is an $\ab$-power.
(Note that in the proof of Theorem~\ref{2 2a-b} we only checked factors of length $a$ by a table, suggesting the better bound $\mmax = 1$.
But there we used a slightly different argument, treating separately the case where $x$ consists only of zeros, which allowed the argument for large $m$ to apply to factors of length $2a$.)
\end{example*}

We would like to use Proposition~\ref{big m} to prove, with as much automation as possible, that a given morphism $\varphi$, symbolic in $a$ and $b$, is $\ab$-power-free.
For now we assume that the interval restricting $\ab$ is also given.
The main steps required are the following.
\begin{enumerate}
\item\label{step1}
Identify an integer $\ell$ such that $\varphi$ locates words of length $\ell$.
\item\label{step2}
Verify that for every word $w \in \Sigma^*$ the word $\varphi(w)$ contains no $\ab$-power of length $m a$ with $1 \leq m \leq \mmax$, where $\mmax$ is determined by $\ell$.
\end{enumerate}

Step~\eqref{step1} can be accomplished simply by using Lemma~\ref{locating length}, but to carry out the computations for morphisms of moderate size we need to obtain a smaller value for $\ell$; we return to this in Section~\ref{determining a locating length}.
Step~\eqref{step2} works as in the proof of Theorem~\ref{2 2a-b} by listing, for each $m$, all symbolic factors of length $m a$ and verifying that none are $\ab$-powers.
We discuss the details in Sections~\ref{listing factors} and \ref{testing inequality}.
At this point the reader may be interested in looking at the theorems in Section~\ref{symbolic power-free morphisms} that we have proved with this approach.

\subsection{Listing factors of a given length}\label{listing factors}

In the proof of Theorem~\ref{2 2a-b} we generated a table of all possible length-$a$ factors of words of the form $\varphi(w)$.
The idea behind the automatic generation of tables like this is that we slide a window of some length through the infinite word
\[
	\varphi(n) \varphi(n) \varphi(n) \cdots
\]
and stop when we reach a factor that is identical to the first.
In Theorem~\ref{2 2a-b}, the window was short enough relative to $|\varphi(n)|$ that each factor contained at most one letter $n + 1$.
In general, this may not be the case; since we are considering arbitrary words we must slide a window through the word
\[
	\varphi(n_0) \varphi(n_1) \varphi(n_2) \cdots.
\]
However, it suffices to slide a window through the periodic word $\varphi(n) \varphi(n) \cdots$, stopping as before, and simply rename each occurrence of $n$ in a factor to be a unique symbol $n_i$ before performing a test on that factor (such as determining whether it is an $\ab$-power).

When $\ab$ is an explicit rational number, $\varphi(n)$ is a word of some explicit integer length.
In this case, sliding a window through $\varphi(n) \varphi(n) \cdots$ is trivial; we simply increment the starting and ending positions of the window by $1$ at each step.

For symbolic $\ab$, we build a table as in Theorem~\ref{2 2a-b}, where each symbolic factor is parameterized by $i$ in some interval.
We treat $\varphi(n) \varphi(n) \cdots$ as a queue.
Suppose the window length is $m a$.
To compute the prefix of $\varphi(n) \varphi(n) \cdots$ of length $m a$, we begin with $f = \epsilon$ as the empty word and record the number $m a - |f|$ of remaining letters to add.
Initially there are $m a$ remaining letters.
At each step, we have a block $c^l$ of $l$ identical letters $c$ at the beginning of the queue, and we need to determine whether to add the entire block or just a part of it to our factor.
If $m a - |f| \geq l$, then we take the entire block; otherwise we take the partial block $c^{m a - |f|}$.
If we are creating a table of factors that are further factored into subfactors of lengths $m \cdot (a - b)$, $m \cdot (2b - a)$, and $m \cdot (a - b)$, then we do this procedure once for each subfactor.

To slide the window to the right, we add a parameter $i$ to the run lengths of the current factor.
We must determine the maximum value $i_\text{max}$ such that sliding the window through $i_\text{max}$ positions maintains factors whose run-length encodings differ from the current factor only in their exponents.
The value of $i_\text{max}$ is the minimum of the first block length in $f$ and the first block length in the queue (or, if we are factoring $f$ into three subfactors, the minimum of the first block length in each of the three subfactors as well as the first block length in the queue).
Drop $i$ letters from the front of each subfactor, and add those $i$ letters onto the end of the preceding subfactor (or throw them away from the first subfactor).
This gives a factor parameterized by $i$ for $1 \leq i \leq i_\text{max}$ (or $0 \leq i \leq i_\text{max}$ for the first factor), which we add to the list of factors.
(If an interval contains only one point and the next interval contains more than one point, we can merge them to get an interval $0 \leq i \leq i_\text{max}$ as in the tables of Theorem~\ref{2 2a-b}.)
Then replace $i$ with $i_\text{max}$ in the factor, drop $i_\text{max}$ letters from the front of the queue, and repeat.

The run lengths for the symbolic morphisms we encounter are linear combinations of $a, b, 1$.
Therefore to compute $i_\text{max}$ we must be able to compute the minimum of two such expressions over an interval.
For example, if $\frac{5}{3} < \ab < 2$ and $a,b \in \Z_{\geq 1}$ then $\min(2 a - 2 b, a - 1) = 2 a - 2 b$.

If the minimum is not equal to either of its arguments on the entire interval, then we split the interval.
For example, when we encounter $\min(a - b - 1, -4 a + 7 b)$ for the interval $\frac{3}{2} < \ab < \frac{5}{3}$, we solve the homogeneous equation $a - b = -4 a + 7 b$ to find $\ab = \frac{8}{5}$ and split the interval into three subintervals $\frac{3}{2} < \ab < \frac{8}{5}$, $\ab = \frac{8}{5}$, and $\frac{8}{5} < \ab < \frac{5}{3}$.
We continue breaking up subintervals until we can compute the symbolic factors of length $m a$ for $\ab$ in each subinterval.
Later, when we compute factors of length $(m + 1) a$, we start with the set of subintervals obtained for length~$m a$.
Even though the length has changed, empirically it seems that most of the same subintervals reappear, so this saves the work of recomputing them.

\subsection{Testing inequality of symbolic words}\label{testing inequality}

Once we have generated all possible factors of $\varphi(w)$ of length $m a$, we must verify that, for all values of parameters that appear ($n$, $a$, $b$, and any interval parameters $i, j$), each factor is not an $\ab$-power.
Since we have factored each word as $xyz$ with $|x| = |z| = m \cdot (a-b)$, it suffices to verify that $x \neq z$.

An example from the proof of Theorem~\ref{2 2a-b} is $x = 0^{a - b - 1 - i} \, 1 \, 0^i$ and $z = 0^{2 a - 3 b - 1 - i} \, (n + 1) \, 0^{2 b - a + i}$.
Under the assumptions $n \geq 0$ and $\frac{5}{3} \leq \ab < 2$, we conclude that $x \neq z$ by comparing prefixes.
Namely, the two words $c_1^{l_0} c_2^{l_2} \cdots$ and $c_1^{l_1} c_3^{l_3} \cdots$ on the letters $c_i$ are unequal if $c_1 \neq c_2$, $c_1 \neq c_3$, $l_0 \neq l_1$, $l_2 \geq 1$, and $l_3 \geq 1$.

We have not developed a decision procedure to decide if there exist parameter values for which two symbolic words are equal.
For our purposes it is sufficient to show that pairs of symbolic words we encounter are unequal for all parameter values.
We have implemented a number of criteria under which this is true.
For example, if two words have identical prefixes (or suffixes), we can remove the common factor and recursively test inequality of the remaining factors.
If the first letters or last letters in two words are unequal, then the words are unequal.

Another criterion is the following.
Delete all explicit $0$ letters in both words.
If all remaining letters are unequal to $0$ and the two new words are unequal, then the original words are unequal.
It may happen that deleting $0$s does not result in words that are unequal.
For example, deleting $0$s in the words
\begin{align*}
	& 0^{352 a-621 b-i-1} \, 1 \, 0^{-51 a+91 b-1} \, (n+1) \, 0^{i}, \\
	& 0^{-51 a+91 b-j-1} \, (n+1) \, 0^{352 a-621 b-1} \, 1 \, 0^{j}
\end{align*}
produces $1 \, (n + 1)$ and $(n + 1) \, 1$, which are not unequal if $n = 0$.
We may still conclude the original words are unequal if the system of equalities of the corresponding deleted block lengths has no solution.
In this example, $-51 a+91 b-1 \neq 352 a-621 b-1$ on the interval $\frac{30}{17} < \ab < \frac{53}{30}$.

We use our collection of inequality criteria to verify that no symbolic factor represents an $\ab$-power, for the list of factors obtained on each subinterval by the process in Section~\ref{listing factors}.
Since the subintervals are open intervals, we must also verify their endpoints.
That is, when $\ab$ is an endpoint of a subinterval and satisfies the conditions of the theorem we are trying to prove (that is, it lies in the interval and is not eliminated by a $\gcd$ condition), then we check that the factors for this value of $\ab$ (which have explicit integer run lengths) are not $\ab$-powers.

\subsection{Determining a locating length}\label{determining a locating length}

The final step to automate is the identification of an integer $\ell$ such that $\varphi$ locates words of length $\ell$.
Lemma~\ref{big x} and Proposition~\ref{big m} show that the length $\ell$ of words that a morphism $\varphi$ locates affects the length of factors we must check.
Lemma~\ref{locating length} gives one possible length, but often we can find a smaller length.
For example, as mentioned in Section~\ref{bounding the factor length}, the $(2 a - b)$-uniform morphism in Theorem~\ref{2 2a-b} not only locates words of length $2 a - b$ but also locates words of length $a$.

Given a morphism $\varphi$ and an interval $\Imin < \ab < \Imax$, we determine $\ell$ as follows.
Generate all linear combinations $c a - d b$ with $10 \geq c \geq d \geq 0$.
The upper bound $10$ is sufficient for all morphisms we encounter below; for a more general upper bound, one could use $s$, where $k = s a - t b$, and then we are guaranteed to find a suitable $\ell$ under the conditions of Lemma~\ref{locating length}.
Eliminate any linear combinations that do not satisfy the hypothesis of Lemma~\ref{short words are a/b-power-free}, where $\mmax = \lceil \frac{c \Imin - d}{\Imin - 1} \rceil - 1$.
Then sort the remaining linear combinations by the upper bound $\mmax$ that each would imply if $\varphi$ locates words of that length.
Starting with the lowest potential bounds, test whether $\varphi$ locates words of each length until a length is found.

To determine whether $\varphi$ locates words of a given length, use the procedure described in Section~\ref{listing factors} to compute all symbolic factors of $\varphi(n) \varphi(n) \cdots$ of the candidate length.
As before, this may involve breaking up the interval $\Imin < \ab < \Imax$ into subintervals.
Then check whether each pair of symbolic factors is unequal using the tests described in Section~\ref{testing inequality}.

\subsection{Symbolic $\ab$-power-free morphisms}\label{symbolic power-free morphisms}

We now give $30$ symbolic $\ab$-power-free morphisms defined on the alphabet $\Z_{\geq 0}$.
With the exception of Theorem~\ref{8 5a-3b power-free}, these morphisms were discovered empirically from prefixes of words $\word_{a/b}$.
In Section~\ref{finding conjectures} we discuss the details.
We list morphisms in order of increasing number of nonzero letters in $\varphi(n)$.
The $30$ intervals are shown in Figure~\ref{number line plot}.

All $30$ theorems have the same form.
Each concerns a $k$-uniform morphism $\varphi$ parameterized by $a$ and $b$ with $\gcd(a, b) = 1$.
The ratio $\ab$ is restricted to some interval, the lower endpoint of which is $\Imin$, and we have $k = s a - t b$ for some integers $s, t$.
There is a divisibility condition on $b$ coming from Lemma~\ref{big x}, namely $\gcd(b,k) = 1$.
Since $a$ and $b$ are relatively prime, $\gcd(b,k) = 1$ is equivalent to $\gcd(b,s) = 1$; we write the latter since it is more explicit.

For each morphism, $\ab$-power-freeness is proved completely automatically.
The computations are available from the web site of the second-named author\footnote{\url{http://people.hofstra.edu/Eric_Rowland/papers.html\#Avoiding_fractional_powers_over_the_natural_numbers} as of this writing.}.
Theorems~\ref{2 2a-b power-free} and \ref{3 a power-free} are each proved in approximately $1$ second, including the time required to find $\ell$.
However, longer morphisms require more time; Theorem~\ref{279 67a-30b power-free} took more than $7$ hours to prove on a modern laptop (although this could be reduced by computing in parallel).
Theorems~\ref{2 2a-b power-free} and \ref{14 8a-4b power-free} include the lower endpoint of their intervals; a separate step establishes each theorem at this endpoint.

\begin{figure}
	\includegraphics[width=\textwidth]{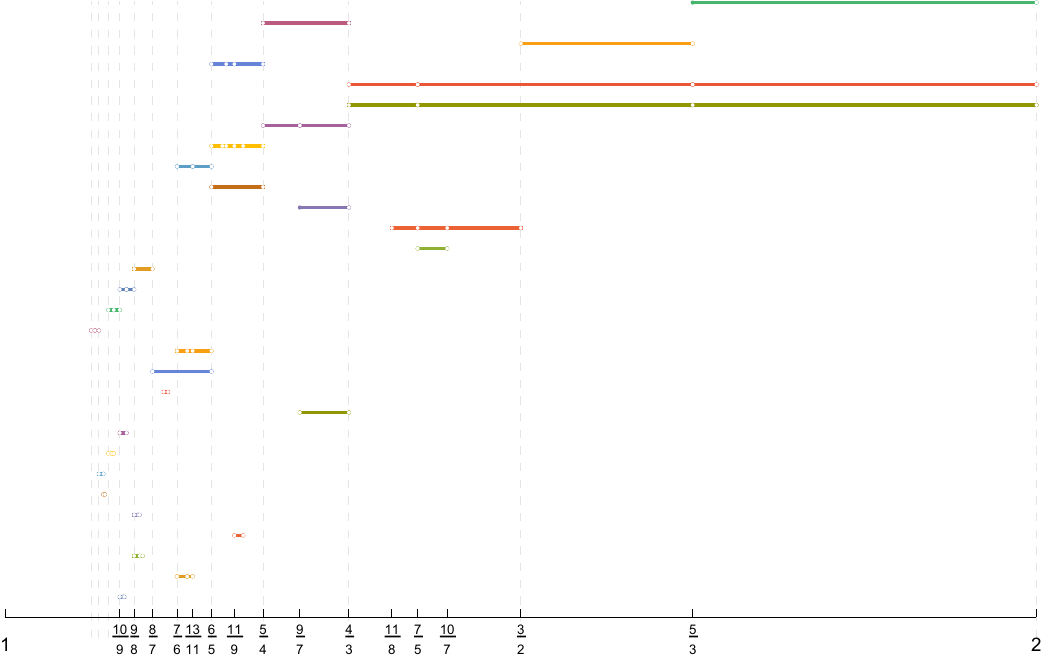}
	\caption{A plot of the intervals in Theorems~\ref{2 2a-b power-free}--\ref{279 67a-30b power-free}.  On each interval, the rational numbers satisfying a $\gcd$ condition support an $\ab$-power-free morphism.}
	\label{number line plot}
\end{figure}

The first theorem concerns the morphism in Theorem~\ref{2 2a-b}.
We have already proven that $\varphi$ is $\ab$-power-free, but we include it here for completeness.
Previously we showed that $\varphi$ locates words of length $a$; our automatic procedure reduces this length.

\begin{theorem}\label{2 2a-b power-free}
Let $a, b$ be relatively prime positive integers such that $\frac{5}{3} \leq \ab < 2$ and $\gcd(b,2) = 1$.
Then the $(2a-b)$-uniform morphism
\[
	\varphi(n) = 0^{a - 1} \, 1 \, 0^{a - b - 1} \, (n+1),
\]
with $2$ nonzero letters, locates words of length $2 a - 2 b$ and is $\ab$-power-free.
\end{theorem}

\begin{theorem}\label{3 a power-free}
Let $a, b$ be relatively prime positive integers such that $\frac{5}{4} < \ab < \frac{4}{3}$.
Then the $a$-uniform morphism
\[
	\varphi(n) = 0^{5 a - 6 b - 1} \, 1 \, 0^{-2 a + 3 b - 1} \, 1 \, 0^{-2 a + 3 b - 1} \, (n+2),
\]
with $3$ nonzero letters, locates words of length $3 a - 3 b$ and is $\ab$-power-free.
\end{theorem}

We have two morphisms with $4$ nonzero letters.

\begin{theorem}\label{4 5a-4b power-free}
Let $a, b$ be relatively prime positive integers such that $\frac{3}{2} < \ab < \frac{5}{3}$ and $\gcd(b,5) = 1$.
Then the $(5a-4b)$-uniform morphism
\[
	\varphi(n) = 0^{a-1} \, 1 \, 0^{a-b-1} \, 1 \, 0^{2 a-2 b-1} \, 1 \, 0^{a-b-1} \, (n+1),
\]
with $4$ nonzero letters, locates words of length $5 a - 5 b$ and is $\ab$-power-free.
\end{theorem}

\begin{theorem}\label{4 a power-free}
Let $a, b$ be relatively prime positive integers such that $\frac{6}{5} < \ab < \frac{5}{4}$ and $\ab \notin \{\frac{11}{9}, \frac{17}{14}\}$.
Then the $a$-uniform morphism
\[
	\varphi(n) = 0^{6 a-7 b-1} \, 1 \, 0^{-3 a+4 b-1} \, 1 \, 0^{-8 a+10 b-1} \, 1 \, 0^{6 a-7 b-1} \, (n+1),
\]
with $4$ nonzero letters, locates words of length $a$ and is $\ab$-power-free.
\end{theorem}

For $6$ nonzero letters, there are two $(4a-2b)$-uniform morphisms that apply to the same set of rationals $\ab$, given in Theorems~\ref{6 4a-2b I power-free} and \ref{6 4a-2b II power-free}.
The morphisms are closely related; they are
\begin{align*}
	& \varphi(n) = u \, 1 \, v \, (n+1), \\
	& \varphi(n) = v \, 1 \, u \, (n+1)
\end{align*}
for two words $u, v$.
Let us use them to illustrate some details of the algorithm.
Proving that each of these morphisms is $\ab$-power-free takes the computer approximately $2$ minutes; the interval $\frac{4}{3} < \ab < 2$ is broken up into $21$ subintervals during the process of Section~\ref{listing factors}.
Some of the endpoints of these subintervals have even denominators; we need not check these for $\ab$-powers since we assume $\gcd(b,4) = 1$.
In checking the subinterval endpoints with odd denominators, we find that $\frac{5}{3}$ and $\frac{7}{5}$ do yield $\ab$-powers.
For example, for the morphism $\varphi$ in Theorem~\ref{6 4a-2b I power-free}, the word $\varphi(n_0) \varphi(n_1) \cdots$ contains $(001100 \cdot 110)^{5/3} = 001100 \cdot 110 \cdot 001(n_1+1)00$ and
\begin{multline*}
	(1001000100 \cdot (n_0+1)00010011001000)^{7/5} \\
	= 1001000100 \cdot (n_0+1)00010011001000 \cdot 100(n_1+1)000100
\end{multline*}
in the case $n_1 = 0$.
Therefore we must exclude $\frac{5}{3}$ and $\frac{7}{5}$ in the hypotheses of these theorems.
These exceptional residues are detected by the algorithm automatically and are not part of the input.

\begin{theorem}\label{6 4a-2b I power-free}
Let $a, b$ be relatively prime positive integers such that $\frac{4}{3} < \ab < 2$ and $\ab \notin \{\frac{5}{3}, \frac{7}{5}\}$ and $\gcd(b,4) = 1$.
Then the $(4a-2b)$-uniform morphism
\[
	\varphi(n) = 0^{2 a-2 b-1} \, 1 \, 0^{-a+2 b-1} \, 1 \, 0^{3 a-4 b-1} \, 1 \, 0^{-a+2 b-1} \, 1 \, 0^{2 a-2 b-1} \, 1 \, 0^{-a+2 b-1} \, (n+1),
\]
with $6$ nonzero letters, locates words of length $5 a - 4 b$ and is $\ab$-power-free.
\end{theorem}

\begin{theorem}\label{6 4a-2b II power-free}
Let $a, b$ be relatively prime positive integers such that $\frac{4}{3} < \ab < 2$ and $\ab \notin \{\frac{5}{3}, \frac{7}{5}\}$ and $\gcd(b,4) = 1$.
Then the $(4a-2b)$-uniform morphism
\[
	\varphi(n) = 0^{-a+2 b-1} \, 1 \, 0^{2 a-2 b-1} \, 1 \, 0^{-a+2 b-1} \, 1 \, 0^{2 a-2 b-1} \, 1 \, 0^{-a+2 b-1} \, 1 \, 0^{3 a-4 b-1} \, (n+1),
\]
with $6$ nonzero letters, locates words of length $5 a - 4 b$ and is $\ab$-power-free.
\end{theorem}

The following morphism was obtained by specializing Conjecture~\ref{4r power-free} and is not directly related to any words $\word_{a/b}$ as far as we know.

\begin{theorem}\label{8 5a-3b power-free}
Let $a, b$ be relatively prime positive integers such that $\frac{5}{4} < \ab < \frac{4}{3}$ and $\ab \neq \frac{9}{7}$ and $\gcd(b,5) = 1$.
Then the $(5a-3b)$-uniform morphism
\begin{align*}
	\varphi(n) = {}
	& 0^{5 a-6 b-1} \, 1 \, 0^{2 a-2 b-1} \, 1 \, 0^{a-b-1} \, 1 \, 0^{-2 a+3 b-1} \, 1 \\
	& 0^{3 a-3 b-1} \, 1 \, 0^{-6 a+8 b-1} \, 1 \, 0^{a-b-1} \, 1 \, 0^{a-b-1} \, (n+1),
\end{align*}
with $8$ nonzero letters, locates words of length $6 a - 6 b$ and is $\ab$-power-free.
\end{theorem}

\begin{theorem}\label{10 11a-10b power-free}
Let $a, b$ be relatively prime positive integers such that $\frac{6}{5} < \ab < \frac{5}{4}$ and $\ab \notin \{\frac{11}{9}, \frac{16}{13}, \frac{17}{14}, \frac{23}{19}\}$ and $\gcd(b,11) = 1$.
Then the $(11a-10b)$-uniform morphism
\begin{align*}
	\varphi(n) = {}
	& 0^{-2 a+3 b-1} \, 1 \, 0^{3 a-3 b-1} \, 1 \, 0^{a-b-1} \, 1 \, 0^{-3 a+4 b-1} \, 1 \, 0^{5 a-6 b-1} \, 1 \\
	& 0^{a-b-1} \, 1 \, 0^{2 a-2 b-1} \, 1 \, 0^{a-b-1} \, 1 \, 0^{2 a-2 b-1} \, 1 \, 0^{a-b-1} \, (n+1),
\end{align*}
with $10$ nonzero letters, locates words of length $7 a - 7 b$ and is $\ab$-power-free.
\end{theorem}

\begin{theorem}\label{12 7a-5b power-free}
Let $a, b$ be relatively prime positive integers such that $\frac{7}{6} < \ab < \frac{6}{5}$ and $\ab \neq \frac{13}{11}$ and $\gcd(b,7) = 1$.
Then the $(7a-5b)$-uniform morphism
\begin{align*}
	\varphi(n) = {}
	& 0^{7 a-8 b-1} \, 1 \, 0^{-4 a+5 b-1} \, 1 \, 0^{6 a-7 b-1} \, 1 \, 0^{2 a-2 b-1} \, 1 \, 0^{a-b-1} \, 1 \, 0^{-4 a+5 b-1} \, 1 \\
	& 0^{2 a-2 b-1} \, 1 \, 0^{3 a-3 b-1} \, 1 \, 0^{-4 a+5 b-1} \, 1 \, 0^{-4 a+5 b-1} \, 1 \, 0^{a-b-1} \, 1 \, 0^{a-b-1} \, (n+1),
\end{align*}
with $12$ nonzero letters, locates words of length $5 a - 3 b$ and is $\ab$-power-free.
\end{theorem}

\begin{theorem}\label{13 7a-4b power-free}
Let $a, b$ be relatively prime positive integers such that $\frac{6}{5} < \ab < \frac{5}{4}$ and $\gcd(b,7) = 1$.
Then the $(7a-4b)$-uniform morphism
\begin{align*}
	\varphi(n) = {}
	& 0^{a-b-1} \, 1 \, 0^{-3 a+4 b-1} \, 1 \, 0^{4 a-4 b-1} \, 1 \, 0^{-3 a+4 b-1} \, 1 \, 0^{5 a-6 b-1} \, 1 \, 0^{a-b-1} \, 1 \, 0^{a-b-1} \, 1 \\
	& 0^{-3 a+4 b-1} \, 1 \, 0^{2 a-2 b-1} \, 2 \, 0^{2 a-2 b-1} \, 1 \, 0^{-3 a+4 b-1} \, 1 \, 0^{2 a-2 b-1} \, 1 \, 0^{a-b-1} \, (n+2),
\end{align*}
with $13$ nonzero letters, locates words of length $5 a - 5 b$ and is $\ab$-power-free.
\end{theorem}

\begin{theorem}\label{14 8a-4b power-free}
Let $a, b$ be relatively prime positive integers such that $\frac{9}{7} \leq \ab < \frac{4}{3}$ and $\gcd(b,8) = 1$.
Then the $(8a-4b)$-uniform morphism
\begin{align*}
	\varphi(n) = {}
	& 0^{2 a-2 b-1} \, 1 \, 0^{a-b-1} \, 1 \, 0^{-2 a+3 b-1} \, 1 \, 0^{a-b-1} \, 2 \, 0^{2 a-2 b-1} \, 1 \, 0^{-a+2 b-1} \, 1 \, 0^{a-b-1} \, 2 \\
	& 0^{a-b-1} \, 1 \, 0^{a-b-1} \, 1 \, 0^{b-1} \, 2 \, 0^{a-b-1} \, 1 \, 0^{a-b-1} \, 1 \, 0^{a-b-1} \, 1 \, 0^{-a+2 b-1} \, (n+2),
\end{align*}
with $14$ nonzero letters, locates words of length $a$ and is $\ab$-power-free.
\end{theorem}

\begin{theorem}\label{14 6a-b I power-free}
Let $a, b$ be relatively prime positive integers such that $\frac{11}{8} < \ab < \frac{3}{2}$ and $\ab \notin \{\frac{7}{5}, \frac{10}{7}\}$ and $\gcd(b,6) = 1$.
Then the $(6a-b)$-uniform morphism
\begin{align*}
	\varphi(n) = {}
	& 0^{a-b-1} \, 1 \, 0^{-6 a+9 b-1} \, 1 \, 0^{8 a-11 b-1} \, 1 \, 0^{-a+2 b-1} \, 1 \, 0^{2 a-2 b-1} \, 1 \\
	& 0^{-a+2 b-1} \, 1 \, 0^{-6 a+9 b-1} \, 1 \, 0^{a-b-1} \, 1 \, 0^{6 a-8 b-1} \, 1 \, 0^{-4 a+6 b-1} \, 1 \\
	& 0^{-a+2 b-1} \, 1 \, 0^{2 a-2 b-1} \, 1 \, 0^{-a+2 b-1} \, 1 \, 0^{6 a-8 b-1} \, (n+1),
\end{align*}
with $14$ nonzero letters, locates words of length $6 a - b$ and is $\ab$-power-free.
\end{theorem}

\begin{theorem}\label{14 6a-b II power-free}
Let $a, b$ be relatively prime positive integers such that $\frac{7}{5} < \ab < \frac{10}{7}$ and $\gcd(b,6) = 1$.
Then the $(6a-b)$-uniform morphism
\begin{align*}
	\varphi(n) = {}
	& 0^{a-b-1} \, 1 \, 0^{-23 a+33 b-1} \, 1 \, 0^{25 a-35 b-1} \, 1 \, 0^{-a+2 b-1} \, 1 \, 0^{2 a-2 b-1} \, 1 \\
	& 0^{-a+2 b-1} \, 1 \, 0^{-23 a+33 b-1} \, 1 \, 0^{a-b-1} \, 1 \, 0^{23 a-32 b-1} \, 1 \, 0^{-21 a+30 b-1} \, 1 \\
	& 0^{-a+2 b-1} \, 1 \, 0^{2 a-2 b-1} \, 1 \, 0^{-a+2 b-1} \, 1 \, 0^{23 a-32 b-1} \, (n+1),
\end{align*}
with $14$ nonzero letters, locates words of length $2 a$ and is $\ab$-power-free.
\end{theorem}

Theorems~\ref{14 6a-b I power-free} and \ref{14 6a-b II power-free} contain a new obstacle, which is that their morphisms do not satisfy the conditions of Lemma~\ref{locating length} on the given intervals.
Namely, for the morphism $\varphi$ in Theorem~\ref{14 6a-b I power-free}, if $\ab = \frac{17}{12}$ then the word $\varphi(0)$ is a square, and therefore $\varphi$ does not locate words of any length $\ell$.
This rational does not satisfy $\gcd(b,6) = 1$, so it is excluded by the hypotheses, but the algorithm for finding $\ell$ does not take this into account.
In practice, when the algorithm is unable to find a suitable $\ell$, we perform a separate search for obstructions, which reveals the square for $\ab = \frac{17}{12}$.
(This separate search could also be performed preemptively, but it is not exhaustive since we only check rationals with small denominators.)
Adding the assumption $\ab \neq \frac{17}{12}$ as input then effectively causes $\frac{17}{12}$ to become an interior endpoint, splitting the interval into the two subintervals $\frac{11}{8} < \ab < \frac{17}{12}$ and $\frac{17}{12} < \ab < \frac{3}{2}$.

Similarly, the morphism in Theorem~\ref{14 6a-b II power-free} produces a square $\varphi(0)$ for $\ab = \frac{65}{46}$; again this rational does not satisfy the $\gcd$ condition but must be taken into account to find $\ell$.
For some morphisms (namely, those in Theorems~\ref{26 9a-4b power-free}, \ref{42 13a-8b power-free}, \ref{54 13a-5b power-free}, \ref{158 53a-30b power-free}, and \ref{279 67a-30b power-free}) there exists rationals for which $\varphi(0)$ is a perfect power but that are not excluded by the $\gcd$ condition.
These rationals must be added to the hypotheses, and this is the second reason why there might be exceptional rationals in an interval.

\begin{theorem}\label{16 9a-7b power-free}
Let $a, b$ be relatively prime positive integers such that $\frac{9}{8} < \ab < \frac{8}{7}$ and $\gcd(b,9) = 1$.
Then the $(9a-7b)$-uniform morphism
\begin{align*}
	\varphi(n) = {}
	& 0^{9 a-10 b-1} \, 1 \, 0^{-6 a+7 b-1} \, 1 \, 0^{8 a-9 b-1} \, 1 \, 0^{-6 a+7 b-1} \, 1 \, 0^{8 a-9 b-1} \, 1 \, 0^{2 a-2 b-1} \, 1 \\
	& 0^{a-b-1} \, 1 \, 0^{-6 a+7 b-1} \, 1 \, 0^{2 a-2 b-1} \, 1 \, 0^{2 a-2 b-1} \, 1 \, 0^{3 a-3 b-1} \, 1 \, 0^{-6 a+7 b-1} \, 1 \\
	& 0^{2 a-2 b-1} \, 1 \, 0^{-6 a+7 b-1} \, 1 \, 0^{a-b-1} \, 1 \, 0^{a-b-1} \, (n+1),
\end{align*}
with $16$ nonzero letters, locates words of length $7 a - 7 b$ and is $\ab$-power-free.
\end{theorem}

\begin{theorem}\label{18 10a-8b power-free}
Let $a, b$ be relatively prime positive integers such that $\frac{10}{9} < \ab < \frac{9}{8}$ and $\ab \neq \frac{19}{17}$ and $\gcd(b,10) = 1$.
Then the $(10a-8b)$-uniform morphism
\begin{align*}
	\varphi(n) = {}
	& 0^{10 a-11 b-1} \, 1 \, 0^{-7 a+8 b-1} \, 1 \, 0^{9 a-10 b-1} \, 1 \, 0^{2 a-2 b-1} \, 1 \, 0^{a-b-1} \, 1 \, 0^{2 a-2 b-1} \, 1 \\
	& 0^{a-b-1} \, 1 \, 0^{-7 a+8 b-1} \, 1 \, 0^{2 a-2 b-1} \, 1 \, 0^{3 a-3 b-1} \, 1 \, 0^{3 a-3 b-1} \, 1 \, 0^{-7 a+8 b-1} \, 1 \\
	& 0^{-7 a+8 b-1} \, 1 \, 0^{a-b-1} \, 1 \, 0^{a-b-1} \, 1 \, 0^{10 a-11 b-1} \, 1 \, 0^{-8 a+9 b-1} \, 1 \, 0^{a-b-1} \, (n+1),
\end{align*}
with $18$ nonzero letters, locates words of length $6 a - 5 b$ and is $\ab$-power-free.
\end{theorem}

\begin{theorem}\label{20 11a-9b power-free}
Let $a, b$ be relatively prime positive integers such that $\frac{11}{10} < \ab < \frac{10}{9}$ and $\ab \neq \frac{21}{19}$ and $\gcd(b,11) = 1$.
Then the $(11a-9b)$-uniform morphism
\begin{align*}
	\varphi(n) = {}
	& 0^{11 a-12 b-1} \, 1 \, 0^{-8 a+9 b-1} \, 1 \, 0^{10 a-11 b-1} \, 1 \, 0^{-8 a+9 b-1} \, 1 \, 0^{10 a-11 b-1} \, 1 \\
	& 0^{-8 a+9 b-1} \, 1 \, 0^{10 a-11 b-1} \, 1 \, 0^{2 a-2 b-1} \, 1 \, 0^{a-b-1} \, 1 \, 0^{-8 a+9 b-1} \, 1 \\
	& 0^{2 a-2 b-1} \, 1 \, 0^{2 a-2 b-1} \, 1 \, 0^{2 a-2 b-1} \, 1 \, 0^{3 a-3 b-1} \, 1 \, 0^{-8 a+9 b-1} \, 1 \\
	& 0^{2 a-2 b-1} \, 1 \, 0^{2 a-2 b-1} \, 1 \, 0^{-8 a+9 b-1} \, 1 \, 0^{a-b-1} \, 1 \, 0^{a-b-1} \, (n+1),
\end{align*}
with $20$ nonzero letters, locates words of length $4 a - 3 b$ and is $\ab$-power-free.
\end{theorem}

\begin{theorem}\label{24 13a-11b power-free}
Let $a, b$ be relatively prime positive integers such that $\frac{13}{12} < \ab < \frac{12}{11}$ and $\ab \neq \frac{25}{23}$ and $\gcd(b,13) = 1$.
Then the $(13a-11b)$-uniform morphism
\begin{align*}
	\varphi(n) = {}
	& 0^{13 a-14 b-1} \, 1 \, 0^{-10 a+11 b-1} \, 1 \, 0^{12 a-13 b-1} \, 1 \, 0^{-10 a+11 b-1} \, 1 \, 0^{12 a-13 b-1} \, 1 \, 0^{-10 a+11 b-1} \, 1 \\
	& 0^{12 a-13 b-1} \, 1 \, 0^{-10 a+11 b-1} \, 1 \, 0^{12 a-13 b-1} \, 1 \, 0^{2 a-2 b-1} \, 1 \, 0^{a-b-1} \, 1 \, 0^{-10 a+11 b-1} \, 1 \\
	& 0^{2 a-2 b-1} \, 1 \, 0^{2 a-2 b-1} \, 1 \, 0^{2 a-2 b-1} \, 1 \, 0^{2 a-2 b-1} \, 1 \, 0^{3 a-3 b-1} \, 1 \, 0^{-10 a+11 b-1} \, 1 \\
	& 0^{2 a-2 b-1} \, 1 \, 0^{2 a-2 b-1} \, 1 \, 0^{2 a-2 b-1} \, 1 \, 0^{-10 a+11 b-1} \, 1 \, 0^{a-b-1} \, 1 \, 0^{a-b-1} \, (n+1),
\end{align*}
with $24$ nonzero letters, locates words of length $4 a - 3 b$ and is $\ab$-power-free.
\end{theorem}

\begin{theorem}\label{26 9a-4b power-free}
Let $a, b$ be relatively prime positive integers such that $\frac{7}{6} < \ab < \frac{6}{5}$ and $\ab \notin \{\frac{13}{11}, \frac{20}{17}\}$ and $\gcd(b,9) = 1$.
Then the $(9a-4b)$-uniform morphism
\begin{align*}
	\varphi(n) = {}
	& 0^{7 a-8 b-1} \, 1 \, 0^{2 a-2 b-1} \, 1 \, 0^{-10 a+12 b-1} \, 1 \, 0^{2 a-2 b-1} \, 1 \, 0^{7 a-8 b-1} \, 1 \, 0^{2 a-2 b-1} \, 1 \\
	& 0^{-10 a+12 b-1} \, 1 \, 0^{2 a-2 b-1} \, 1 \, 0^{7 a-8 b-1} \, 1 \, 0^{-5 a+6 b-1} \, 1 \, 0^{7 a-8 b-1} \, 1 \, 0^{-10 a+12 b-1} \, 1 \\
	& 0^{12 a-14 b-1} \, 1 \, 0^{-10 a+12 b-1} \, 1 \, 0^{2 a-2 b-1} \, 1 \, 0^{7 a-8 b-1} \, 1 \, 0^{2 a-2 b-1} \, 1 \, 0^{-10 a+12 b-1} \, 1 \\
	& 0^{2 a-2 b-1} \, 1 \, 0^{7 a-8 b-1} \, 1 \, 0^{2 a-2 b-1} \, 1 \, 0^{-10 a+12 b-1} \, 1 \, 0^{12 a-14 b-1} \, 1 \, 0^{-10 a+12 b-1} \, 1 \\
	& 0^{7 a-8 b-1} \, 1 \, 0^{-5 a+6 b-1} \, (n+1),
\end{align*}
with $26$ nonzero letters, locates words of length $7 a - 6 b$ and is $\ab$-power-free.
\end{theorem}

\begin{theorem}\label{29 14a-9b power-free}
Let $a, b$ be relatively prime positive integers such that $\frac{8}{7} < \ab < \frac{6}{5}$ and $\gcd(b,14) = 1$.
Then the $(14a-9b)$-uniform morphism
\begin{align*}
	\varphi(n) = {}
	& 0^{a-b-1} \, 1 \, 0^{a-b-1} \, 2 \, 0^{2 a-2 b-1} \, 1 \, 0^{2 a-2 b-1} \, 1 \, 0^{-3 a+4 b-1} \, 1 \, 0^{a-b-1} \, 2 \\
	& 0^{a-b-1} \, 1 \, 0^{2 a-2 b-1} \, 1 \, 0^{a-b-1} \, 1 \, 0^{a-b-1} \, 1 \, 0^{-3 a+4 b-1} \, 2 \, 0^{2 a-2 b-1} \, 2 \\
	& 0^{2 a-2 b-1} \, 1 \, 0^{a-b-1} \, 1 \, 0^{a-b-1} \, 1 \, 0^{-5 a+6 b-1} \, 1 \, 0^{7 a-8 b-1} \, 1 \, 0^{-5 a+6 b-1} \, 1 \\
	& 0^{a-b-1} \, 2 \, 0^{2 a-2 b-1} \, 2 \, 0^{a-b-1} \, 1 \, 0^{2 a-2 b-1} \, 1 \, 0^{-4 a+5 b-1} \, 1 \, 0^{3 a-3 b-1} \, 1 \\
	& 0^{2 a-2 b-1} \, 1 \, 0^{a-b-1} \, 1 \, 0^{-5 a+6 b-1} \, 1 \, 0^{a-b-1} \, 1 \, 0^{a-b-1} \, (n+2),
\end{align*}
with $29$ nonzero letters, locates words of length $6 a - 6 b$ and is $\ab$-power-free.
\end{theorem}

\begin{theorem}\label{30 10a-5b power-free}
Let $a, b$ be relatively prime positive integers such that $\frac{15}{13} < \ab < \frac{22}{19}$ and $\gcd(b,10) = 1$.
Then the $(10a-5b)$-uniform morphism
\begin{align*}
	\varphi(n) = {}
	& 0^{-11 a+13 b-1} \, 1 \, 0^{14 a-16 b-1} \, 1 \, 0^{-18 a+21 b-1} \, 1 \, 0^{20 a-23 b-1} \, 1 \, 0^{2 a-2 b-1} \, 1 \, 0^{a-b-1} \, 1 \\
	& 0^{-18 a+21 b-1} \, 1 \, 0^{14 a-16 b-1} \, 1 \, 0^{2 a-2 b-1} \, 1 \, 0^{3 a-3 b-1} \, 1 \, 0^{-18 a+21 b-1} \, 1 \, 0^{14 a-16 b-1} \, 1 \\
	& 0^{-18 a+21 b-1} \, 1 \, 0^{a-b-1} \, 1 \, 0^{a-b-1} \, 1 \, 0^{21 a-24 b-1} \, 1 \, 0^{-18 a+21 b-1} \, 1 \, 0^{14 a-16 b-1} \, 1 \\
	& 0^{-12 a+14 b-1} \, 1 \, 0^{2 a-2 b-1} \, 1 \, 0^{a-b-1} \, 1 \, 0^{14 a-16 b-1} \, 1 \, 0^{-18 a+21 b-1} \, 1 \, 0^{2 a-2 b-1} \, 1 \\
	& 0^{3 a-3 b-1} \, 1 \, 0^{14 a-16 b-1} \, 1 \, 0^{-18 a+21 b-1} \, 1 \, 0^{14 a-16 b-1} \, 1 \, 0^{a-b-1} \, 1 \, 0^{a-b-1} \, (n+1),
\end{align*}
with $30$ nonzero letters, locates words of length $a$ and is $\ab$-power-free.
\end{theorem}

\begin{theorem}\label{37 24a-15b power-free}
Let $a, b$ be relatively prime positive integers such that $\frac{9}{7} < \ab < \frac{4}{3}$ and $\gcd(b,24) = 1$.
Then the $(24a-15b)$-uniform morphism
\begin{align*}
	\varphi(n) = {}
	& 0^{a-b-1} \, 1 \, 0^{2 a-2 b-1} \, 1 \, 0^{-a+2 b-1} \, 1 \, 0^{2 a-2 b-1} \, 1 \, 0^{a-b-1} \, 1 \, 0^{-2 a+3 b-1} \, 1 \, 0^{4 a-5 b-1} \, 1 \\
	& 0^{-a+2 b-1} \, 1 \, 0^{2 a-2 b-1} \, 1 \, 0^{a-b-1} \, 1 \, 0^{-2 a+3 b-1} \, 1 \, 0^{-2 a+3 b-1} \, 1 \, 0^{5 a-6 b-1} \, 1 \\
	& 0^{-2 a+3 b-1} \, 1 \, 0^{4 a-5 b-1} \, 1 \, 0^{a-b-1} \, 1 \, 0^{-2 a+3 b-1} \, 1 \, 0^{3 a-3 b-1} \, 1 \, 0^{-2 a+3 b-1} \, 1 \\
	& 0^{a-b-1} \, 1 \, 0^{-3 a+4 b-1} \, 1 \, 0^{5 a-6 b-1} \, 1 \, 0^{2 a-2 b-1} \, 1 \, 0^{a-b-1} \, 1 \, 0^{-2 a+3 b-1} \, 1 \\
	& 0^{3 a-3 b-1} \, 1 \, 0^{-2 a+3 b-1} \, 1 \, 0^{4 a-5 b-1} \, 1 \, 0^{a-b-1} \, 1 \, 0^{-2 a+3 b-1} \, 1 \, 0^{2 a-2 b-1} \, 2 \\
	& 0^{a-b-1} \, 1 \, 0^{-2 a+3 b-1} \, 1 \, 0^{3 a-3 b-1} \, 1 \, 0^{-2 a+3 b-1} \, 1 \, 0^{a-b-1} \, 1 \, 0^{a-b-1} \, (n+2),
\end{align*}
with $37$ nonzero letters, locates words of length $2 a$ and is $\ab$-power-free.
\end{theorem}

\begin{theorem}\label{38 12a-7b power-free}
Let $a, b$ be relatively prime positive integers such that $\frac{10}{9} < \ab < \frac{19}{17}$ and $\gcd(b,12) = 1$.
Then the $(12a-7b)$-uniform morphism
\begin{align*}
	\varphi(n) = {}
	& 0^{19 a-21 b-1} \, 1 \, 0^{-16 a+18 b-1} \, 1 \, 0^{10 a-11 b-1} \, 1 \, 0^{-8 a+9 b-1} \, 1 \, 0^{10 a-11 b-1} \, 1 \, 0^{-8 a+9 b-1} \, 1 \, 0^{2 a-2 b-1} \, 1 \\
	& 0^{a-b-1} \, 1 \, 0^{10 a-11 b-1} \, 1 \, 0^{-16 a+18 b-1} \, 1 \, 0^{2 a-2 b-1} \, 1 \, 0^{2 a-2 b-1} \, 1 \, 0^{3 a-3 b-1} \, 1 \, 0^{10 a-11 b-1} \, 1 \\
	& 0^{-16 a+18 b-1} \, 1 \, 0^{2 a-2 b-1} \, 1 \, 0^{10 a-11 b-1} \, 1 \, 0^{a-b-1} \, 1 \, 0^{a-b-1} \, 1 \, 0^{-7 a+8 b-1} \, 1 \\
	& 0^{10 a-11 b-1} \, 1 \, 0^{-16 a+18 b-1} \, 1 \, 0^{18 a-20 b-1} \, 1 \, 0^{-16 a+18 b-1} \, 1 \, 0^{18 a-20 b-1} \, 1 \, 0^{2 a-2 b-1} \, 1 \\
	& 0^{a-b-1} \, 1 \, 0^{-16 a+18 b-1} \, 1 \, 0^{10 a-11 b-1} \, 1 \, 0^{2 a-2 b-1} \, 1 \, 0^{2 a-2 b-1} \, 1 \, 0^{3 a-3 b-1} \, 1 \\
	& 0^{-16 a+18 b-1} \, 1 \, 0^{10 a-11 b-1} \, 1 \, 0^{2 a-2 b-1} \, 1 \, 0^{-16 a+18 b-1} \, 1 \, 0^{a-b-1} \, 1 \, 0^{a-b-1} \, (n+1),
\end{align*}
with $38$ nonzero letters, locates words of length $a$ and is $\ab$-power-free.
\end{theorem}

\begin{theorem}\label{42 13a-8b power-free}
Let $a, b$ be relatively prime positive integers such that $\frac{11}{10} < \ab < \frac{21}{19}$ and $\ab \neq \frac{32}{29}$ and $\gcd(b,13) = 1$.
Then the $(13a-8b)$-uniform morphism
\begin{align*}
	\varphi(n) = {}
	& 0^{21 a-23 b-1} \, 1 \, 0^{-18 a+20 b-1} \, 1 \, 0^{11 a-12 b-1} \, 1 \, 0^{-9 a+10 b-1} \, 1 \, 0^{2 a-2 b-1} \, 1 \, 0^{a-b-1} \, 1 \\
	& 0^{2 a-2 b-1} \, 1 \, 0^{a-b-1} \, 1 \, 0^{11 a-12 b-1} \, 1 \, 0^{-18 a+20 b-1} \, 1 \, 0^{2 a-2 b-1} \, 1 \, 0^{3 a-3 b-1} \, 1 \\
	& 0^{3 a-3 b-1} \, 1 \, 0^{11 a-12 b-1} \, 1 \, 0^{-18 a+20 b-1} \, 1 \, 0^{11 a-12 b-1} \, 1 \, 0^{a-b-1} \, 1 \, 0^{a-b-1} \, 1 \\
	& 0^{-8 a+9 b-1} \, 1 \, 0^{10 a-11 b-1} \, 1 \, 0^{a-b-1} \, 1 \, 0^{-8 a+9 b-1} \, 1 \, 0^{11 a-12 b-1} \, 1 \, 0^{-18 a+20 b-1} \, 1 \\
	& 0^{20 a-22 b-1} \, 1 \, 0^{2 a-2 b-1} \, 1 \, 0^{a-b-1} \, 1 \, 0^{2 a-2 b-1} \, 1 \, 0^{a-b-1} \, 1 \, 0^{-18 a+20 b-1} \, 1 \\
	& 0^{11 a-12 b-1} \, 1 \, 0^{2 a-2 b-1} \, 1 \, 0^{3 a-3 b-1} \, 1 \, 0^{3 a-3 b-1} \, 1 \, 0^{-18 a+20 b-1} \, 1 \, 0^{11 a-12 b-1} \, 1 \\
	& 0^{-18 a+20 b-1} \, 1 \, 0^{a-b-1} \, 1 \, 0^{a-b-1} \, 1 \, 0^{21 a-23 b-1} \, 1 \, 0^{-19 a+21 b-1} \, 1 \, 0^{a-b-1} \, (n+1),
\end{align*}
with $42$ nonzero letters, locates words of length $a$ and is $\ab$-power-free.
\end{theorem}

\begin{theorem}\label{46 14a-9b I power-free}
Let $a, b$ be relatively prime positive integers such that $\frac{12}{11} < \ab < \frac{23}{21}$ and $\gcd(b,14) = 1$.
Then the $(14a-9b)$-uniform morphism
\begin{align*}
	\varphi(n) = {}
	& 0^{-9 a+10 b-1} \, 1 \, 0^{2 a-2 b-1} \, 1 \, 0^{2 a-2 b-1} \, 1 \, 0^{12 a-13 b-1} \, 1 \, 0^{2 a-2 b-1} \, 1 \, 0^{2 a-2 b-1} \, 1 \\
	& 0^{a-b-1} \, 1 \, 0^{-20 a+22 b-1} \, 1 \, 0^{2 a-2 b-1} \, 1 \, 0^{2 a-2 b-1} \, 1 \, 0^{12 a-13 b-1} \, 1 \, 0^{-20 a+22 b-1} \, 1 \\
	& 0^{a-b-1} \, 1 \, 0^{24 a-26 b-1} \, 1 \, 0^{-20 a+22 b-1} \, 1 \, 0^{22 a-24 b-1} \, 1 \, 0^{-20 a+22 b-1} \, 1 \, 0^{22 a-24 b-1} \, 1 \\
	& 0^{-20 a+22 b-1} \, 1 \, 0^{12 a-13 b-1} \, 1 \, 0^{-10 a+11 b-1} \, 1 \, 0^{-9 a+10 b-1} \, 1 \, 0^{a-b-1} \, 1 \, 0^{23 a-25 b-1} \, 1 \\
	& 0^{2 a-2 b-1} \, 1 \, 0^{2 a-2 b-1} \, 1 \, 0^{-20 a+22 b-1} \, 1 \, 0^{2 a-2 b-1} \, 1 \, 0^{2 a-2 b-1} \, 1 \, 0^{a-b-1} \, 1 \\
	& 0^{12 a-13 b-1} \, 1 \, 0^{2 a-2 b-1} \, 1 \, 0^{2 a-2 b-1} \, 1 \, 0^{-20 a+22 b-1} \, 1 \, 0^{12 a-13 b-1} \, 1 \, 0^{a-b-1} \, 1 \\
	& 0^{-8 a+9 b-1} \, 1 \, 0^{12 a-13 b-1} \, 1 \, 0^{-10 a+11 b-1} \, 1 \, 0^{12 a-13 b-1} \, 1 \, 0^{-10 a+11 b-1} \, 1 \, 0^{12 a-13 b-1} \, 1 \\
	& 0^{-20 a+22 b-1} \, 1 \, 0^{22 a-24 b-1} \, 1 \, 0^{-9 a+10 b-1} \, 1 \, 0^{a-b-1} \, (n+1),
\end{align*}
with $46$ nonzero letters, locates words of length $10 a - 10 b$ and is $\ab$-power-free.
\end{theorem}

\begin{theorem}\label{46 14a-9b II power-free}
Let $a, b$ be relatively prime positive integers such that $\frac{23}{21} < \ab < \frac{34}{31}$ and $\gcd(b,14) = 1$.
Then the $(14a-9b)$-uniform morphism
\begin{align*}
	\varphi(n) = {}
	& 0^{33 a-36 b-1} \, 1 \, 0^{-30 a+33 b-1} \, 1 \, 0^{22 a-24 b-1} \, 1 \, 0^{-20 a+22 b-1} \, 1 \, 0^{22 a-24 b-1} \, 1 \, 0^{-20 a+22 b-1} \, 1 \\
	& 0^{22 a-24 b-1} \, 1 \, 0^{-20 a+22 b-1} \, 1 \, 0^{2 a-2 b-1} \, 1 \, 0^{a-b-1} \, 1 \, 0^{22 a-24 b-1} \, 1 \, 0^{-30 a+33 b-1} \, 1 \\
	& 0^{2 a-2 b-1} \, 1 \, 0^{2 a-2 b-1} \, 1 \, 0^{2 a-2 b-1} \, 1 \, 0^{3 a-3 b-1} \, 1 \, 0^{22 a-24 b-1} \, 1 \, 0^{-30 a+33 b-1} \, 1 \\
	& 0^{2 a-2 b-1} \, 1 \, 0^{2 a-2 b-1} \, 1 \, 0^{22 a-24 b-1} \, 1 \, 0^{a-b-1} \, 1 \, 0^{a-b-1} \, 1 \, 0^{-19 a+21 b-1} \, 1 \\
	& 0^{22 a-24 b-1} \, 1 \, 0^{-30 a+33 b-1} \, 1 \, 0^{32 a-35 b-1} \, 1 \, 0^{-30 a+33 b-1} \, 1 \, 0^{32 a-35 b-1} \, 1 \, 0^{-30 a+33 b-1} \, 1 \\
	& 0^{32 a-35 b-1} \, 1 \, 0^{2 a-2 b-1} \, 1 \, 0^{a-b-1} \, 1 \, 0^{-30 a+33 b-1} \, 1 \, 0^{22 a-24 b-1} \, 1 \, 0^{2 a-2 b-1} \, 1 \\
	& 0^{2 a-2 b-1} \, 1 \, 0^{2 a-2 b-1} \, 1 \, 0^{3 a-3 b-1} \, 1 \, 0^{-30 a+33 b-1} \, 1 \, 0^{22 a-24 b-1} \, 1 \, 0^{2 a-2 b-1} \, 1 \\
	& 0^{2 a-2 b-1} \, 1 \, 0^{-30 a+33 b-1} \, 1 \, 0^{a-b-1} \, 1 \, 0^{a-b-1} \, (n+1),
\end{align*}
with $46$ nonzero letters, locates words of length $a$ and is $\ab$-power-free.
\end{theorem}

\begin{theorem}\label{54 13a-5b power-free}
Let $a, b$ be relatively prime positive integers such that $\frac{9}{8} < \ab < \frac{26}{23}$ and $\ab \neq \frac{35}{31}$ and $\gcd(b,13) = 1$.
Then the $(13a-5b)$-uniform morphism
\begin{align*}
	\varphi(n) = {}
	& 0^{25 a-28 b-1} \, 1 \, 0^{-22 a+25 b-1} \, 1 \, 0^{9 a-10 b-1} \, 1 \, 0^{9 a-10 b-1} \, 1 \, 0^{-7 a+8 b-1} \, 1 \, 0^{2 a-2 b-1} \, 1 \\
	& 0^{a-b-1} \, 1 \, 0^{9 a-10 b-1} \, 1 \, 0^{-22 a+25 b-1} \, 1 \, 0^{9 a-10 b-1} \, 1 \, 0^{2 a-2 b-1} \, 1 \, 0^{3 a-3 b-1} \, 1 \\
	& 0^{9 a-10 b-1} \, 1 \, 0^{-22 a+25 b-1} \, 1 \, 0^{9 a-10 b-1} \, 1 \, 0^{9 a-10 b-1} \, 1 \, 0^{a-b-1} \, 1 \, 0^{a-b-1} \, 1 \\
	& 0^{-6 a+7 b-1} \, 1 \, 0^{9 a-10 b-1} \, 1 \, 0^{-22 a+25 b-1} \, 1 \, 0^{9 a-10 b-1} \, 1 \, 0^{-7 a+8 b-1} \, 1 \, 0^{2 a-2 b-1} \, 1 \\
	& 0^{a-b-1} \, 1 \, 0^{9 a-10 b-1} \, 1 \, 0^{9 a-10 b-1} \, 1 \, 0^{-22 a+25 b-1} \, 1 \, 0^{2 a-2 b-1} \, 1 \, 0^{3 a-3 b-1} \, 1 \\
	& 0^{9 a-10 b-1} \, 1 \, 0^{9 a-10 b-1} \, 1 \, 0^{-22 a+25 b-1} \, 1 \, 0^{9 a-10 b-1} \, 1 \, 0^{a-b-1} \, 1 \, 0^{a-b-1} \, 1 \\
	& 0^{-6 a+7 b-1} \, 1 \, 0^{9 a-10 b-1} \, 1 \, 0^{9 a-10 b-1} \, 1 \, 0^{-22 a+25 b-1} \, 1 \, 0^{24 a-27 b-1} \, 1 \, 0^{2 a-2 b-1} \, 1 \\
	& 0^{a-b-1} \, 1 \, 0^{-22 a+25 b-1} \, 1 \, 0^{9 a-10 b-1} \, 1 \, 0^{9 a-10 b-1} \, 1 \, 0^{2 a-2 b-1} \, 1 \, 0^{3 a-3 b-1} \, 1 \\
	& 0^{-22 a+25 b-1} \, 1 \, 0^{9 a-10 b-1} \, 1 \, 0^{9 a-10 b-1} \, 1 \, 0^{-22 a+25 b-1} \, 1 \, 0^{a-b-1} \, 1 \, 0^{a-b-1} \, (n+1),
\end{align*}
with $54$ nonzero letters, locates words of length $a$ and is $\ab$-power-free.
\end{theorem}

\begin{theorem}\label{102 38a-15b power-free}
Let $a, b$ be relatively prime positive integers such that $\frac{11}{9} < \ab < \frac{16}{13}$ and $\gcd(b,38) = 1$.
Then the $(38a-15b)$-uniform morphism
\small
\begin{align*}
	\varphi(n) = {}
	& 0^{-7 a+9 b-1} \, 1 \, 0^{10 a-12 b-1} \, 1 \, 0^{2 a-2 b-1} \, 1 \, 0^{a-b-1} \, 1 \, 0^{-12 a+15 b-1} \, 1 \, 0^{10 a-12 b-1} \, 1 \\
	& 0^{3 a-3 b-1} \, 1 \, 0^{-12 a+15 b-1} \, 1 \, 0^{10 a-12 b-1} \, 1 \, 0^{-8 a+10 b-1} \, 1 \, 0^{a-b-1} \, 1 \, 0^{10 a-12 b-1} \, 1 \\
	& 0^{-12 a+15 b-1} \, 1 \, 0^{2 a-2 b-1} \, 2 \, 0^{a-b-1} \, 1 \, 0^{10 a-12 b-1} \, 1 \, 0^{-12 a+15 b-1} \, 1 \, 0^{3 a-3 b-1} \, 1 \\
	& 0^{10 a-12 b-1} \, 1 \, 0^{-12 a+15 b-1} \, 1 \, 0^{a-b-1} \, 1 \, 0^{a-b-1} \, 2 \, 0^{a-b-1} \, 1 \, 0^{10 a-12 b-1} \, 1 \\
	& 0^{2 a-2 b-1} \, 1 \, 0^{-11 a+14 b-1} \, 1 \, 0^{10 a-12 b-1} \, 1 \, 0^{2 a-2 b-1} \, 1 \, 0^{a-b-1} \, 1 \, 0^{-12 a+15 b-1} \, 1 \\
	& 0^{10 a-12 b-1} \, 1 \, 0^{-8 a+10 b-1} \, 1 \, 0^{11 a-13 b-1} \, 1 \, 0^{-12 a+15 b-1} \, 1 \, 0^{2 a-2 b-1} \, 1 \, 0^{a-b-1} \, 1 \\
	& 0^{10 a-12 b-1} \, 1 \, 0^{-12 a+15 b-1} \, 1 \, 0^{10 a-12 b-1} \, 1 \, 0^{-7 a+9 b-1} \, 1 \, 0^{10 a-12 b-1} \, 1 \, 0^{-12 a+15 b-1} \, 1 \\
	& 0^{14 a-17 b-1} \, 1 \, 0^{a-b-1} \, 1 \, 0^{-12 a+15 b-1} \, 1 \, 0^{10 a-12 b-1} \, 1 \, 0^{3 a-3 b-1} \, 1 \, 0^{-12 a+15 b-1} \, 1 \\
	& 0^{10 a-12 b-1} \, 1 \, 0^{a-b-1} \, 1 \, 0^{-13 a+16 b-1} \, 1 \, 0^{15 a-18 b-1} \, 1 \, 0^{-12 a+15 b-1} \, 1 \, 0^{2 a-2 b-1} \, 1 \\
	& 0^{a-b-1} \, 1 \, 0^{10 a-12 b-1} \, 1 \, 0^{-12 a+15 b-1} \, 1 \, 0^{3 a-3 b-1} \, 1 \, 0^{10 a-12 b-1} \, 1 \, 0^{-12 a+15 b-1} \, 1 \\
	& 0^{14 a-17 b-1} \, 1 \, 0^{a-b-1} \, 1 \, 0^{-12 a+15 b-1} \, 1 \, 0^{10 a-12 b-1} \, 1 \, 0^{2 a-2 b-1} \, 2 \, 0^{a-b-1} \, 1 \\
	& 0^{-12 a+15 b-1} \, 1 \, 0^{10 a-12 b-1} \, 1 \, 0^{3 a-3 b-1} \, 1 \, 0^{-12 a+15 b-1} \, 1 \, 0^{10 a-12 b-1} \, 1 \, 0^{a-b-1} \, 1 \\
	& 0^{a-b-1} \, 2 \, 0^{a-b-1} \, 1 \, 0^{-12 a+15 b-1} \, 1 \, 0^{2 a-2 b-1} \, 1 \, 0^{11 a-13 b-1} \, 1 \, 0^{-12 a+15 b-1} \, 1 \\
	& 0^{2 a-2 b-1} \, 1 \, 0^{a-b-1} \, 1 \, 0^{10 a-12 b-1} \, 1 \, 0^{-12 a+15 b-1} \, 1 \, 0^{14 a-17 b-1} \, 1 \, 0^{-11 a+14 b-1} \, 1 \\
	& 0^{10 a-12 b-1} \, 1 \, 0^{2 a-2 b-1} \, 1 \, 0^{a-b-1} \, 1 \, 0^{-12 a+15 b-1} \, 1 \, 0^{10 a-12 b-1} \, 1 \, 0^{-12 a+15 b-1} \, 1 \\
	& 0^{15 a-18 b-1} \, 1 \, 0^{-12 a+15 b-1} \, 1 \, 0^{10 a-12 b-1} \, 1 \, 0^{-8 a+10 b-1} \, 1 \, 0^{a-b-1} \, 1 \, 0^{10 a-12 b-1} \, 1 \\
	& 0^{-12 a+15 b-1} \, 1 \, 0^{3 a-3 b-1} \, 1 \, 0^{10 a-12 b-1} \, 1 \, 0^{-12 a+15 b-1} \, 1 \, 0^{a-b-1} \, 1 \, 0^{9 a-11 b-1} \, (n+1),
\end{align*}
\normalsize
with $102$ nonzero letters, locates words of length $2 a$ and is $\ab$-power-free.
\end{theorem}

%Prevent vertical space in the previous theorem.
\newpage

\begin{theorem}\label{158 53a-30b power-free}
Let $a, b$ be relatively prime positive integers such that $\frac{9}{8} < \ab < \frac{17}{15}$ and $\ab \neq \frac{26}{23}$ and $\gcd(b,53) = 1$.
Then the $(53a-30b)$-uniform morphism
\small
\begin{align*}
	\varphi(n) = {}
	& 0^{a-b-1} \, 1 \, 0^{9 a-10 b-1} \, 1 \, 0^{-7 a+8 b-1} \, 1 \, 0^{a-b-1} \, 1 \, 0^{9 a-10 b-1} \, 1 \, 0^{2 a-2 b-1} \, 1 \\
	& 0^{-13 a+15 b-1} \, 1 \, 0^{2 a-2 b-1} \, 2 \, 0^{a-b-1} \, 1 \, 0^{9 a-10 b-1} \, 1 \, 0^{2 a-2 b-1} \, 1 \, 0^{a-b-1} \, 1 \\
	& 0^{-14 a+16 b-1} \, 1 \, 0^{3 a-3 b-1} \, 1 \, 0^{9 a-10 b-1} \, 1 \, 0^{-7 a+8 b-1} \, 1 \, 0^{10 a-11 b-1} \, 1 \, 0^{-14 a+16 b-1} \, 1 \\
	& 0^{a-b-1} \, 1 \, 0^{a-b-1} \, 2 \, 0^{a-b-1} \, 1 \, 0^{2 a-2 b-1} \, 1 \, 0^{a-b-1} \, 1 \, 0^{9 a-10 b-1} \, 1 \\
	& 0^{2 a-2 b-1} \, 1 \, 0^{-13 a+15 b-1} \, 1 \, 0^{9 a-10 b-1} \, 1 \, 0^{-6 a+7 b-1} \, 1 \, 0^{9 a-10 b-1} \, 1 \, 0^{2 a-2 b-1} \, 1 \\
	& 0^{a-b-1} \, 1 \, 0^{-14 a+16 b-1} \, 1 \, 0^{16 a-18 b-1} \, 1 \, 0^{a-b-1} \, 1 \, 0^{-14 a+16 b-1} \, 1 \, 0^{9 a-10 b-1} \, 1 \\
	& 0^{-7 a+8 b-1} \, 1 \, 0^{10 a-11 b-1} \, 1 \, 0^{3 a-3 b-1} \, 1 \, 0^{-14 a+16 b-1} \, 1 \, 0^{2 a-2 b-1} \, 1 \, 0^{a-b-1} \, 1 \\
	& 0^{9 a-10 b-1} \, 1 \, 0^{a-b-1} \, 1 \, 0^{-15 a+17 b-1} \, 1 \, 0^{17 a-19 b-1} \, 1 \, 0^{-14 a+16 b-1} \, 1 \, 0^{9 a-10 b-1} \, 1 \\
	& 0^{-6 a+7 b-1} \, 1 \, 0^{2 a-2 b-1} \, 1 \, 0^{a-b-1} \, 1 \, 0^{9 a-10 b-1} \, 1 \, 0^{-14 a+16 b-1} \, 1 \, 0^{16 a-18 b-1} \, 1 \\
	& 0^{a-b-1} \, 1 \, 0^{-14 a+16 b-1} \, 1 \, 0^{3 a-3 b-1} \, 1 \, 0^{9 a-10 b-1} \, 1 \, 0^{3 a-3 b-1} \, 1 \, 0^{-14 a+16 b-1} \, 1 \\
	& 0^{16 a-18 b-1} \, 1 \, 0^{a-b-1} \, 1 \, 0^{-14 a+16 b-1} \, 1 \, 0^{9 a-10 b-1} \, 1 \, 0^{a-b-1} \, 1 \, 0^{-15 a+17 b-1} \, 1 \\
	& 0^{17 a-19 b-1} \, 1 \, 0^{2 a-2 b-1} \, 2 \, 0^{a-b-1} \, 1 \, 0^{-14 a+16 b-1} \, 1 \, 0^{2 a-2 b-1} \, 1 \, 0^{a-b-1} \, 1 \\
	& 0^{9 a-10 b-1} \, 1 \, 0^{3 a-3 b-1} \, 1 \, 0^{-14 a+16 b-1} \, 1 \, 0^{3 a-3 b-1} \, 1 \, 0^{9 a-10 b-1} \, 1 \, 0^{a-b-1} \, 1 \\
	& 0^{a-b-1} \, 2 \, 0^{a-b-1} \, 1 \, 0^{-14 a+16 b-1} \, 1 \, 0^{16 a-18 b-1} \, 1 \, 0^{a-b-1} \, 1 \, 0^{-14 a+16 b-1} \, 1 \\
	& 0^{2 a-2 b-1} \, 1 \, 0^{10 a-11 b-1} \, 1 \, 0^{2 a-2 b-1} \, 2 \, 0^{a-b-1} \, 1 \, 0^{-14 a+16 b-1} \, 1 \, 0^{2 a-2 b-1} \, 1 \\
	& 0^{a-b-1} \, 1 \, 0^{9 a-10 b-1} \, 1 \, 0^{3 a-3 b-1} \, 1 \, 0^{-14 a+16 b-1} \, 1 \, 0^{16 a-18 b-1} \, 1 \, 0^{-13 a+15 b-1} \, 1 \\
	& 0^{9 a-10 b-1} \, 1 \, 0^{a-b-1} \, 1 \, 0^{a-b-1} \, 2 \, 0^{a-b-1} \, 1 \, 0^{2 a-2 b-1} \, 1 \, 0^{a-b-1} \, 1 \\
	& 0^{-14 a+16 b-1} \, 1 \, 0^{2 a-2 b-1} \, 1 \, 0^{10 a-11 b-1} \, 1 \, 0^{-14 a+16 b-1} \, 1 \, 0^{17 a-19 b-1} \, 1 \, 0^{-14 a+16 b-1} \, 1 \\
	& 0^{2 a-2 b-1} \, 1 \, 0^{a-b-1} \, 1 \, 0^{9 a-10 b-1} \, 1 \, 0^{-7 a+8 b-1} \, 1 \, 0^{a-b-1} \, 1 \, 0^{9 a-10 b-1} \, 1 \\
	& 0^{-14 a+16 b-1} \, 1 \, 0^{16 a-18 b-1} \, 1 \, 0^{-13 a+15 b-1} \, 1 \, 0^{3 a-3 b-1} \, 1 \, 0^{9 a-10 b-1} \, 1 \, 0^{2 a-2 b-1} \, 1 \\
	& 0^{a-b-1} \, 1 \, 0^{-14 a+16 b-1} \, 1 \, 0^{a-b-1} \, 1 \, 0^{8 a-9 b-1} \, 1 \, 0^{-6 a+7 b-1} \, 1 \, 0^{9 a-10 b-1} \, 1 \\
	& 0^{-14 a+16 b-1} \, 1 \, 0^{17 a-19 b-1} \, 1 \, 0^{2 a-2 b-1} \, 1 \, 0^{a-b-1} \, 1 \, 0^{-14 a+16 b-1} \, 1 \, 0^{9 a-10 b-1} \, 1 \\
	& 0^{-7 a+8 b-1} \, 1 \, 0^{a-b-1} \, 1 \, 0^{9 a-10 b-1} \, 1 \, 0^{3 a-3 b-1} \, 1 \, 0^{-14 a+16 b-1} \, 1 \, 0^{3 a-3 b-1} \, 1 \\
	& 0^{9 a-10 b-1} \, 1 \, 0^{-7 a+8 b-1} \, 1 \, 0^{a-b-1} \, 1 \, 0^{9 a-10 b-1} \, 1 \, 0^{-14 a+16 b-1} \, 1 \, 0^{a-b-1} \, 1 \\
	& 0^{8 a-9 b-1} \, 1 \, 0^{-6 a+7 b-1} \, 1 \, 0^{2 a-2 b-1} \, 2 \, 0^{a-b-1} \, 1 \, 0^{9 a-10 b-1} \, 1 \, 0^{2 a-2 b-1} \, 1 \\
	& 0^{a-b-1} \, 1 \, 0^{-14 a+16 b-1} \, 1 \, 0^{3 a-3 b-1} \, 1 \, 0^{9 a-10 b-1} \, 1 \, 0^{3 a-3 b-1} \, 1 \, 0^{-14 a+16 b-1} \, 1 \\
	& 0^{a-b-1} \, 1 \, 0^{a-b-1} \, (n+2),
\end{align*}
\normalsize
with $158$ nonzero letters, locates words of length $8 a - 7 b$ and is $\ab$-power-free.
\end{theorem}

%Prevent vertical space in the previous theorem.
\newpage

\begin{theorem}\label{191 66a-28b power-free}
Let $a, b$ be relatively prime positive integers such that $\frac{7}{6} < \ab < \frac{13}{11}$ and $\ab \neq \frac{20}{17}$ and $\gcd(b,66) = 1$.
Then the $(66a-28b)$-uniform morphism
\small
\begin{align*}
	\varphi(n) = {}
	& 0^{2 a-2 b-1} \, 1 \, 0^{7 a-8 b-1} \, 1 \, 0^{-10 a+12 b-1} \, 1 \, 0^{7 a-8 b-1} \, 1 \, 0^{a-b-1} \, 1 \, 0^{-4 a+5 b-1} \, 1 \\
	& 0^{7 a-8 b-1} \, 1 \, 0^{-10 a+12 b-1} \, 1 \, 0^{8 a-9 b-1} \, 1 \, 0^{-4 a+5 b-1} \, 1 \, 0^{7 a-8 b-1} \, 1 \, 0^{-10 a+12 b-1} \, 1 \\
	& 0^{a-b-1} \, 1 \, 0^{7 a-8 b-1} \, 1 \, 0^{-5 a+6 b-1} \, 1 \, 0^{6 a-7 b-1} \, 1 \, 0^{2 a-2 b-1} \, 1 \, 0^{-9 a+11 b-1} \, 1 \\
	& 0^{8 a-9 b-1} \, 1 \, 0^{-5 a+6 b-1} \, 1 \, 0^{7 a-8 b-1} \, 1 \, 0^{a-b-1} \, 1 \, 0^{-10 a+12 b-1} \, 1 \, 0^{a-b-1} \, 1 \\
	& 0^{2 a-2 b-1} \, 1 \, 0^{8 a-9 b-1} \, 1 \, 0^{-9 a+11 b-1} \, 1 \, 0^{2 a-2 b-1} \, 1 \, 0^{a-b-1} \, 1 \, 0^{7 a-8 b-1} \, 1 \\
	& 0^{a-b-1} \, 1 \, 0^{-10 a+12 b-1} \, 1 \, 0^{12 a-14 b-1} \, 1 \, 0^{-9 a+11 b-1} \, 1 \, 0^{8 a-9 b-1} \, 1 \, 0^{2 a-2 b-1} \, 1 \\
	& 0^{a-b-1} \, 1 \, 0^{-10 a+12 b-1} \, 1 \, 0^{a-b-1} \, 1 \, 0^{7 a-8 b-1} \, 1 \, 0^{-10 a+12 b-1} \, 1 \, 0^{13 a-15 b-1} \, 1 \\
	& 0^{-9 a+11 b-1} \, 1 \, 0^{7 a-8 b-1} \, 1 \, 0^{-5 a+6 b-1} \, 1 \, 0^{a-b-1} \, 1 \, 0^{7 a-8 b-1} \, 1 \, 0^{a-b-1} \, 1 \\
	& 0^{-10 a+12 b-1} \, 1 \, 0^{3 a-3 b-1} \, 1 \, 0^{8 a-9 b-1} \, 1 \, 0^{-10 a+12 b-1} \, 1 \, 0^{a-b-1} \, 1 \, 0^{6 a-7 b-1} \, 1 \\
	& 0^{-4 a+5 b-1} \, 1 \, 0^{a-b-1} \, 1 \, 0^{7 a-8 b-1} \, 1 \, 0^{2 a-2 b-1} \, 1 \, 0^{a-b-1} \, 1 \, 0^{-9 a+11 b-1} \, 1 \\
	& 0^{7 a-8 b-1} \, 1 \, 0^{3 a-3 b-1} \, 1 \, 0^{a-b-1} \, 1 \, 0^{-10 a+12 b-1} \, 1 \, 0^{7 a-8 b-1} \, 1 \, 0^{-5 a+6 b-1} \, 1 \\
	& 0^{a-b-1} \, 1 \, 0^{8 a-9 b-1} \, 1 \, 0^{-10 a+12 b-1} \, 1 \, 0^{2 a-2 b-1} \, 2 \, 0^{a-b-1} \, 1 \, 0^{a-b-1} \, 1 \\
	& 0^{7 a-8 b-1} \, 1 \, 0^{-10 a+12 b-1} \, 1 \, 0^{3 a-3 b-1} \, 2 \, 0^{a-b-1} \, 1 \, 0^{7 a-8 b-1} \, 1 \, 0^{-10 a+12 b-1} \, 1 \\
	& 0^{a-b-1} \, 1 \, 0^{3 a-3 b-1} \, 1 \, 0^{7 a-8 b-1} \, 1 \, 0^{2 a-2 b-1} \, 1 \, 0^{a-b-1} \, 1 \, 0^{a-b-1} \, 1 \\
	& 0^{-10 a+12 b-1} \, 1 \, 0^{7 a-8 b-1} \, 1 \, 0^{-9 a+11 b-1} \, 1 \, 0^{13 a-15 b-1} \, 1 \, 0^{-10 a+12 b-1} \, 1 \, 0^{7 a-8 b-1} \, 1 \\
	& 0^{-5 a+6 b-1} \, 1 \, 0^{-4 a+5 b-1} \, 1 \, 0^{13 a-15 b-1} \, 1 \, 0^{-10 a+12 b-1} \, 1 \, 0^{2 a-2 b-1} \, 1 \, 0^{-4 a+5 b-1} \, 1 \\
	& 0^{13 a-15 b-1} \, 1 \, 0^{-10 a+12 b-1} \, 1 \, 0^{7 a-8 b-1} \, 1 \, 0^{a-b-1} \, 1 \, 0^{-10 a+12 b-1} \, 1 \, 0^{12 a-14 b-1} \, 1 \\
	& 0^{-11 a+13 b-1} \, 1 \, 0^{13 a-15 b-1} \, 1 \, 0^{-3 a+4 b-1} \, 1 \, 0^{-9 a+11 b-1} \, 1 \, 0^{12 a-14 b-1} \, 1 \, 0^{-11 a+13 b-1} \, 1 \\
	& 0^{a-b-1} \, 1 \, 0^{a-b-1} \, 1 \, 0^{7 a-8 b-1} \, 1 \, 0^{a-b-1} \, 1 \, 0^{3 a-3 b-1} \, 1 \, 0^{-10 a+12 b-1} \, 1 \\
	& 0^{8 a-9 b-1} \, 1 \, 0^{a-b-1} \, 1 \, 0^{2 a-2 b-1} \, 1 \, 0^{-10 a+12 b-1} \, 1 \, 0^{a-b-1} \, 1 \, 0^{2 a-2 b-1} \, 1 \\
	& 0^{a-b-1} \, 1 \, 0^{7 a-8 b-1} \, 1 \, 0^{-9 a+11 b-1} \, 1 \, 0^{3 a-3 b-1} \, 1 \, 0^{7 a-8 b-1} \, 1 \, 0^{a-b-1} \, 1 \\
	& 0^{-10 a+12 b-1} \, 1 \, 0^{12 a-14 b-1} \, 1 \, 0^{a-b-1} \, 1 \, 0^{-10 a+12 b-1} \, 1 \, 0^{8 a-9 b-1} \, 1 \, 0^{2 a-2 b-1} \, 2 \\
	& 0^{a-b-1} \, 1 \, 0^{-10 a+12 b-1} \, 1 \, 0^{a-b-1} \, 1 \, 0^{7 a-8 b-1} \, 1 \, 0^{3 a-3 b-1} \, 1 \, 0^{-9 a+11 b-1} \, 1 \\
	& 0^{7 a-8 b-1} \, 1 \, 0^{a-b-1} \, 1 \, 0^{a-b-1} \, 2 \, 0^{a-b-1} \, 1 \, 0^{a-b-1} \, 1 \, 0^{-10 a+12 b-1} \, 1 \\
	& 0^{2 a-2 b-1} \, 1 \, 0^{8 a-9 b-1} \, 2 \, 0^{a-b-1} \, 1 \, 0^{-10 a+12 b-1} \, 1 \, 0^{2 a-2 b-1} \, 1 \, 0^{a-b-1} \, 1 \\
	& 0^{8 a-9 b-1} \, 1 \, 0^{-10 a+12 b-1} \, 1 \, 0^{12 a-14 b-1} \, 1 \, 0^{-9 a+11 b-1} \, 1 \, 0^{a-b-1} \, 1 \, 0^{7 a-8 b-1} \, 1 \\
	& 0^{2 a-2 b-1} \, 1 \, 0^{a-b-1} \, 1 \, 0^{-9 a+11 b-1} \, 1 \, 0^{7 a-8 b-1} \, 1 \, 0^{-10 a+12 b-1} \, 1 \, 0^{13 a-15 b-1} \, 1 \\
	& 0^{a-b-1} \, 1 \, 0^{-10 a+12 b-1} \, 1 \, 0^{7 a-8 b-1} \, 1 \, 0^{-5 a+6 b-1} \, 1 \, 0^{a-b-1} \, 1 \, 0^{8 a-9 b-1} \, 1 \\
	& 0^{-10 a+12 b-1} \, 1 \, 0^{3 a-3 b-1} \, 1 \, 0^{a-b-1} \, 1 \, 0^{7 a-8 b-1} \, 1 \, 0^{-10 a+12 b-1} \, 1 \, 0^{a-b-1} \, 1 \\
	& 0^{6 a-7 b-1} \, 1 \, 0^{-3 a+4 b-1} \, 1 \, 0^{7 a-8 b-1} \, 1 \, 0^{2 a-2 b-1} \, 1 \, 0^{a-b-1} \, 1 \, 0^{a-b-1} \, 1 \\
	& 0^{-10 a+12 b-1} \, 1 \, 0^{7 a-8 b-1} \, 1 \, 0^{-3 a+4 b-1} \, 1 \, 0^{7 a-8 b-1} \, 1 \, 0^{-10 a+12 b-1} \, 1 \, 0^{7 a-8 b-1} \, 1 \\
	& 0^{-5 a+6 b-1} \, 1 \, 0^{2 a-2 b-1} \, 1 \, 0^{7 a-8 b-1} \, 1 \, 0^{-10 a+12 b-1} \, 1 \, 0^{2 a-2 b-1} \, (n+2),
\end{align*}
\normalsize
with $191$ nonzero letters, locates words of length $2 a$ and is $\ab$-power-free.
\end{theorem}

\begin{theorem}\label{279 67a-30b power-free}
Let $a, b$ be relatively prime positive integers such that $\frac{10}{9} < \ab < \frac{29}{26}$ and $\ab \neq \frac{39}{35}$ and $\gcd(b,67) = 1$.
Then the $(67a-30b)$-uniform morphism
\tiny
\begin{align*}
	\varphi(n) = {}
	& 0^{-7 a+8 b-1} \, 1 \, 0^{10 a-11 b-1} \, 1 \, 0^{10 a-11 b-1} \, 1 \, 0^{a-b-1} \, 1 \, 0^{-26 a+29 b-1} \, 1 \, 0^{28 a-31 b-1} \, 1 \, 0^{2 a-2 b-1} \, 1 \\
	& 0^{a-b-1} \, 1 \, 0^{-25 a+28 b-1} \, 1 \, 0^{10 a-11 b-1} \, 1 \, 0^{2 a-2 b-1} \, 1 \, 0^{a-b-1} \, 1 \, 0^{10 a-11 b-1} \, 1 \, 0^{3 a-3 b-1} \, 1 \\
	& 0^{-25 a+28 b-1} \, 1 \, 0^{10 a-11 b-1} \, 1 \, 0^{3 a-3 b-1} \, 1 \, 0^{10 a-11 b-1} \, 1 \, 0^{-8 a+9 b-1} \, 1 \, 0^{a-b-1} \, 1 \, 0^{10 a-11 b-1} \, 1 \\
	& 0^{-25 a+28 b-1} \, 1 \, 0^{10 a-11 b-1} \, 1 \, 0^{-8 a+9 b-1} \, 1 \, 0^{a-b-1} \, 1 \, 0^{10 a-11 b-1} \, 1 \, 0^{2 a-2 b-1} \, 2 \, 0^{a-b-1} \, 1 \\
	& 0^{10 a-11 b-1} \, 1 \, 0^{-25 a+28 b-1} \, 1 \, 0^{2 a-2 b-1} \, 2 \, 0^{a-b-1} \, 1 \, 0^{10 a-11 b-1} \, 1 \, 0^{3 a-3 b-1} \, 1 \, 0^{10 a-11 b-1} \, 1 \\
	& 0^{-25 a+28 b-1} \, 1 \, 0^{3 a-3 b-1} \, 1 \, 0^{10 a-11 b-1} \, 1 \, 0^{a-b-1} \, 1 \, 0^{a-b-1} \, 2 \, 0^{a-b-1} \, 1 \, 0^{10 a-11 b-1} \, 1 \\
	& 0^{-25 a+28 b-1} \, 1 \, 0^{a-b-1} \, 1 \, 0^{a-b-1} \, 2 \, 0^{a-b-1} \, 1 \, 0^{2 a-2 b-1} \, 1 \, 0^{11 a-12 b-1} \, 1 \, 0^{10 a-11 b-1} \, 1 \\
	& 0^{2 a-2 b-1} \, 1 \, 0^{-24 a+27 b-1} \, 1 \, 0^{2 a-2 b-1} \, 1 \, 0^{a-b-1} \, 1 \, 0^{10 a-11 b-1} \, 1 \, 0^{10 a-11 b-1} \, 1 \, 0^{2 a-2 b-1} \, 1 \\
	& 0^{a-b-1} \, 1 \, 0^{-25 a+28 b-1} \, 1 \, 0^{27 a-30 b-1} \, 1 \, 0^{-24 a+27 b-1} \, 1 \, 0^{10 a-11 b-1} \, 1 \, 0^{10 a-11 b-1} \, 1 \, 0^{-8 a+9 b-1} \, 1 \\
	& 0^{11 a-12 b-1} \, 1 \, 0^{2 a-2 b-1} \, 1 \, 0^{a-b-1} \, 1 \, 0^{-25 a+28 b-1} \, 1 \, 0^{10 a-11 b-1} \, 1 \, 0^{2 a-2 b-1} \, 1 \, 0^{a-b-1} \, 1 \\
	& 0^{10 a-11 b-1} \, 1 \, 0^{-25 a+28 b-1} \, 1 \, 0^{28 a-31 b-1} \, 1 \, 0^{-25 a+28 b-1} \, 1 \, 0^{10 a-11 b-1} \, 1 \, 0^{10 a-11 b-1} \, 1 \, 0^{-7 a+8 b-1} \, 1 \\
	& 0^{10 a-11 b-1} \, 1 \, 0^{-8 a+9 b-1} \, 1 \, 0^{a-b-1} \, 1 \, 0^{10 a-11 b-1} \, 1 \, 0^{-25 a+28 b-1} \, 1 \, 0^{10 a-11 b-1} \, 1 \, 0^{-8 a+9 b-1} \, 1 \\
	& 0^{a-b-1} \, 1 \, 0^{10 a-11 b-1} \, 1 \, 0^{3 a-3 b-1} \, 1 \, 0^{10 a-11 b-1} \, 1 \, 0^{-25 a+28 b-1} \, 1 \, 0^{3 a-3 b-1} \, 1 \, 0^{10 a-11 b-1} \, 1 \\
	& 0^{a-b-1} \, 1 \, 0^{9 a-10 b-1} \, 1 \, 0^{-7 a+8 b-1} \, 1 \, 0^{10 a-11 b-1} \, 1 \, 0^{-25 a+28 b-1} \, 1 \, 0^{a-b-1} \, 1 \, 0^{9 a-10 b-1} \, 1 \\
	& 0^{-7 a+8 b-1} \, 1 \, 0^{2 a-2 b-1} \, 1 \, 0^{a-b-1} \, 1 \, 0^{10 a-11 b-1} \, 1 \, 0^{10 a-11 b-1} \, 1 \, 0^{2 a-2 b-1} \, 1 \, 0^{a-b-1} \, 1 \\
	& 0^{-25 a+28 b-1} \, 1 \, 0^{3 a-3 b-1} \, 1 \, 0^{10 a-11 b-1} \, 1 \, 0^{10 a-11 b-1} \, 1 \, 0^{3 a-3 b-1} \, 1 \, 0^{-25 a+28 b-1} \, 1 \, 0^{27 a-30 b-1} \, 1 \\
	& 0^{a-b-1} \, 1 \, 0^{-25 a+28 b-1} \, 1 \, 0^{10 a-11 b-1} \, 1 \, 0^{10 a-11 b-1} \, 1 \, 0^{-8 a+9 b-1} \, 1 \, 0^{a-b-1} \, 1 \, 0^{10 a-11 b-1} \, 1 \\
	& 0^{2 a-2 b-1} \, 2 \, 0^{a-b-1} \, 1 \, 0^{-25 a+28 b-1} \, 1 \, 0^{10 a-11 b-1} \, 1 \, 0^{2 a-2 b-1} \, 2 \, 0^{a-b-1} \, 1 \, 0^{10 a-11 b-1} \, 1 \\
	& 0^{3 a-3 b-1} \, 1 \, 0^{-25 a+28 b-1} \, 1 \, 0^{10 a-11 b-1} \, 1 \, 0^{3 a-3 b-1} \, 1 \, 0^{10 a-11 b-1} \, 1 \, 0^{a-b-1} \, 1 \, 0^{a-b-1} \, 2 \\
	& 0^{a-b-1} \, 1 \, 0^{-25 a+28 b-1} \, 1 \, 0^{10 a-11 b-1} \, 1 \, 0^{a-b-1} \, 1 \, 0^{a-b-1} \, 2 \, 0^{a-b-1} \, 1 \, 0^{2 a-2 b-1} \, 1 \\
	& 0^{11 a-12 b-1} \, 1 \, 0^{-25 a+28 b-1} \, 1 \, 0^{2 a-2 b-1} \, 1 \, 0^{11 a-12 b-1} \, 1 \, 0^{2 a-2 b-1} \, 1 \, 0^{a-b-1} \, 1 \, 0^{10 a-11 b-1} \, 1 \\
	& 0^{-25 a+28 b-1} \, 1 \, 0^{2 a-2 b-1} \, 1 \, 0^{a-b-1} \, 1 \, 0^{10 a-11 b-1} \, 1 \, 0^{-8 a+9 b-1} \, 1 \, 0^{11 a-12 b-1} \, 1 \, 0^{10 a-11 b-1} \, 1 \\
	& 0^{-25 a+28 b-1} \, 1 \, 0^{27 a-30 b-1} \, 1 \, 0^{-24 a+27 b-1} \, 1 \, 0^{2 a-2 b-1} \, 1 \, 0^{a-b-1} \, 1 \, 0^{10 a-11 b-1} \, 1 \, 0^{10 a-11 b-1} \, 1 \\
	& 0^{2 a-2 b-1} \, 1 \, 0^{a-b-1} \, 1 \, 0^{-25 a+28 b-1} \, 1 \, 0^{10 a-11 b-1} \, 1 \, 0^{-7 a+8 b-1} \, 1 \, 0^{10 a-11 b-1} \, 1 \, 0^{10 a-11 b-1} \, 1 \\
	& 0^{-25 a+28 b-1} \, 1 \, 0^{28 a-31 b-1} \, 1 \, 0^{-25 a+28 b-1} \, 1 \, 0^{27 a-30 b-1} \, 1 \, 0^{a-b-1} \, 1 \, 0^{-25 a+28 b-1} \, 1 \, 0^{10 a-11 b-1} \, 1 \\
	& 0^{10 a-11 b-1} \, 1 \, 0^{-8 a+9 b-1} \, 1 \, 0^{a-b-1} \, 1 \, 0^{10 a-11 b-1} \, 1 \, 0^{3 a-3 b-1} \, 1 \, 0^{-25 a+28 b-1} \, 1 \, 0^{10 a-11 b-1} \, 1 \\
	& 0^{3 a-3 b-1} \, 1 \, 0^{10 a-11 b-1} \, 1 \, 0^{a-b-1} \, 1 \, 0^{-26 a+29 b-1} \, 1 \, 0^{28 a-31 b-1} \, 1 \, 0^{-25 a+28 b-1} \, 1 \, 0^{10 a-11 b-1} \, 1 \\
	& 0^{a-b-1} \, 1 \, 0^{9 a-10 b-1} \, 1 \, 0^{-7 a+8 b-1} \, 1 \, 0^{2 a-2 b-1} \, 1 \, 0^{a-b-1} \, 1 \, 0^{10 a-11 b-1} \, 1 \, 0^{-25 a+28 b-1} \, 1 \\
	& 0^{2 a-2 b-1} \, 1 \, 0^{a-b-1} \, 1 \, 0^{10 a-11 b-1} \, 1 \, 0^{3 a-3 b-1} \, 1 \, 0^{10 a-11 b-1} \, 1 \, 0^{-25 a+28 b-1} \, 1 \, 0^{3 a-3 b-1} \, 1 \\
	& 0^{10 a-11 b-1} \, 1 \, 0^{-8 a+9 b-1} \, 1 \, 0^{a-b-1} \, 1 \, 0^{10 a-11 b-1} \, 1 \, 0^{10 a-11 b-1} \, 1 \, 0^{-25 a+28 b-1} \, 1 \, 0^{27 a-30 b-1} \, 1 \\
	& 0^{a-b-1} \, 1 \, 0^{-25 a+28 b-1} \, 1 \, 0^{2 a-2 b-1} \, 2 \, 0^{a-b-1} \, 1 \, 0^{10 a-11 b-1} \, 1 \, 0^{10 a-11 b-1} \, 1 \, 0^{2 a-2 b-1} \, 2 \\
	& 0^{a-b-1} \, 1 \, 0^{-25 a+28 b-1} \, 1 \, 0^{3 a-3 b-1} \, 1 \, 0^{10 a-11 b-1} \, 1 \, 0^{10 a-11 b-1} \, 1 \, 0^{3 a-3 b-1} \, 1 \, 0^{-25 a+28 b-1} \, 1 \\
	& 0^{a-b-1} \, 1 \, 0^{a-b-1} \, 2 \, 0^{a-b-1} \, 1 \, 0^{10 a-11 b-1} \, 1 \, 0^{10 a-11 b-1} \, 1 \, 0^{a-b-1} \, 1 \, 0^{a-b-1} \, 2 \\
	& 0^{a-b-1} \, 1 \, 0^{2 a-2 b-1} \, 1 \, 0^{-24 a+27 b-1} \, 1 \, 0^{10 a-11 b-1} \, 1 \, 0^{2 a-2 b-1} \, 1 \, 0^{11 a-12 b-1} \, 1 \, 0^{2 a-2 b-1} \, 1 \\
	& 0^{a-b-1} \, 1 \, 0^{-25 a+28 b-1} \, 1 \, 0^{10 a-11 b-1} \, 1 \, 0^{2 a-2 b-1} \, 1 \, 0^{a-b-1} \, 1 \, 0^{10 a-11 b-1} \, 1 \, 0^{-8 a+9 b-1} \, 1 \\
	& 0^{11 a-12 b-1} \, 1 \, 0^{-25 a+28 b-1} \, 1 \, 0^{10 a-11 b-1} \, 1 \, 0^{-8 a+9 b-1} \, 1 \, 0^{11 a-12 b-1} \, 1 \, 0^{2 a-2 b-1} \, 1 \, 0^{a-b-1} \, 1 \\
	& 0^{10 a-11 b-1} \, 1 \, 0^{-25 a+28 b-1} \, 1 \, 0^{2 a-2 b-1} \, 1 \, 0^{a-b-1} \, 1 \, 0^{10 a-11 b-1} \, 1 \, 0^{10 a-11 b-1} \, 1 \, 0^{-7 a+8 b-1} \, 1 \\
	& 0^{10 a-11 b-1} \, 1 \, 0^{-25 a+28 b-1} \, 1 \, 0^{10 a-11 b-1} \, 1 \, 0^{-7 a+8 b-1} \, 1 \, 0^{10 a-11 b-1} \, 1 \, 0^{-8 a+9 b-1} \, 1 \, 0^{a-b-1} \, 1 \\
	& 0^{10 a-11 b-1} \, 1 \, 0^{10 a-11 b-1} \, 1 \, 0^{-25 a+28 b-1} \, 1 \, 0^{27 a-30 b-1} \, 1 \, 0^{a-b-1} \, 1 \, 0^{-25 a+28 b-1} \, 1 \, 0^{3 a-3 b-1} \, 1 \\
	& 0^{10 a-11 b-1} \, 1 \, 0^{10 a-11 b-1} \, 1 \, 0^{3 a-3 b-1} \, 1 \, 0^{-25 a+28 b-1} \, 1 \, 0^{a-b-1} \, 1 \, 0^{9 a-10 b-1} \, (n+1),
\end{align*}
\normalsize
with $279$ nonzero letters, locates words of length $5 a - 4 b$ and is $\ab$-power-free.
\end{theorem}

The morphisms in Theorems~\ref{12 7a-5b power-free}, \ref{16 9a-7b power-free}, \ref{20 11a-9b power-free}, and \ref{24 13a-11b power-free} appear to belong to a general family of parameterized morphisms with $4r$ nonzero letters for each $r \geq 3$.

\begin{conjecture}\label{4r power-free}
Let $r \geq 3$ be an integer.
Let $a, b$ be relatively prime positive integers such that $\frac{2 r + 1}{2 r} < \ab < \frac{2 r}{2 r - 1}$ and $\ab \neq \frac{4 r + 1}{4 r - 1}$ and $\gcd(b, 2 r + 1) = 1$.
Let
\begin{align*}
	A &= 0^{a - b - 1} \, 1 \\
	B &= 0^{2 a - 2 b - 1} \, 1 \\
	C &= 0^{3 a - 3 b - 1} \, 1 \\
	X &= 0^{(2r+1) a - (2r+2) b - 1} \, 1 \\
	Y &= 0^{-(2r-2) a + (2r-1) b - 1} \, 1 \\
	Z &= 0^{(2r) a - (2r+1) b - 1} \, 1.
\end{align*}
Then the $((2 r + 1) a - (2 r - 1) b)$-uniform morphism
\[
	\varphi(n) = 
	X (Y Z)^{r-2} B A Y B^{r-2} C Y B^{r-3} Y A \,
	0^{a - b - 1} \, (n+1),
\]
with $4r$ nonzero letters, is $\ab$-power-free.
\end{conjecture}

Because of the additional symbolic parameter $r$, proving Conjecture~\ref{4r power-free} is beyond the scope of our code.
If this conjecture is true, then it exhibits $\ab$-power-free morphisms with $\ab$ arbitrarily close to $1$.

For $r = 2$ the morphism in Conjecture~\ref{4r power-free} is not defined, due to the factor $B^{r-3}$.
However, moving from the free monoid to the free group on $\Z_{\geq 0}$ allows us to interpret the factor $Y B^{r - 3}$ for $r = 2$ as $0^{-2 a + 3 b - 1} \, 1 \cdot (0^{2 a - 2 b - 1} \, 1)^{-1} = 0^{-4 a + 5 b}$.
By doing this, we obtain the morphism in Theorem~\ref{8 5a-3b power-free}.

\subsection{Finding morphisms experimentally}\label{finding conjectures}

The statements of the theorems in Section~\ref{symbolic power-free morphisms}, with the exception of Theorem~\ref{8 5a-3b power-free}, were discovered by computing prefixes of $\word_{a/b}$ for $910$ rational numbers in the interval $1 < \ab < 2$.
In all, we computed over $256$ million letters.
For each $\ab$, we attempted to find an integer $k$ such that partitioning (the prefix of) $\word_{a/b}$ into rows of length $k$ as in Figures~\ref{3/2 5/3 9/5 arrays} and \ref{6/5 array} produces an array with $k-1$ eventually periodic columns and one self-similar column in which $\word_{a/b}$ reappears in some modified form.

In some cases, $k$ can be found easily by determining the largest integer $c$ that occurs in the prefix, computing the positions where it occurs, and computing the $\gcd$ of the successive differences of these positions.
If $\word_{a/b} = \varphi^\infty(0)$ for some $\varphi(n) = u \, (n + d)$ where $c$ does not occur in $u$, then this $\gcd$ is a multiple of $|\varphi(n)| = k$.
When this method does not identify a candidate $k$, one can look for periodic blocks in the difference sequence of the positions of $1$ (or some larger integer), and add the integers in the repetition period to get the length of the corresponding repeating factor of $\word_{a/b}$; this procedure can detect repetitions of $\varphi(0)$ in $\word_{a/b}$.

We identified conjectural structure in $\word_{a/b}$ for $520$ of the $910$ rational numbers.
(Note that these $910$ numbers were not chosen uniformly; some were chosen to bound the interval endpoints for symbolic morphisms that had already been conjectured.)
Of these $520$ words, $510$ have the property that $\word_{a/b} + d$ (the word obtained by incrementing each letter by $d$) appears in the self-similar column for some integer $d \geq 0$.

The remaining $10$ words do not have a constant difference $d$.
For the $4$ rationals $\frac{59}{48}, \frac{65}{57}, \frac{73}{60}, \frac{113}{99}$ the word $\word_{a/b}$ reappears in the self-similar column with its letters incremented by a periodic but not constant sequence.
For the $6$ rationals $\frac{4}{3}, \frac{13}{12}, \frac{15}{14}, \frac{28}{25}, \frac{37}{34}, \frac{64}{59}$ the increment appears to depend on the letter, as in Theorem~\ref{4/3}.

To find families of words $\word_{a/b}$ with related structure among the $510$ with constant $d$, for each word we record
\begin{itemize}
\item
the difference $d$,
\item
the number of columns $k$,
\item
the index of the self-similar column,
and
\item
the number of transient rows.
\end{itemize}

For words $\word_{a/b}$ such that all columns except the self-similar column are eventually constant, we build a word $u$ of length $k-1$ from the eventual values of these columns.
Letting $\varphi(n) = u \, (n + d)$, the structure of $\word_{a/b}$ is potentially related to $\varphi^\infty(0)$.

If $u_1$ and $u_2$, arising from two different words, have the same subsequence $c_1, \dots, c_{r-1}$ of nonzero letters, we can look for a symbolic morphism $\varphi(n) = u \, (n + d)$ that generalizes the two morphisms $\varphi_1(n) = u_1 \, (n + d)$ and $\varphi_2(n) = u_2 \, (n + d)$ by writing
\[
	u = 0^{i_1 a + j_1 b - 1} c_1 \, \cdots \, 0^{i_{r-1} a + j_{r-1} b - 1} c_{r-1} \, 0^{i_r a + j_r b - 1}
\]
and solving each linear system $i a + j b = l$ using the two rationals $\ab$ to determine~$i, j$.
If there is not a unique solution for some block, discard this pair of rationals; otherwise this gives a unique symbolic morphism.

For each pair of words with the same $d$ and same subsequence of nonzero letters, we construct a symbolic morphism, if possible.
If multiple pairs of rationals give the same symbolic morphism, this suggests a general family.
On the other hand, a symbolic morphism is likely not meaningful if it only appears for one pair of rationals and contains run lengths where the coefficients of $a, b$ are rational numbers with large denominators.
In practice, the coefficients of $a, b$ in all morphisms in Section~\ref{symbolic power-free morphisms} are integers, although it is conceivable that families with non-integer coefficients exist (in which case $\ab$ would be restricted by a $\gcd$ condition on $a$ or~$b$).

For each symbolic morphism, we then attempt to determine an interval $\Imin < \ab < \Imax$ on which the morphism is $\ab$-power-free.
Each block $0^{i a + j b - 1} c$ in $\varphi(n)$ restricts the values that $\ab$ can take, since the run length must be non-negative.
We solve the homogeneous equation $i a + j b = 0$ to get a lower bound or upper bound on $\ab$ for each block.
We use the maximum lower bound and minimum upper bound as initial guesses for the interval endpoints.
If this guessed interval is too wide, then running the algorithm identifies obstructions to $\ab$-power-freeness, and we shorten the interval by removing the subintervals on which $\ab$-power-freeness failed to be verified.

Many symbolic morphisms do not turn out to be $\ab$-power-free on a general interval.
A common problem is that solving a homogeneous equation to split an interval gives a value that is not in the interval.
A particularly disappointing case occurs among morphisms with $14$ nonzero letters.
We identified $30$ rational numbers in the interval $\frac{4}{3} < \ab < \frac{3}{2}$ for which the correct value of $k$ for $\word_{a/b}$ seems to be $6 a - b$ and which have $14$ eventually nonzero columns when partitioned into rows of length $k$.
One might expect the structure of all these words to be explained by the same symbolic morphism, but in fact no three of these words are captured by the same symbolic morphism.
We only found suitable intervals for two of the resulting morphisms (Theorems~\ref{14 6a-b I power-free} and \ref{14 6a-b II power-free}).
Both morphisms use $\frac{24}{17}$ as one of their sources, which leaves $27$ of the $30$ rationals without a symbolic morphism.

%%%%%%%%%%%%%%%%%%%%%%%%%%%%%%%%%%%%%%%%%%
\section{Families of words $\word_{a/b}$}\label{Families of words}
%%%%%%%%%%%%%%%%%%%%%%%%%%%%%%%%%%%%%%%%%%

In Section~\ref{power-free morphisms} we identified a number of symbolic $\ab$-power-free morphisms that were derived from words $\word_{a/b}$.
In this section we discuss exact relationships between some of these morphisms and $\word_{a/b}$.
We begin with the morphism in Theorem~\ref{4 5a-4b power-free}.

\begin{theorem}\label{4 5a-4b}
Let $a, b$ be relatively prime positive integers such that $\frac{3}{2} < \ab < \frac{5}{3}$ and $\gcd(b, 5) = 1$.
Let $\varphi$ be the $(5a-4b)$-uniform morphism defined by
\[
	\varphi(n) = 0^{a-1} \, 1 \, 0^{a-b-1} \, 1 \, 0^{2 a-2 b-1} \, 1 \, 0^{a-b-1} \, (n+1).
\]
Then $\word_{a/b} = \varphi^\infty(0)$.
In particular, $\rho(\ab) = 5a-4b$.
\end{theorem}

\begin{proof}
The morphism $\varphi$ is $\ab$-power-free by Theorem~\ref{4 5a-4b power-free}, so it suffices to show that $\varphi^\infty(0)$ is the lexicographically least $\ab$-power-free word.
Write $\varphi(n) = u \, (n+1)$.
Decrementing one of the three $1$ letters in $u$ to $0$ introduces the $\ab$-power $(0^b)^{a/b}$, $(0^{b-1} 1)^{a/b}$, or $(0^{-a+2b-1} 1 0^{a-b})^{a/b}$.
The word $u \, 0$ ends with $(0^{b-1} 1)^{a/b}$, and the induction argument showing that decrementing $n + 1$ to $c \geq 1$ introduces an $\ab$-power works exactly as in the proof of Theorem~\ref{2 2a-b}.
Namely, decrementing $n + 1$ to $c$ corresponds, under $\varphi$, to decrementing an earlier letter $n$ to $c - 1$.
\end{proof}

The morphisms in Theorems~\ref{2 2a-b power-free} and \ref{4 5a-4b power-free} are the only morphisms in Section~\ref{symbolic power-free morphisms} that were derived from words $\word_{a/b}$ with no transient.
We can see $\word_{a/b} \neq \varphi^\infty(0)$ for the other morphisms, since the length-$a$ prefix of $\word_{a/b}$ is $0^{a-1} 1$ (as in Proposition~\ref{inequality of words}) but the length-$a$ prefix of $\varphi^\infty(0)$ is not $0^{a-1} 1$.
To account for a transient, we extend $\varphi$ to the alphabet $\Z_{\geq 0} \cup \{0'\}$ and consider morphisms of the form
\[
	\varphi(n) =
	\begin{cases}
		v \, \varphi(0)		& \text{if $n = 0'$} \\
		u \, (n + d)			& \text{if $n \in \Z_{\geq 0}$}.
	\end{cases}
\]

Still, we cannot expect every morphism in Section~\ref{symbolic power-free morphisms} to be related to $\word_{a/b}$ for each $\ab$ in its corresponding interval, since there exist rationals to which multiple theorems apply with different values of $k$; by Corollary~\ref{unique k}, the word $\tau(\varphi^\infty(0'))$ is $k$-regular for a unique value of $k$ up to multiplicative dependence.
For example, $\ab = \frac{24}{17}$ satisfies the conditions of Theorems~\ref{6 4a-2b I power-free}, \ref{6 4a-2b II power-free}, \ref{14 6a-b I power-free}, and \ref{14 6a-b II power-free}.
The corresponding values of $k$ are $4 a - 2 b = 62$ for Theorems~\ref{6 4a-2b I power-free} and \ref{6 4a-2b II power-free}, and $6 a - b = 127$ for Theorems~\ref{14 6a-b I power-free} and \ref{14 6a-b II power-free}; these two row widths are shown in Figure~\ref{24/17 arrays}.
While there are regions of the word $\word_{24/17}$ that have constant columns when partitioned into rows of width $62$, these do not persist, and therefore it seems the morphisms in Theorems~\ref{6 4a-2b I power-free} and \ref{6 4a-2b II power-free} do not determine the long-term structure.

\begin{figure}
	\includegraphics[scale=.33]{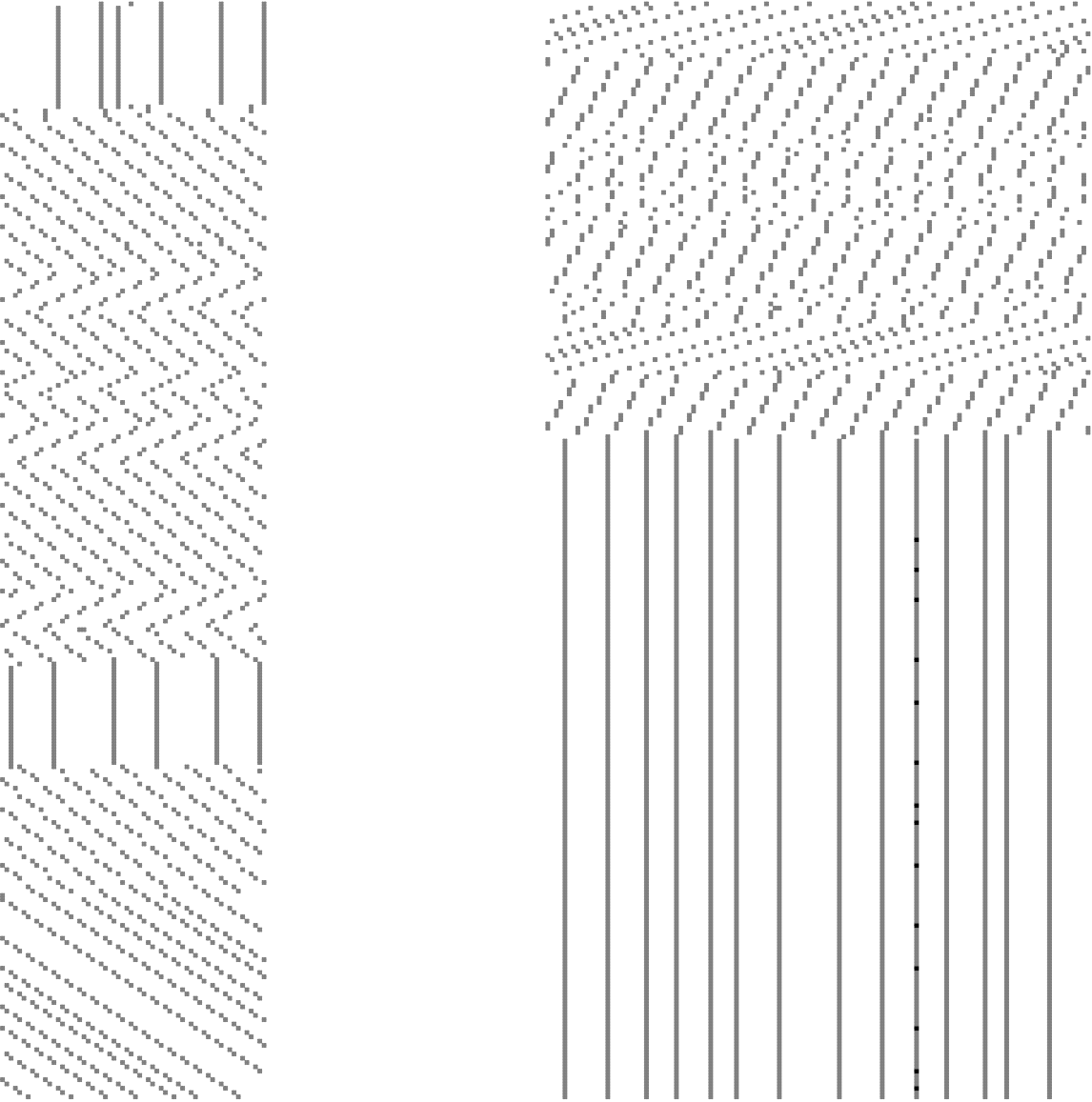}
	\caption{Prefixes of $\word_{24/17}$, partitioned into rows of width $k = 62$ (left) and $k = 127$ (right).}
	\label{24/17 arrays}
\end{figure}

To establish the structure of a word $\word_{a/b}$ with a transient, we must generalize our approach to proving $\ab$-power-freeness and lexicographic-leastness.
For rationals satisfying the conditions of Theorem~\ref{37 24a-15b power-free}, the word $\word_{a/b}$ has a short transient.
Recall that $\tau(0') = 0$ and $\tau(n) = n$ for $n \in \Z_{\geq 0}$.

\begin{theorem}\label{37 24a-15b}
Let $a, b$ be relatively prime positive integers such that $\frac{9}{7} < \ab < \frac{4}{3}$ and $\gcd(b, 24) = 1$.
Let
\[
	v = 0' \, 0^{a-2} \, 1 \, 0^{a-b-1} \, 1 \, 0^{a-b-1} \, 1.
\]
Let $\varphi$ be defined as in Theorem~\ref{37 24a-15b power-free}, extended to the alphabet $\Z_{\geq 0} \cup \{0'\}$ by
\[
	\varphi(0') = v \, \varphi(0).
\]
Then $\word_{a/b} = \tau(\varphi^\infty(0'))$.
In particular, $\rho(\ab) = 24a-15b$.
\end{theorem}

\begin{proof}
First we show that $\tau(\varphi^\infty(0'))$ is $\ab$-power-free.
Since
\[
	\varphi^\infty(0') = v \, \varphi(\tau(v)) \, \varphi^2(\tau(v)) \, \cdots,
\]
a factor beginning at position $i \geq |v|$ in $\tau(\varphi^\infty(0'))$ is a factor of $\varphi(w)$ for some finite factor $w$ of $\tau(\varphi^\infty(0'))$.
Since ${\varphi |}_{\Z_{\geq 0}}$ is $\ab$-power-free by Theorem~\ref{37 24a-15b power-free}, it follows that if $\tau(\varphi^\infty(0'))$ contains an $\ab$-power then it contains an $\ab$-power beginning at some position $i \leq |v| - 1$.

For each $n \in \Z_{\geq 0}$, the word $\varphi(n)$ begins with $p \colonequal 0^{a-b-1} \, 1 \, 0^{2 a-2 b-1} \, 1$.
The prefix $v$ ends with $1$, while $\varphi(n)$ ends with $n + 2 \neq 1$.
By examining $\varphi(n)$, one checks that if $w$ is a word on the alphabet $\Z_{\geq 0}$ then $1 \, p$ is not a factor of $\varphi(w)$.
It follows that $1 \, p$ occurs at position $|v| - 1$ in $\tau(\varphi^\infty(0'))$ but does not occur after position $|v| - 1$.
If a factor $x$ begins at position $i \leq |v| - 1$ in $\tau(\varphi^\infty(0'))$ and has length $|x| \geq |v| + |p| = (3 a - 2 b) + (3 a - 3 b) = 6 a - 5 b$, then $x$ contains $1 \, p$ and therefore only occurs once in $\tau(\varphi^\infty(0'))$.
Therefore if $\tau(\varphi^\infty(0'))$ contains an $\ab$-power factor $(xy)^{a/b} = xyx$ then $|x| < 6 a - 5 b$; since $\frac{9}{7} < \ab < \frac{4}{3}$, we have $|x| < 6 a - 5 b \leq 10 (a-b)$.

It remains to show that $\tau(\varphi^\infty(0'))$ contains no $\ab$-power $(xy)^{a/b} = xyx$ with $|x| \leq 9 (a-b)$ beginning at a position $i \leq |v| - 1$.
For each $m$ in the range $1 \leq m \leq 9$, this is accomplished by sliding a window of length $m a$ through $\tau(\varphi^\infty(0'))$ from position $0$ to position $|v| - 1$ and verifying inequality of symbolic factors as in Section~\ref{testing inequality}.

Now we show that decrementing any nonzero letter of $\tau(\varphi^\infty(0'))$ introduces an $\ab$-power.
Decrementing one of the three $1$ letters in the prefix $\tau(v)$ to $0$ introduces the $\ab$-power $(0^b)^{a/b}$ or $(0^{b-1} 1)^{a/b}$.
Every other nonzero letter is a factor of $\varphi(n)$ for some integer $n$.
Write $\varphi(n) = u \, (n+2)$ for $n \in \Z_{\geq 0}$.
The word $u \, 2$ contains $37$ nonzero letters.
One checks that decrementing each $1$ in $u \, 2$, except the first two, to $0$ and each $2$ in $u \, 2$ to $0$ or $1$ introduces an $\ab$-power of length $a$ or $2 a$.
For the first two $1$s, there are two cases each, depending on whether $u$ is immediately preceded by $\tau(v)$ or $\varphi(n)$.
Decrementing the first $1$ to $0$ introduces the $\ab$-power
\[
	0^{a-b} \cdot 0^{-3 a+4 b-1} \, 1 \, 0^{a-b-1} \, 1 \, 0^{a-b-1} \, 1 \cdot 0^{a-b}
\]
if preceded by $\tau(v)$ and the $\ab$-power
\[
	0^{a-b} \cdot 0^{-3 a+4 b-1} \, 1 \, 0^{a-b-1} \, 1 \, 0^{a-b-1} \, (n+2) \cdot 0^{a-b}
\]
if preceded by $\varphi(n)$.
Decrementing the second $1$ to $0$ introduces the $\ab$-power
\[
	0^{2 a-2 b} \cdot 0^{-5 a+7 b-1} \, 1 \, 0^{a-b-1} \, 1 \, 0^{a-b-1} \, 1 \, 0^{a-b-1} \, 1 \cdot 0^{2 a-2 b}
\]
if preceded by $\tau(v)$ and the $\ab$-power
\[
	0^{2 a-2 b} \cdot 0^{-3 a+4 b-1} \, 1 \, 0^{-2 a+3 b-1} \, 1 \, 0^{a-b-1} \, 1 \, 0^{a-b-1} \, (n+2) \, 0^{a-b-1} \, 1 \cdot 0^{2 a-2 b}
\]
if preceded by $\varphi(n)$.
As in the proof of Theorem~\ref{4 5a-4b}, decrementing $n + 2$ to $c \geq 2$ corresponds to decrementing an earlier letter $n$ to $c - 2$ and therefore, inductively, introduces an $\ab$-power.
\end{proof}

In addition to the presence of transients, another complication is that some words $\word_{a/b}$ reappear with finitely many modified letters in the self-similar column.
As shown in Figure~\ref{19/16 array}, $\word_{19/16}$ has $k-1$ eventually constant columns when partitioned into $k = 53$ columns.
After $4$ transient rows, the self-similar column consists of $\word_{19/16} + 1$ with the letter at position $18$ changed from $2$ to $0$.
The previous letter in $\word_{19/16}$ is $1$, despite it being in an eventually-$0$ column.
However, changing this $1$ to $0$ introduces an $\ab$-power
\[
	\left(0^{-5 a + 6 b - 1} \, 1 \, \varphi(0) \, 1^{-1} \, 0^{-(-5 a + 6 b - 1)}\right)^a
\]
ending at that position, where we work in the free group on $\Z_{\geq 0}$ in order to remove some letters from the end of $\varphi(0)$.
This happens because the prefix $v$ and $\varphi(0)$ have a nonempty common suffix.
In contrast, in Theorem~\ref{37 24a-15b} the prefix $v$ ends with $1$ whereas $\varphi(n)$ ends with $n+2$, so we aren't at risk of completing the $a$-power before we reach the self-similar column on the $(a+1)$th row.

To capture this modification, we introduce a new letter $1'$, define $\tau(1') = 1$, and define $\varphi(1')$ to be identical to $\varphi(1)$ except in two positions.
Only minor changes to the proof of Theorem~\ref{regular sequences} are necessary to establish that the sequence of letters in $\tau(\varphi^\infty(0'))$ is $k$-regular.
The following conjecture claims the structure of $\word_{19/16}$ is related to the morphism in Theorem~\ref{12 7a-5b power-free}.
Note however that the interval is shorter than in Theorem~\ref{12 7a-5b power-free}.

\begin{conjecture}\label{12 7a-5b}
Let $a, b$ be relatively prime positive integers such that $\frac{13}{11} < \ab < \frac{6}{5}$ and $\gcd(b, 7) = 1$.
Let $\varphi$ be the morphism defined in Theorem~\ref{12 7a-5b power-free}.
Let
\begin{align*}
	v = {}
	& 0' \, 0^{a-2} \, 1' \, 0^{a-b-1} \, 1 \, 0^{a-b-1} \, 1 \, 0^{a-b-1} \, 1 \, 0^{a-b-1} \, 1 \, 0^{a-b-1} \, 1 \, 0^{2 a-2 b-1} \, 1 \, 0^{2 a-2 b-1} \, 1 \\
	& 0^{-3 a+4 b-1} \, 1 \, 0^{2 a-2 b-1} \, 1 \, 0^{2 a-2 b-1} \, 1 \, 0^{a-b-1} \, 1 \, 0^{-4 a+5 b-1} \, 1 \, 0^{6 a-7 b-1} \, 1 \, 0^{-4 a+5 b-1} \, 1 \\
	& 0^{6 a-7 b-1} \, 1 \, 0^{-3 a+4 b-1} \, 1 \, 0^{2 a-2 b-1} \, 1 \, 0^{2 a-2 b-1} \, 1 \, 0^{a-b-1} \, 1 \, 0^{-4 a+5 b-1} \, 1 \, 0^{2 a-2 b-1} \, 1 \\
	& 0^{-4 a+5 b-1} \, 1 \, 0^{7 a-8 b-1} \, 1 \, 0^{-4 a+5 b-1} \, 1 \, 0^{6 a-7 b-1} \, 1 \, 0^{-4 a+5 b-1} \, 1 \, 0^{6 a-7 b-1} \, 1 \, 0^{a-b-1} \, 1 \\
	& 0^{-4 a+5 b-1} \, 1 \, 0^{2 a-2 b-1} \, 1 \, 0^{3 a-3 b-1} \, 1 \, 0^{-4 a+5 b-1} \, 1 \, 0^{2 a-2 b-1} \, 1 \, 0^{a-b-1} \, 1 \, 0^{-5 a+6 b-1} \, 1.
\end{align*}
Extend $\varphi$ to the alphabet $\Z_{\geq 0} \cup \{0', 1'\}$ by
\[
	\varphi(0') = v \, \varphi(0),
\]
and define $\varphi(1')$ to be the word obtained by taking $\varphi(1)$ and changing the $0$ at position $12 a - 11 b - 1$ to $1$ and the last letter (at position $7 a - 5 b - 1$) to $0$.
Then $\word_{a/b} = \tau(\varphi^\infty(0'))$.
In particular, $\rho(\ab) = 7a-5b$.
\end{conjecture}

\begin{figure}
	\includegraphics[scale=.5]{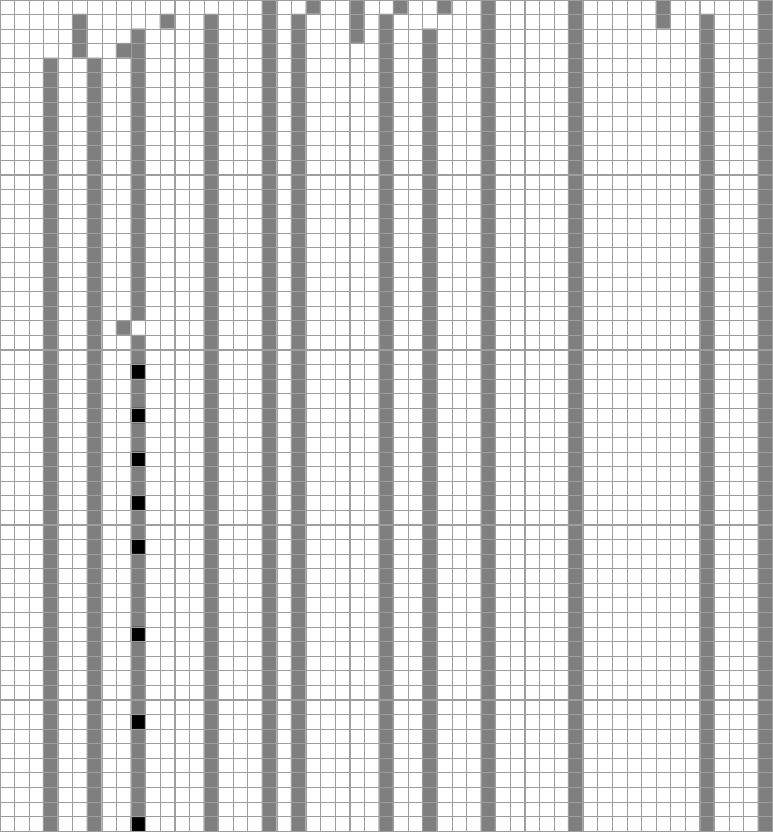}
	\caption{A prefix of $\word_{19/16}$, partitioned into rows of width $53$.}
	\label{19/16 array}
\end{figure}

The prefix $v$ has length $|v| = 19 a - 12 b$.
The length of $\varphi(1)$ is $|\varphi(1)| = 7a-5b$, so the $0$ at position $12 a - 11 b - 1$ in $\varphi(1)$ is contained in the last block of $0$s.
Therefore
\begin{align*}
	\varphi(1') = {}
	& 0^{7 a-8 b-1} \, 1 \, 0^{-4 a+5 b-1} \, 1 \, 0^{6 a-7 b-1} \, 1 \, 0^{2 a-2 b-1} \, 1 \, 0^{a-b-1} \, 1 \, 0^{-4 a+5 b-1} \, 1 \\
	& 0^{2 a-2 b-1} \, 1 \, 0^{3 a-3 b-1} \, 1 \, 0^{-4 a+5 b-1} \, 1 \, 0^{-4 a+5 b-1} \, 1 \, 0^{a-b-1} \, 1 \, 0^{6 a-7 b-1} \, 1 \, 0^{-5 a+6 b}.
\end{align*}

Conjecture~\ref{12 7a-5b} is the case $r=3$ of a more general conjecture corresponding to the morphism in Conjecture~\ref{4r power-free} with $4r$ nonzero letters.
As in Conjecture~\ref{12 7a-5b}, the interval is shorter than in Conjecture~\ref{4r power-free}.

\begin{conjecture}\label{4r}
Let $r \geq 3$ be an integer.
Let $a, b$ be relatively prime positive integers.
Let $\varphi$ be the morphism defined in Conjecture~\ref{4r power-free}.
Let
\[
	v = 0' \, 0^{a-2} \, 1' \, A^{2r-1} B^{r-1} V B^{r-1} A (Y Z)^{r-1} V B^{r-1} A Y B^{r-2} Y X (Y Z)^{r-1} A Y B^{r-2} C Y B^{r-2} A W,
\]
where $A, B, C, X, Y, Z$ are as defined in Conjecture~\ref{4r power-free} and
\begin{align*}
	V &= 0^{-(2r-3) a + (2r-2) b - 1} \, 1 \\
	W &= 0^{-(2r-1) a + (2r) b - 1} \, 1.
\end{align*}
Extend $\varphi$ to $\Z_{\geq 0} \cup \{0', 1'\}$ by
\[
	\varphi(0') = v \, \varphi(0),
\]
and define $\varphi(1')$ to be the word obtained by taking $\varphi(1)$ and changing the $0$ at position $(4 r) a - (4 r - 1) b - 1$ to $1$ and the last letter to $0$.
There exists a finite set $Q_r \subset \Q$ such that if $\frac{4 r + 1}{4 r - 1} < \ab < \frac{2 r}{2 r - 1}$ and $\ab \notin Q_r$ and $\gcd(b, 2 r + 1) = 1$ then $\word_{a/b} = \tau(\varphi^\infty(0'))$.
In particular, $\rho(\ab) = (2 r + 1) a - (2 r - 1) b$.
\end{conjecture}

Exceptional rationals can arise when the morphism fails to produce an $\ab$-power-free word.
For example, $\frac{37}{34} \in Q_6$, since $\varphi(0')$ contains an $\frac{37}{34}$-power of length $8 \cdot 37$ ending at position $410$.
The correct value of $k$ for $\word_{37/34}$ seems to be $777$.

Exceptional rationals can also arise when a morphism fails to produce a lexicographically least word.
The following conjecture gives the structure of $\word_{a/b}$ for most rationals satisfying the conditions of Theorem~\ref{42 13a-8b power-free}.
However, $\frac{74}{67}$ is a new exceptional rational because the transient $v$ contains a $1$ at position $1247$, whereas a $0$ does not complete any $\frac{74}{67}$-powers.
Interestingly, the correct value of $k$ for $\ab = \frac{74}{67}$ seems to nonetheless be $k = 13a-8b = 426$, and the nonzero columns have the same spacings as in the morphism, but the self-similar column is in a different place.
It is possible there are additional exceptions like $\frac{74}{67}$.

\begin{conjecture}\label{42 13a-8b}
Let $a, b$ be relatively prime positive integers such that $\frac{11}{10} < \ab < \frac{21}{19}$ and $\ab \notin \{\frac{32}{29}, \frac{74}{67}\}$ and $\gcd(b,13) = 1$.
There exists a word
\[
	v = 0' \, 0^{a - 2} \, 1' \, 0^{a - b - 1} \, 1 \, 0^{a - b - 1} \, 1 \, \cdots \, 0^{-19 a + 21 b - 1} \, 1
\]
of length $317a-178b$ with $1127$ letters from $\{1', 1, 2\}$ such that if we extend the morphism $\varphi$ defined in Theorem~\ref{42 13a-8b power-free} to $\Z_{\geq 0} \cup \{0', 1'\}$ by
\[
	\varphi(0') = v \, \varphi(0)
\]
and define $\varphi(1')$ to be the word obtained by taking $\varphi(1)$ and changing the $0$ at position $32 a - 29 b - 1$ to $1$ and the last letter to $0$, then $\word_{a/b} = \tau(\varphi^\infty(0'))$.
In particular, $\rho(\ab) = 13a-8b$.
\end{conjecture}

For rationals satisfying the conditions of Theorem~\ref{54 13a-5b power-free}, the word $\word_{a/b}$ appears to be given by the following.

\begin{conjecture}\label{54 13a-5b}
Let $a, b$ be relatively prime positive integers such that $\frac{9}{8} < \ab < \frac{26}{23}$ and $\ab \neq \frac{35}{31}$ and $\gcd(b,13) = 1$.
There exists a word
\[
	v = 0' \, 0^{a - 2} \, 1' \, 0^{a - b - 1} \, 1 \, 0^{a - b - 1} \, 1 \, \cdots \, 0^{-23 a + 26 b - 1} \, 1,
\]
of length $148a-66b$ with $552$ letters from $\{1', 1, 2\}$ such that if we extend the morphism $\varphi$ defined in Theorem~\ref{54 13a-5b power-free} to $\Z_{\geq 0} \cup \{0', 1'\}$ by
\[
	\varphi(0') = v \, \varphi(0)
\]
and define $\varphi(1')$ to be the word obtained by taking $\varphi(1)$ and changing the $0$ at position $36 a - 31 b - 1$ to $1$ and the last letter to $0$, then $\word_{a/b} = \tau(\varphi^\infty(0'))$.
In particular, $\rho(\ab) = 13a-5b$.
\end{conjecture}

Of the $29$ symbolic morphisms given in Section~\ref{symbolic power-free morphisms} that were derived from words $\word_{a/b}$, at this point we have given a theorem or conjecture relating $9$ of them to an infinite family of words $\word_{a/b}$.
The $10$ morphisms in Theorems~\ref{4 a power-free}, \ref{6 4a-2b I power-free}, \ref{6 4a-2b II power-free}, \ref{10 11a-10b power-free}, \ref{13 7a-4b power-free}, \ref{14 6a-b I power-free}, \ref{14 6a-b II power-free}, \ref{26 9a-4b power-free}, \ref{46 14a-9b II power-free}, and \ref{191 66a-28b power-free} were derived from at most four each and therefore are perhaps not likely to be related to infinitely many words $\word_{a/b}$.
It seems likely that the final $10$ morphisms do govern the large-scale structure of infinitely many words $\word_{a/b}$.
However, none of these families support a symbolic prefix $v$ of the form we have seen previously, since the number of nonzero letters in the prefix is not constant.
Specifying the general structure of these words will therefore require a new idea.
We state the following conjectures in terms of $\rho$ only.

For the two families of words that give rise to the morphisms in Theorems~\ref{18 10a-8b power-free} and \ref{158 53a-30b power-free}, the index of the self-similar column is a linear combination of $a, b$, which might indicate some structure that would be useful in a proof.
Note that the intervals are shorter than the intervals for $\ab$-power-freeness.

\begin{conjecture}\label{18 10a-8b}
Let $a, b$ be relatively prime positive integers such that $\frac{28}{25} < \ab < \frac{9}{8}$ and $\gcd(b, 10) = 1$.
Then $\rho(\ab) = 10a-8b$.
\end{conjecture}

\begin{conjecture}\label{158 53a-30b}
Let $a, b$ be relatively prime positive integers such that $\frac{26}{23} < \ab < \frac{17}{15}$ and $\gcd(b, 53) = 1$.
Then $\rho(\ab) = 53a-30b$.
\end{conjecture}

The remaining morphisms are those in Theorems~\ref{3 a power-free}, \ref{14 8a-4b power-free}, \ref{29 14a-9b power-free}, \ref{30 10a-5b power-free}, \ref{38 12a-7b power-free}, \ref{46 14a-9b I power-free}, \ref{102 38a-15b power-free}, and \ref{279 67a-30b power-free}.
For words related to these, the index of the self-similar column is not a linear combination of $a, b$.
The morphism in Theorem~\ref{3 a power-free} seems to be related to $\word_{a/b}$ only for odd $b$, so we add this condition to Conjecture~\ref{3 a}.
Additionally, the morphism in Theorem~\ref{14 8a-4b power-free} seems to be related to $\word_{a/b}$ only for $b$ divisible by $3$, so we add this condition to Conjecture~\ref{14 8a-4b}.

\begin{conjecture}\label{3 a}
Let $a, b$ be relatively prime positive integers such that $\frac{5}{4} < \ab < \frac{9}{7}$ and $\ab \notin \{\frac{14}{11}, \frac{44}{35}\}$ and $\gcd(b, 2) = 1$.
Then $\rho(\ab) = a$.
\end{conjecture}

Two potential counterexamples to Conjecture~\ref{3 a} are $\frac{62}{49}$ and $\frac{79}{63}$, for which we do not have conjectural values of $k$.
On the other hand, $\frac{41}{32}$ is not covered by the conjecture because its denominator is even, but nonetheless the conclusion $\rho(\frac{41}{32}) = 41$ seems to hold, and the structure of $\word_{41/32}$ is related to the morphism in Theorem~\ref{3 a power-free}.
Some of the following conjectures have rationals with this same property.

\begin{conjecture}\label{14 8a-4b}
Let $a, b$ be relatively prime positive integers such that $\frac{9}{7} < \ab < \frac{4}{3}$ and $b \equiv 0 \mod 3$ and $\gcd(b, 8) = 1$.
Then $\rho(\ab) = 8a-4b$.
\end{conjecture}

\begin{conjecture}\label{29 14a-9b}
Let $a, b$ be relatively prime positive integers such that $\frac{8}{7} < \ab < \frac{23}{20}$ and $\gcd(b, 14) = 1$.
Then $\rho(\ab) = 14a-9b$.
\end{conjecture}

\begin{conjecture}\label{30 10a-5b}
Let $a, b$ be relatively prime positive integers such that $\frac{15}{13} < \ab < \frac{22}{19}$ and $\gcd(b, 10) = 1$.
Then $\rho(\ab) = 10a-5b$.
\end{conjecture}

\begin{conjecture}\label{38 12a-7b}
Let $a, b$ be relatively prime positive integers such that $\frac{29}{26} < \ab < \frac{19}{17}$ and $\gcd(b, 12) = 1$.
Then $\rho(\ab) = 12a-7b$.
\end{conjecture}

\begin{conjecture}\label{46 14a-9b I}
Let $a, b$ be relatively prime positive integers such that $\frac{12}{11} < \ab < \frac{23}{21}$ and $\ab \neq \frac{95}{87}$ and $\gcd(b, 14) = 1$.
Then $\rho(\ab) = 14a-9b$.
\end{conjecture}

\begin{conjecture}\label{102 38a-15b}
Let $a, b$ be relatively prime positive integers such that $\frac{11}{9} < \ab < \frac{16}{13}$ and $\gcd(b, 38) = 1$.
Then $\rho(\ab) = 38a-15b$.
\end{conjecture}

\begin{conjecture}\label{279 67a-30b}
Let $a, b$ be relatively prime positive integers such that $\frac{10}{9} < \ab < \frac{29}{26}$ and $\ab \notin \{\frac{39}{35}, \frac{49}{44}, \frac{69}{62}, \frac{89}{80}\}$ and $\gcd(b, 67) = 1$.
Then $\rho(\ab) = 67a-30b$.
\end{conjecture}

Our method for identifying families of related rationals is restrictive in several ways.
For example, we did not look for relationships among words with different numbers of nonzero letters.

Additionally, we do not know how to prove $\ab$-power-freeness for families such as the following two families (with $10$ nonzero letters each), where there is an extra congruence condition on the denominator.

\begin{conjecture}\label{10 5a-2b}
Let $a, b$ be relatively prime positive integers such that $\frac{5}{4} < \ab < \frac{9}{7}$ and $b \equiv 2 \mod 4$ and $\gcd(b, 5) = 1$.
Then $\rho(\ab) = 5a-2b$.
\end{conjecture}

\begin{conjecture}\label{10 4a-2b}
Let $a, b$ be relatively prime positive integers such that $\frac{13}{10} \leq \ab < \frac{4}{3}$ and $b \equiv 2 \mod 4$ and $\gcd(b, 3) = 1$.
Then $\rho(\ab) = 4a-2b$.
\end{conjecture}

For the rationals in Conjecture~\ref{10 4a-2b}, the word $\word_{a/b}$ contains $6$ columns that are not eventually constant but are eventually periodic with repetition period $01$, reminiscent of $\word_{3/2}$.

%%%%%%%%%%%%%%%%%%%%%%%%%%%%%%%%%%%%%%%%%%
\section{Sporadic words $\word_{a/b}$}\label{Sporadic words}
%%%%%%%%%%%%%%%%%%%%%%%%%%%%%%%%%%%%%%%%%%

Of the $520$ rational numbers $\ab$ for which we identified conjectural structure in $\word_{a/b}$, there are $277$ that fall under Theorem~\ref{2 2a-b} or one of the theorems or conjectures in Section~\ref{Families of words}.
In this section we establish the structure of $\word_{a/b}$ for some of the remaining rationals.
In particular, we prove Theorems~\ref{8/5}--\ref{6/5}.
For proving $\ab$-power-freeness, much of the algorithm is the same as for symbolic $\ab$.
There are some differences, however.
The most obvious difference is that we need not work with run-length encodings of words; we can manipulate $\varphi(n)$ as a word on the alphabet $\Z_{\geq 0} \cup \{n+d\}$.
Computationally this is much faster, especially considering the scale involved:
Whereas the largest symbolic morphism we identified was the morphism in Theorem~\ref{279 67a-30b power-free} with $279$ nonzero letters, the $50847$-uniform morphism for $\word_{7/4}$ has $11099$ nonzero letters.

Another difference is how we find an integer $\ell$ such that $\varphi$ locates words of length~$\ell$.
Rather than use the method of Section~\ref{determining a locating length}, we compute the minimum possible $\ell$ as follows.
Begin with $\ell = 0$, and maintain a set of sets of the positions of length-$\ell$ factors of $\varphi(n)$ that are not unequal (that is, each pair of these factors is equal for some value of $n$).
Initially, this set is $\{\{0, 1, \dots, |\varphi(n)| - 1\}\}$, since all length-$0$ factors of $\varphi(n)$ are equal.
We treat $\varphi(n)$ as a cyclic word so we visit all factors of $\varphi(n) \varphi(n) \cdots$.
Then increase $\ell$ by $1$, and update the sets of positions for the new length by extracting a single letter from $\varphi(n)$ for each position; in this way we avoid holding many large words in memory.
During this update, a set of positions breaks into multiple sets if it contains a pair of positions corresponding to unequal words of length $\ell$ that were not unequal for $\ell - 1$.
Delete any sets containing a single position, since the corresponding factor of length $\ell$ or greater is uniquely located modulo $|\varphi(n)|$.
When the set of position sets becomes empty, then $\varphi$ locates words of length $\ell$.

We use a variant of Proposition~\ref{big m} in which $\mmax \colonequal \lceil \frac{\ell}{a-b} \rceil - 1$, since $\frac{\ell}{a-b}$ is an explicit rational number.
We verify the conclusion of Lemma~\ref{short words are a/b-power-free} directly by checking that $\lceil \frac{\mmax a - 1}{k} \rceil + 1 \leq a - 1$.

The structure of words $\word_{a/b}$ with no transient can now be established automatically.
We condense results for $24$ words (including Theorems~\ref{8/5} and \ref{7/4}) into the following theorem.
In particular, this establishes the value of $\rho(\ab)$ for these $24$ rationals.

\begin{theorem}\label{sporadic rationals with no transient}
For each $\ab$ in Table~\ref{sporadic rationals}, there is a $k$-uniform morphism $\varphi(n) = u \, (n + d)$ such that $\word_{a/b} = \varphi^\infty(0)$.
Moreover, $\varphi$ locates words of length $\ell$.
\end{theorem}

\begin{table}
\[
\def\arraystretch{1.3}
\begin{array}{lrrrclrrr}
	\ab & d & k & \ell & \hspace{.5cm} & \, \ab & d & k & \ell \\ \cline{1-4} \cline{6-9}
 \frac{7}{4}\approx 1.75 & 2 & 50847 & 12940 & & \frac{37}{26}\approx 1.42308 & 1 & 2359 & 1680 \\
 \frac{8}{5}\approx 1.6 & 2 & 733 & 301 & & \frac{37}{28}\approx 1.32143 & 1 & 5349 & 3861 \\
 \frac{13}{9}\approx 1.44444 & 1 & 45430 & 11400 & & \frac{41}{28}\approx 1.46429 & 1 & 2103 & 999 \\
 \frac{17}{10}\approx 1.7 & 2 & 55657 & 37104 & & \frac{49}{34}\approx 1.44118 & 1 & 4171 & 3008 \\
 \frac{15}{11}\approx 1.36364 & 1 & 6168 & 711 & & \frac{55}{38}\approx 1.44737 & 1 & 5269 & 3816 \\
 \frac{16}{13}\approx 1.23077 & 1 & 12945 & 1321 & & \frac{53}{40}\approx 1.325 & 1 & 9933 & 4149 \\
 \frac{18}{13}\approx 1.38462 & 1 & 4188 & 2094 & & \frac{59}{42}\approx 1.40476 & 1 & 5861 & 4332 \\
 \frac{19}{13}\approx 1.46154 & 1 & 7698 & 946 & & \frac{65}{46}\approx 1.41304 & 1 & 7151 & 5292 \\
 \frac{21}{16}\approx 1.3125 & 2 & 25441 & 5606 & & \frac{67}{46}\approx 1.45652 & 1 & 7849 & 5720 \\
 \frac{25}{17}\approx 1.47059 & 1 & 11705 & 3268 & & \frac{71}{50}\approx 1.42 & 1 & 8569 & 6348 \\
 \frac{31}{22}\approx 1.40909 & 1 & 1645 & 1160 & & \frac{73}{50}\approx 1.46 & 1 & 9331 & 6816 \\
 \frac{33}{23}\approx 1.43478 & 1 & 24995 & 3576 & & \frac{77}{54}\approx 1.42593 & 1 & 10115 & 7500 \\
\end{array}
\]
	\caption{Rational numbers $\ab$ for which the structure of $\word_{a/b}$ is established by Theorem~\ref{sporadic rationals with no transient}.}
	\label{sporadic rationals}
\end{table}

Some of the words $u$ in Theorem~\ref{sporadic rationals with no transient} are themselves highly structured.
For example, for $\ab = \frac{71}{50}$ we have
\begin{align*}
	\varphi(n) = {}
	& \left(0^{70} \, 1 \, 0^{20} \, 1 \, 0^{20} \, 1 \, 0^{41} \, 1 \, 0^{28} \, 1\right) \\
	& \cdot \left(0^{41} \, 1 \, 0^{28} \, 1 \, 0^{12} \, 1 \, 0^{28} \, 1 \, 0^{41} \, 1 \, 0^{28} \, 1\right)^{34} \left(0^{41} \, 1 \, 0^{28} \, 1 \, 0^{12} \, 1 \, 0^{28} \, 1 \, 0^{41} \, 1 \, 0^{28} \, 2\right) \\
	& \cdot \left(0^{41} \, 1 \, 0^{28} \, 1 \, 0^{12} \, 1 \, 0^{28} \, 1 \, 0^{41} \, 1 \, 0^{28} \, 1\right)^{10} \left(0^{41} \, 1 \, 0^{28} \, 1 \, 0^{12} \, 1 \, 0^{7} \, 1 \, 0^{12} \, (n + 1)\right).
\end{align*}
Perhaps $\word_{71/50}$ can be generalized to an infinite family using this structure.

Words $\word_{a/b}$ with a transient can be handled as in Section~\ref{Families of words}.
We now prove Theorem~\ref{6/5} on the structure of $\word_{6/5}$.

\begin{theorem6/5}
There exist words $u, v$ of lengths $|u| = 1001 - 1$ and $|v| = 29949$ such that $\word_{6/5} = \tau(\varphi^\infty(0'))$, where
\[
	\varphi(n) =
	\begin{cases}
		v \, \varphi(0)		& \text{if $n = 0'$} \\
		u \, (n + 3)			& \text{if $n \in \Z_{\geq 0}$}.
	\end{cases}
\]
\end{theorem6/5}

\begin{proof}
Let $k = 1001$.
Compute the prefix of $v' u$ of $\word_{6/5}$ where $|v'| = 29949$ and $|u| = k - 1$.
Let $v = 0' \, 0^4 \, 1^5 \, 0^1 \, 2^1 \cdots 1^1 \, 2^1$ be the word obtained by changing the first letter of $v'$ to $0'$.

First we show that $\tau(\varphi^\infty(0'))$ is $\frac{6}{5}$-power-free.
We compute that ${\varphi |}_{\Z_{\geq 0}}$ locates words of length $315$ and is $\frac{6}{5}$-power-free.
It follows that if $\tau(\varphi^\infty(0'))$ contains a $\frac{6}{5}$-power then it contains a $\frac{6}{5}$-power beginning at some position $i \leq |v| - 1$.

We seek a length $\ell$ such that the factor $x$ of length $\ell$ beginning at position $i$ in $\tau(\varphi^\infty(0'))$, for each $i$ in the interval $0 \leq i \leq |v| - 1$, only occurs once in $\tau(\varphi^\infty(0'))$.
We can compute $\ell$ just as we computed the length $315$, but by beginning with the set $\{\{0, 1, \dots, |\tau(v) \varphi(n)| - 1\}\}$ instead of $\{\{0, 1, \dots, |\varphi(n)| - 1\}\}$; integers in the range $0$ to $|\tau(v)| - 1$ represent positions in $\tau(v)$, while the remaining integers represent general positions (modulo $k$) in the rest of the infinite word $\tau(v) \varphi(n) \varphi(n) \varphi(n) \cdots$.
The minimal such length is $\ell = 18215$.
Therefore if $\tau(\varphi^\infty(0'))$ contains a $\frac{6}{5}$-power factor $(xy)^{6/5} = xyx$ then $|x| < 18215 = 18215 (a-b)$.
Let $\mmax \colonequal 18214$.

It remains to show that $\tau(\varphi^\infty(0'))$ contains no $\frac{6}{5}$-power $(xy)^{6/5} = xyx$ with $|x| \leq \mmax (a-b)$ beginning at a position $i \leq |v| - 1$.
For each $m$ in the range $1 \leq m \leq \mmax$, this could be accomplished by sliding a window of length $m a$ through $\tau(\varphi^\infty(0'))$ from position $0$ to position $|v| - 1$ and verifying inequality of factors.
But since we already computed the length-$50000$ prefix of $\word_{6/5}$ in order to guess its structure, and $(|v| - 1) + \mmax a - 1 = 48161 \leq 50000$, we simply check that this prefix agrees with $\tau(\varphi^\infty(0'))$ and conclude that there are no $\frac{6}{5}$-powers in that range.

Now we show that decrementing any nonzero letter of $\tau(\varphi^\infty(0'))$ introduces a $\frac{6}{5}$-power.
Decrementing any nonzero letter of $\tau(v)$ introduces a $\frac{6}{5}$-power, since we defined $v' = \tau(v)$ to be a prefix of $\word_{6/5}$.
Every other nonzero letter is a factor of $\varphi(n)$ for some integer $n$.
Decrementing any but four of the nonzero letters in $u \, 3$ to $0$ introduces a $\frac{6}{5}$-power of length $m a$ for some $m \leq 25$.
Decrementing any $2$ or $3$ in $u \, 3$ to $1$ introduces a $\frac{6}{5}$-power of length $m a$ for some $m \leq 20$.
Decrementing any of the first five of six $3$s in $u \, 3$ to $2$ introduces a $\frac{6}{5}$-power of length $m a$ for some $m \leq 4$.
For the remaining four $0$s and one $2$, we consider two cases, depending on whether $u$ is immediately preceded by $\tau(v)$ or $\varphi(n)$.
In both cases, decrementing the last letter of $u \, 3$ to $2$ introduces a $\frac{6}{5}$-power of length $233 a$, and decrementing one the four remaining nonzero letters to $0$ introduces a $\frac{6}{5}$-power of length $m a$ for some $m \leq 82$.
As before, decrementing $n + 3$ to $c \geq 3$ corresponds to decrementing an earlier letter $n$ to $c - 3$ and therefore, inductively, introduces a $\frac{6}{5}$-power.
\end{proof}

We have automated the steps of the preceding proof as well.

\begin{theorem4/3}
There exist words $u, v$ of lengths $|u| = 56 - 1$ and $|v| = 18$ such that $\word_{4/3} = \tau(\varphi^\infty(0'))$, where
\[
	\varphi(n) =
	\begin{cases}
		v \, \varphi(0)		& \text{if $n = 0'$} \\
		u \, 1				& \text{if $n = 0$} \\
		u \, (n + 2)			& \text{if $n \in \Z_{\geq 1}$}.
	\end{cases}
\]
\end{theorem4/3}

\begin{proof}
Let
\begin{align*}
	u &= 1202110001120202010101020211100012120201010102020211000, \\
	v &= 0'00111020201110002.
\end{align*}
One checks that $\tau(v)$ is a prefix of $\word_{4/3}$.

Showing that $\tau(\varphi^\infty(0'))$ is $\frac{4}{3}$-power-free with our code requires an extra step, since the morphism effectively uses two values of $d$.
Define $\varphi_1(n) = u \, (n + 1)$ for all $n \geq 0$.
We use the morphism $\varphi_1$ and specify as an assumption that $n = 0$ or $n \geq 2$.
The code then verifies that ${\varphi |}_{\Z_{\geq 0}}$ locates words of length $14$ and is $\frac{4}{3}$-power-free.
It follows that if $\tau(\varphi^\infty(0'))$ contains a $\frac{4}{3}$-power then it contains a $\frac{4}{3}$-power beginning at some position $i \leq |v| - 1$.

We proceed as in the proof of Theorem~\ref{6/5}.
The factor of length $\ell = 7$ beginning at position $i$ in $\tau(\varphi^\infty(0'))$, for each $i$ in the interval $0 \leq i \leq |v| - 1$, only occurs once in $\tau(\varphi^\infty(0'))$.
Therefore if $\tau(\varphi^\infty(0'))$ contains a $\frac{4}{3}$-power factor $(xy)^{4/3} = xyx$ then $|x| < 7 = 7 (a-b)$.
Sliding a window of length $m a$ from position $0$ to position $|v| - 1$ for each $m$ in the range $1 \leq m \leq 6$ shows that $\tau(\varphi^\infty(0'))$ contains no $\frac{4}{3}$-power $(xy)^{4/3} = xyx$ with $|x| \leq 6 (a-b)$.

Now we show that decrementing any nonzero letter of $\tau(\varphi^\infty(0'))$ introduces a $\frac{4}{3}$-power.
Decrementing any nonzero letter of $\tau(v)$ introduces a $\frac{4}{3}$-power, since $\tau(v)$ is a prefix of $\word_{4/3}$.
One checks as before that decrementing a nonzero letter of $u$ introduces a $\frac{4}{3}$-power, regardless of whether $u$ is preceded by $\tau(v)$ or $\varphi(n)$.
The last letter of $\varphi(n)$ is either $1$ or $n + 2 \geq 3$.
Decrementing this letter to $0$ introduces the $\frac{4}{3}$-power $0^4$.
If $n = 0$ then $0$ is the only possibility.
If $n \geq 1$ then decrementing $n + 2$ to $c = 1$ corresponds to decrementing an earlier letter $n$ to $0$ and, inductively, introduces a $\frac{4}{3}$-power.
If $n \geq 1$ then decrementing $n + 2$ to $c \geq 2$ corresponds to decrementing an earlier letter $n$ to $c - 2$ and, inductively, introduces a $\frac{4}{3}$-power.
\end{proof}

The morphisms we have encountered up to this point have all been defined for $n \in \Z_{\geq 0}$ as $\varphi(n) = u \, (n + d)$ for some $d \geq 1$ (or two positive values of $d$ in the case of Theorem~\ref{4/3}).
However, $d = 0$ also occurs, and this is significant because it yields a word on a finite alphabet.

\begin{theorem}\label{27/23}
There exist words $u, v \in \{0, 1, 2\}^*$ of lengths $|u| = 353 - 1$ and $|v| = 75019$ such that $\word_{27/23} = \tau(\varphi^\infty(0'))$, where
\[
	\varphi(n) =
	\begin{cases}
		v \, \varphi(0)	& \text{if $n = 0'$} \\
		u \, n			& \text{if $n \in \Z_{\geq 0}$}.
	\end{cases}
\]
In particular, $\word_{27/23}$ is a $353$-automatic sequence on the alphabet $\{0, 1, 2\}$.
\end{theorem}

\begin{proof}
Lexicographic-leastness is proved as in Theorem~\ref{6/5}.
The proof of $\frac{27}{23}$-power-freeness is similar to Theorem~\ref{6/5}.
The morphism ${\varphi |}_{\Z_{\geq 0}}$ locates words of length $52$ and is $\frac{27}{23}$-power-free.
However, there is a new subtlety in finding a length $\ell$ such that the factor of length $\ell$ beginning at position $i$ in $\tau(\varphi^\infty(0'))$, for each $i$ in the interval $0 \leq i \leq |v| - 1$, only occurs once in $\tau(\varphi^\infty(0'))$.
This subtlety arises because $v$ and $\varphi(n)$ have a common suffix for some $n$, namely the word $1$ when $n = 1$.
Consequently, there is no uniquely-occurring factor beginning at position $|v| - 1$ in $\tau(v) \varphi(n) \varphi(n) \varphi(n) \cdots$.
Instead, we use the fact that $0^{18}$ occurs at positions $0, 1, \dots, 8$ in $\tau(\varphi^\infty(0'))$ but nowhere else, so the factor $1 \, \varphi(0)^{18}$ occurs at position $|v| - 1$ but does not re-occur later.
Then we find a length that induces a uniquely-occurring factor at all positions $i$ in the interval $0 \leq i \leq |v| - 2$  using the symbolic $\tau(v) \varphi(n)$ as before.
This length is $29588$, which we do not need to increase on account of position $|v| - 1$ since $29588 > 1 + (18 + 1) \cdot 353$.
So, to finish, one slides a window of length $m a$ from position $0$ to position $|v| - 2$ for each $1 \leq m \leq 7396$.
\end{proof}

Guay-Paquet and Shallit~\cite{Guay-Paquet--Shallit} asked whether $\word_{\geq 5/2}$ is a word on a finite alphabet.
While the answer is not known, Theorem~\ref{27/23} shows that for some rationals the word $\word_{a/b}$ is a word on a finite alphabet.
Theorem~\ref{27/23} also shows that ${\varphi |}_{\{0, 1, 2\}}$ is a $\frac{27}{23}$-power-free morphism on the alphabet $\{0, 1, 2\}$.

We identified $20$ additional rational numbers such that $\word_{a/b}$ seems to be generated by a morphism with $d = 0$ and is therefore conjecturally a word on a finite alphabet.
For the $9$ rationals $\frac{117}{97}$, $\frac{64}{53}$, $\frac{107}{87}$, $\frac{85}{69}$, $\frac{90}{73}$, $\frac{127}{103}$, $\frac{95}{77}$, $\frac{100}{81}$, and $\frac{68}{55}$, the correct value of $k$ seems to be $38a-15b$, just as in Conjecture~\ref{102 38a-15b}.
Moreover, the number of nonzero, eventually constant columns for these words is $102$, which is the number of nonzero letters in the corresponding morphism in Theorem~\ref{102 38a-15b power-free}.
This suggests a further connection between $\ab$-power-free morphisms on finite vs.\ infinite alphabets.

%%%%%%%%%%%%%%%%%%%%%%%%%%%%%%%%%%%%%%%%%%
\section*{Acknowledgments}
%%%%%%%%%%%%%%%%%%%%%%%%%%%%%%%%%%%%%%%%%%

We thank the referee for a careful reading and good suggestions.

\end{document}